\documentclass{imsart}
\usepackage{url}
\usepackage{enumerate}
\RequirePackage{amsthm,amsmath,amsfonts,amssymb}
\RequirePackage[numbers,sort&compress]{natbib}
\usepackage[colorlinks=true, allcolors=blue]{hyperref}
\usepackage{graphicx}
\usepackage[position=top]{subfig}
\usepackage{floatrow}
\usepackage{xr-hyper}
\makeatletter
\newcommand*{\addFileDependency}[1]{
\typeout{(#1)}
\@addtofilelist{#1}
\IfFileExists{#1}{}{\typeout{No file #1.}}
}\makeatother

\startlocaldefs

\newtheorem {Proposition}{Proposition}[section]
\newtheorem {Lemma}[Proposition] {Lemma}
\newtheorem {Theorem}[Proposition]{Theorem}
\newtheorem {Corollary}[Proposition]{Corollary}
\newtheorem {Assumption}[Proposition]{Assumption}
\newtheorem {Remark}[Proposition]{Remark}

\newtheorem {Definition}{Definition}[section]
\numberwithin{equation}{section}

\def\N{\mathbb{N}}
\def\Z{\mathbf{Z}}

\def\R{\mathbb{R}}
 
\def\E{\mathbb{E}}
\def\P{\mathbb{P}}
\def\z{\mathbf{z}} 
\def\x{\mathbf{x}}
\def\y{\mathbf{y}}
\def\u{\mathbf{u}}

\def\X{\mathbf{X}}
\def\A{\mathbf{A}}
\def\Y{\mathbf{Y}}

\DeclareMathOperator*{\argmin}{arg\,min}

\def\P{\mathbb P}
\def\E{\mathbb E}

\def\N{\mathbb N}

\def\x{\mathbf x}
\def\z{\mathbf z}

\def\y{\mathbf y}

\def\P{\mathcal{P}}

\def\bSigma{\boldsymbol{\Sigma}}
\def\bA{\boldsymbol{A}}
\def\Ptac{\P^{\rm {\,}a.c}_2}

\def\Var{\operatorname{Var}}

\newcommand\independent{\protect\mathpalette{\protect\independenT}{\perp}}
\def\independenT#1#2{\mathrel{\rlap{$#1#2$}\mkern2mu{#1#2}}}
\def\independenT#1#2{\mathrel{\rlap{$#1#2$}\mkern2mu{#1#2}}}
\endlocaldefs

\begin{document}

\begin{frontmatter}
\title{Wasserstein Spatial Depth}
\runtitle{Wasserstein Spatial Depth}
\runauthor{Bachoc, Gonz\'alez-Sanz, Loubes, and Yao}

\begin{aug}

\author[A]{\fnms{François}~\snm{Bachoc}\ead[label=e1]{francois.bachoc@univ-lille.fr}},
\author[B]{\fnms{Alberto}~\snm{González-Sanz}\ead[label=e2]{ag4855@columbia.edu}},
\author[C]{\fnms{Jean-Michel}~\snm{Loubes}\ead[label=e3]{jean-michel.a.loubes@inria.fr}},
\author[D]{\fnms{Yisha}~\snm{Yao}\ead[label=e4]{yy3381@columbia.edu}}

\address[A]{Universit\'e de Lille, Institut universitaire de France (IUF),
	France, \printead{e1}}

\address[B]{Department of Statistics,
	Columbia University, New York, USA, \printead{e2}}

\address[C]{INRIA,  Universit\'e de Toulouse, France, \printead{e3}}

\address[D]{Department of Statistics,
	Columbia University, New York, USA, \printead{e4}}
\end{aug}

\begin{abstract}
Modeling observations as random distributions embedded within Wasserstein spaces is becoming increasingly popular across scientific fields, as it captures the variability and geometric structure of the data more effectively. However, the distinct geometry and unique properties of Wasserstein spaces pose challenges to the application of conventional statistical tools, which are primarily designed for Euclidean spaces. Consequently, adapting and developing new methodologies for analysis within Wasserstein spaces has become essential.
The space of distributions on $\mathbb{R}^d$ with $d>1$ is not linear, and ``mimic'' the geometry of a Riemannian manifold.  In this paper, we extend the concept of statistical depth to distribution-valued data, introducing the notion of {\it Wasserstein spatial depth}. This new measure provides a way to rank and order distributions, enabling the development of order-based clustering techniques and inferential tools.  We show that Wasserstein spatial depth (WSD) preserves critical properties of conventional statistical depths, notably, ranging within $[0,1]$, transformation and geodesic invariance, vanishing at infinity, reaching a maximum at the geometric median, and continuity. 
Regarding robustness, we characterize the breakdown points of the empirical depth regions and the influence function of the WSD.
Additionally, the population WSD has a straightforward plug-in estimator based on sampled empirical distributions. We establish the estimator's consistency and asymptotic normality.
We also provide a two-sample test for populations of distributions based on the WSD.
Finally, extensive simulations and a real-data application showcase the practical efficacy of the WSD.
\end{abstract}



\end{frontmatter}

\section{Introduction}

Contemporary data collected in various disciplines is complex and multifaceted. Traditional statistical tools, which model data objects as samples from a Euclidean space or vector space, are inadequate to capture the variation and geometry of the data objects.
Random objects lying in general metric spaces, including spaces of functions \citep{wang2016functional}, Wasserstein spaces \citep{muzellec2018generalizing}, and hyperbolic spaces \citep{zhou2021hyperbolic}, are gaining increasing favor in the scientific community. 
For instances, longitudinal images are treated as functions \cite{wang2016functional}; texts and media are modeled as distributions in modern AI training models \cite{ghorbani2020distributional, chan2022data}; certain trees and graphs are embedded into hyperbolic spaces \citep{chami2020trees}. 
It is widely recognized that statistical efficiency can be gained by utilizing special properties of the above metric spaces \citep{bigot2020statistical}.

In this paper, we focus on modeling distribution-valued data objects within Wasserstein spaces. 
There are several advantages to model certain data objects as distributions or probability measures. 
Firstly, it captures the hierarchical variations in the data by simulating a two-stage data-generating process: initially sampling multiple distributions from a Wasserstein space, followed by drawing data points from each sampled distribution.
Secondly, it captures variations of the data along geodesics of the distribution space that are not straight lines as in the Euclidean setting and thus are closer to the observations. Thirdly, it often provides a low-dimensional embedding that effectively represents high-dimensional data, enabling better statistical inference without the curse of the dimension.
Since the Wasserstein space has different structure and property from the Euclidean space, conventional analytic tools cannot apply to distribution-valued data objects. Therefore, new methods specifically designed for analyzing such data are essential.

There have been some efforts in this line of research, including but not limited to histogram regression \citep{bonneel2016wasserstein}, Wasserstein regression \cite{bachoc2023gaussian,chen2023wasserstein, chen2022uniform}, geodesic PCA \citep{AIHP706}, template estimation \citep{boissard2015distribution}, and Wasserstein clustering \citep{zhuang2022wasserstein,del2020optimalflow}.
Despite the above developments, there is limited effort in agnostic exploratory analysis for data objects in Wasserstein spaces \citep{Geenens2023, LopezTukeyMetric, virta2023spatial, dubey2024metric}. Still, exploratory analysis and descriptional statistics are critical to overview the properties of the data distribution before modeling. In particular, a notion of ``ordering'' for distribution-valued objects in Wasserstein spaces will be of fundamental utility.  Besides exploratory analysis, it will also facilitate nonparametric methods for distribution-valued data.

Quantiles, ranks, and signs are pivotal tools of semiparametric and nonparametric statistics. 
Due to the lack of canonical ordering in multi-dimensional Euclidean space, quantile or rank based tools have been limited to one-dimensional data before the creation of statistical depths. The notion of statistical depths fills this gap, extending the notion of order to higher dimension.
Given a distribution $P$ on $\R^d$, the depth of any data point $\x \in\R^d$ is a non-negative value that measures the ``centrality'' of $\x$ with respect to $P$. 
A larger value of depth indicates the data point is more central within the distribution, while data points with small depths are considered outliers or less typical within the distribution, 
worthy of investigation. Several different types of depths have been proposed, including Tukey depth \citep{tukey1975mathematics,oja1983descriptive}, simplicial depth \citep{liu1990notion}, spatial depth \citep{chaudhuri1996geometric, vardi2000multivariate}, Monge-Kantorovich depth \cite{chernoetal17,HallinEtAl.2020.AoS} and lens depth \citep{Lens}. Via endowing multivariate data points with ``center-outward'' orderings, depths allow extension of order statistics, robust inference \citep{liu1993quality, zuo2006limiting} and classification to multivariate data \citep{zuo2000general, mosler2022choosing}. 

Statistical depth theory is one of the main research areas of functional data analysis (FDA). Most of the Euclidean depth functions extend naturally to Hilbert-space-valued data, see \citep{lopez2009concept, Reyes2009topological,NAGY2017373,CUEVAS2009753}. For instance, this is the case for the h-depth \cite{Cuevas.Feb.Frai.2007.CompStat}, the Tukey depth  \cite{Dutta.et.a.2011.BJ} and   its random version \cite{Cuesta-Albertos.Nieto-Reyes.2008.Compt-stat}, the spatial depth \cite{Chakraborty2014spatial}, the integrated depth \cite{CUEVAS2009753}  and the Monge-Kantorovich depth \cite{gonzalezsanz2023monotone}. For Banach spaces, some examples are the integrated depth \cite{CUEVAS2009753}, the band depth  and its modified version \cite{lopez2009concept},  the half-region depth  and its modified version \cite{LOPEZPINTADO20111679}, the $L^\infty$ depth \cite{long.Huang.2015.Preprint} and  the infimal depth \cite{mosler2013depth}.

While it may be tempting to embed the Wasserstein space into a function space, for instance a reproducing kernel Hilbert space (RKHS) \cite{sriperumbudur2010hilbert}, and apply existing functional depth measures,  this approach neglects the intricate geodesic structure of the Wasserstein space.  There is no linear representation of the Wasserstein distance between distributions on $\mathbb{R}^d$ with $d>1$ \citep{bertrand2012geometric}. Existing depths do not generalize well to nonlinear spaces. 
Besides the nonlinearity, the Wasserstein distance is computationally expensive even for empirical measures \citep{peyre2019computational}, which essentially rules out practical implementation of Tukey depth \cite{tukey1975mathematics} and Monge-Kantorovich depth \cite{chernoetal17,HallinEtAl.2020.AoS}. The computational complexity of these two depths grows exponentially with the sample size.

In conclusion, conventional depth measures cannot be directly extended to Wasserstein spaces due to the unique properties and structure discussed above. This requires the development of a new notion of depth tailored specifically for Wasserstein spaces.


\subsection{Contributions}
%
In this paper we develop a new notion of depth to order or rank distributions. It is inspired by  {\it spatial depth} \cite{vardi2000multivariate}, one of the simplest and most widely used notions of statistical depths.
Recall that the spatial depth of a point $\x\in \R^d$ with respect to a probability measure $ P$ over $\R^d$ is defined as 
\begin{equation}\label{conventional_spatial_depth}
	\rm SD(\x; P)=1-\left\|\E\left[\frac{\X-\x}{\|\X-\x\|}\right]\right\|, \quad \X\sim P.
\end{equation}
The spatial depth has been generalized to Hilbert spaces by following exactly the same definition \cite{serfling2002depth,virta2023spatial}. However, the lack of linear structure of the Wasserstein space prevents a straightforward adaptation of the spatial depth. Nevertheless, the Wasserstein metric endows the space of probability measures with a structure of geodesic  metric space (see \cite{AmbrosioGradient}).
For absolutely continuous probability measures $Q$ and $P$, the constant speed geodesic joining $Q$ and $P$ is given by the curve of probability measures 
$$ [0,1]\ni \lambda\mapsto   ((1-\lambda ){\rm I}+\lambda\, T_{Q,P})_{\# } Q, $$
where $\#$ denotes the push-forward operator
and where $T_{Q,P}$ is the optimal transport map from $Q$ to $P$
(see Section~\ref{subsec:WassersteinGeom} for the definitions). The definition of spatial depth for manifold-valued data 
motivates us to define the depth of a probability measure $Q\in \mathcal{P}^{a.c}_2(\R^d)$ (where $ \mathcal{P}^{a.c}_2(\R^d)$ is the set of absolutely continuous measures on $\R^d$ with finite second moments, see Section~\ref{subsec:WassersteinGeom})
with respect to a probability measure over the Wasserstein space  $ \mathbf{P}\in \mathcal{P}(\mathcal{P}_2(\R^d)) $
as
$$  {\rm WSD}(Q; {\bf P}):=1-  \left(\int{ \left\|\E_{P\sim \mathbf{P}}\left[  \frac{\x-T_{Q,P}(\x) }{\mathcal{W}_2(P,Q)}\right]\right\|^2} \mathrm{d}Q(\x)\right)^{\frac{1}{2}}. $$
Above, $\mathcal{W}_2(P,Q)$ is the Wasserstein distance between $P$ and $Q$, being formally defined in  Section~\ref{subsec:WassersteinGeom}.
We shall show that $ {\rm WSD}(Q; \mathbf{P})$ satisfies the same properties as its Euclidean counterpart, namely transformation and geodesic invariance, taking values in $[0,1]$, decreasing at infinity
and attaining maximum at the median. Moreover, we show that both $Q\mapsto  {\rm WSD}(Q; {\bf P})$ and ${\bf P}\mapsto  {\rm WSD}(Q; {\bf P})$ are continuous.  
We quantify statistical robustness by characterizing the influence function of ${\rm WSD}(Q; {\bf P})$ and the breaking points of the regions of its empirical counterpart ${\rm WSD}(Q; {\bf P}_n)$ (${\bf P}_n$ being an empirical measure).
Next, we propose a finite-sample estimator under the so-called two-stage and one-stage sampling models. In the two-stage sampling model, such an estimator  can be computed in polynomial time. Moreover, we show that in both models, the estimator is  consistent, meaning that it approximates the true population depth function as the sample size increases. 
We also prove asymptotic normality.
Then, we provide a permutation-based two-sample test for comparing two distributions ${\bf P}$ and ${\bf Q}$, based on the finite-sample estimator, for which we prove level and power guarantees.
Finally, we provide numerical simulations for real and synthetic datasets. In particular, we highlight that our suggested depth is more informative than depth methods designed for linear spaces and applied to mappings of distributions to these linear spaces.
We also demonstrate the merits of the two-sample test, and its complementarity to a test based on distance profiles in Wasserstein space \cite{dubey2024metric}. 

In conclusion, Wasserstein spatial depth (WSD) serves as a valid measure for ordering objects within Wasserstein spaces, adhering to the axiomatic properties of depth \citep{zuo2000general} and being computationally feasible. This concept facilitates the extension of depth-based analytic tools to Wasserstein spaces, paving the way for future research.

\subsection{Organization}
General notations are provided in Section~\ref{section:notation}.
The definition of WSD is given in Section~\ref{method} with illustrating examples in Section~\ref{section:examples}. In Section~\ref{section:properties}, we show that WSD shares the desirable properties of conventional statistical depths \citep{zuo2000general} and we establish the robustness properties. 
In Section~\ref{consistency}, we tackle consistent estimation with asymptotic normality.
In Section~\ref{section:comparison}, we compare WSD to several depths in general metric spaces \citep{Geenens2023, LopezTukeyMetric, virta2023spatial} adapted to Wasserstein spaces.  
We advocate WSD over the other depths in terms of computational feasibility and assumption flexibility, while possessing all desirable properties of a depth. 
In Section \ref{section:testing} we provide the two-sample test and its properties.
In section~\ref{section:simulation}, extensive numerical simulations are shown to demonstrate the empirical validity and merits of WSD. 
Finally, in Section~\ref{section:application}, we apply it to explore real-world data and make informative discoveries. 
We give a final discussion in Section \ref{section:future:directions}.
All the proofs are provided in the Appendix.

\section{Notation} 
\label{section:notation}

The space of Borel probability measures on a Polish space $(\mathcal{K},d)$ is denoted as 
$\mathcal{P}(\mathcal{K})$. 
For $P \in \mathcal{P}(\mathcal{K})$, its support is written ${\rm supp}( P )$.
The space of Borel finite (signed) measures is denoted as $\mathcal{M}(\mathcal{K})$ and the space of finite (signed)  measures with $0$ mass as $\mathcal{M}_0(\mathcal{K})$, meaning that 
$h\in \mathcal{M}_0(\mathcal{K})$ if $h\in \mathcal{M}(\mathcal{K})$ and $h(\mathcal{K})=0$. 
The integral of a measurable function $f:\mathcal{K} \to \R$ with respect to $P\in \mathcal{P}(\mathcal{K})$ is denoted as 
$$ \int f(\x) \mathrm{d} P(\x)= \int f \mathrm{d}P=P(f).$$
Set $P\in \mathcal{P}(\mathcal{K})$ and $f:\mathcal{K} \to \R$ be measurable. Then 
$$ \|f\|_{L^2(P)}:= \left(\int f^2 \mathrm{d}P \right)^{1/2}$$
denotes the $L^2(P)$-norm of $f$. 
The Hilbert space of measurable functions with finite $L^2(P)$-norm is denoted as $L^2(P)$
with inner product $\langle \cdot , \cdot \rangle_{L^2(P)}$.
We also extend the definition 
of the Hilbert space $L^2(P)$ and the associated notation 
to vector-valued functions, with for $f,g:\mathcal{K} \to \R^k$, 
$$
\langle
f,g
\rangle_{L^2(P)}:= \int \langle f, g \rangle \mathrm{d}P.
$$
We say that a sequence $\{\mu_n\}_{n\in \N}\subset \mathcal{P}(\mathcal{K})$ converges weakly to $\mu\in \mathcal{P}(\mathcal{K})$ if 
$$ \int f\, \mathrm{d}\mu_n \longrightarrow \int f\, \mathrm{d}\mu  $$
for every bounded and continuous function $f:\mathcal{K}\to \R$. In such a case we write $\mu_n\xrightarrow{\mathcal{P}(\mathcal{K})} \mu$ and also say that $\mu_n\rightarrow \mu$ in the weak sense of $\mathcal{P}(\mathcal{K})$.
For $Z_n \sim \mu_n$ and $Z \sim \mu$ we write similarly $Z_n\xrightarrow{\mathcal{P}(\mathcal{K})} Z$
and we may also write simply $Z_n\xrightarrow{w} Z$.
Such a convergence is metrizable by means of the so-called bounded Lipschitz metric \cite[p.\ 73]{vanderVaart1996}
\begin{multline*}
	d_{{\rm BL}}(\mu,\nu) 
	= \sup\bigg\{ \int f(x)\mathrm{d}(\mu-\nu)(x):   \\ |f(x)|\leq 1\ {\rm and}\ |f(x)-f(y) |\leq d(x,y), \ \forall \, x,y\in \mathcal{K}\bigg\}.  
\end{multline*}

\section{From Euclidean to Wasserstein  spatial depth}\label{method} 
In this section we define our notion of Wasserstein spatial depth. In Section~\ref{subsect:Euclidean} we recall the definition of Euclidean spatial depth and its main properties. In Section~\ref{subsec:Geodesic} we provide our interpretation of  spatial depth in terms of geodesics, which allows for its generalization to the Wasserstein space of measures (see Section~\ref{subsec:WassersteinGeom}). For readers interested in a more comprehensive understanding of the mathematical concepts discussed in Sections~\ref{subsec:Geodesic} and \ref{subsec:WassersteinGeom}, we recommend consulting the monograph \cite{AmbrosioGradient} for an in-depth exposition. 

\subsection{Euclidean spatial depth}\label{subsect:Euclidean}
In $\R^d$, for $d>1$, the spatial depth of a point $\x$ with respect to a random variable $\X\sim P$ is defined as 
$$   \rm SD(\x; P)=1-\left\|\E\left[\frac{\X-\x}{\|\X-\x\|}\right]\right\|.$$
Throughout the paper, we use the convention $\mathbf{0}/0 = \mathbf{0}$.
The spatial depth shares the following properties with the univariate canonical depth function \( 2\min(F(x), 1-F(x)) \). First, \(  \rm SD(\x; P) \) belongs to the interval \([0,1]\). Second, the geometric median, defined as 
\[
{\bf m}_{\X} \in \argmin_{{\bf m} \in \mathbb{R}^d} \E[\|\X - {\bf m}\|],
\]
satisfies \(  \rm SD({\bf m}_{\X}; P) = 1 \). Third, as \( \|\x\| \to \infty \), we have that 
\(
\rm SD(\x; P) \to 0.
\)
Finally, for an isometric  transformation \(T:\R^d\to \R^d\), it holds that 
\[
\rm SD(T(\x); T\# P) 
=  \rm SD(\x; P),
\]
where again the push-forward operator $\#$ is defined in Section \ref{subsec:WassersteinGeom}.

\subsection{Geodesic interpretations of the spatial depth}\label{subsec:Geodesic}
Let $(\mathcal{M}, d) $ be a metric space. A curve $\{\gamma_t^{\x\to \y}\}_{t\in [0,1]}$
valued in $\mathcal{M}$
is a (constant speed) geodesic joining $\x\in \mathcal{M}$ to $\y\in \mathcal{M}$ if   
$$ d(\gamma_t^{\x\to \y},\gamma_s^{\x\to \y})=(t-s) d(\x,\y), \quad \text{for all } 0\leq s\leq t\leq 1. $$
The space $(\mathcal{M}, d) $ is said to be geodesic if any two points are joined by at least one geodesic. The length of a curve $\{\gamma_t\}_{t\in [0,1]}$
with values in $\mathcal{M}$ (not necessarily a geodesic)
is defined as 
$ L(\gamma)= \int_{0}^1 | \gamma_t' | \mathrm{d} t $, where 
$ | \gamma_t' | = \lim_{s\to t}\frac{d(\gamma_t,\gamma_s)}{|t-s|}$.  Assume now that $\mathcal{M}\subset \R^d$ is a Riemannian manifold with metric tensor $\{g_{\x}\}_{\x\in \mathcal{M}}$. Then it holds that 
$$ L(\gamma) =\int_{0}^1 \sqrt{g_{\gamma_t}(\partial_t \gamma_t,\partial_t \gamma_t)} \mathrm{d}t, $$
where $ \{\partial_t \gamma_t\}_{t\in [0,1]}$ denotes the velocity (standard time derivative) of the curve $\{\gamma_t\}_{t\in [0,1]}$. 

In $\R^d$, a geodesic joining $\x$ and $\y$ is just the segment $\gamma_{t}^{\x\to \y}=(1-t)\x+ t \y $, $t\in [0,1]$. Therefore, the spatial depth of $\x$ can be seen as the spatial depth of the velocities at time $0$
$$   \rm SD(\x; P)=1-\left\|\E\left[\frac{\partial_t\vert_{t=0} \gamma_{t}^{\x\to \X}}{\|\partial_t\vert_{t=0} \gamma_{t}^{\x\to \X}\|}\right]\right\|.$$
This allows for the following Riemannian generalization 
of the spatial depth
$$   \rm SD^{general}(\x; P)=1-\sqrt{g_{\x}\left(\E\left[\frac{\partial_t\vert_{t=0} \gamma_{t}^{\x\to \X}}{\|\partial_t\vert_{t=0} \gamma_{t}^{\x\to \X}\|}\right], \E\left[\frac{\partial_t\vert_{t=0} \gamma_{t}^{\x\to \X}}{\|\partial_t\vert_{t=0} \gamma_{t}^{\x\to \X}\|}\right] \right)}.$$

\subsection{Geodesic spatial depth over the space of measures}\label{subsec:WassersteinGeom}
Let $\mathcal{P}_p(\R^d)$ be the space of Borel probability measures over $\R^d$ with finite $p$th order moment. The optimal transport cost between  two probability measures $P,Q\in \mathcal{P}_p(\R^d)$ is defined as 
\begin{equation}\label{Wassertesin}
	{\rm OT}_p(P,Q)=\inf_{\pi\in \Pi(P,Q) } \frac{1}{p} \int \|\x-\y\|^p \mathrm{d}\pi(\x,\y),
\end{equation}
where $\Pi(P,Q)\subset \mathcal{P}_p(\R^d\times \R^d)$ stands for the set of probability measures  with marginals $P$ and $Q$, i.e., $(\X,\Y)\sim \pi\in \Pi(P,Q)$  if $\X\sim P$ and $\Y\sim Q$. For $p\geq 1$, the mapping $(P,Q)\mapsto \mathcal{W}_p(P,Q)=\left({\rm OT}_p(P,Q)\right)^{\frac{1}{p}}$ defines a distance over the space $\mathcal{P}_p(\R^d)$ such that
$$ \mathcal{W}_p(P_n,P)\to 0 \quad  \iff \quad P_n\xrightarrow{\mathcal{P}(\R^d)}  P \quad {\rm and}\quad  \int \|\x\|^p \mathrm{d} P_n(\x) \to \int \|\x\|^p \mathrm{d} P(\x).  $$

We focus now on the case $p=2$. 
We define $\mathcal{P}^{a.c}_2(\R^d)$ as the subset of $\mathcal{P}_2(\R^d)$ composed of absolutely continuous measures.
If $P$ belongs to $\mathcal{P}^{a.c}_2(\R^d)$, there exists a unique minimizer $\pi_{P,Q}$ of \eqref{Wassertesin}, for $p=2$. Moreover, there exists a unique gradient of a convex function $ T_{P,Q}=\nabla \phi_{P,Q}  $ such that $\pi_{P,Q}=({\rm I}\times \nabla \phi_{P,Q})_{\#} P$.
The map $T_{P,Q}$ is called an optimal transport map.
Here, for a probability measure $\mu$ and a Borel mapping $T$,   $T_\# \mu$ denotes the push forward measure, which is the distribution of $T(\X)$, for $\X\sim \mu$.

In \cite{Otto2001}, the author demonstrated that \( \mathcal{W}_2 \) serves as the natural metric for \( \mathcal{P}_2(\mathbb{R}^d) \), aimed at describing the long-term behavior of solutions to the porous medium equation. This metric also imparts a geodesic metric space structure to \( \mathcal{P}_2(\mathbb{R}^d) \).

It is natural in the following sense. If $\{\X_t\}_{t\in [0,1]}$ is a curve of random vectors with $\partial_t \X_t={\bf v}_t(\X_0)$,  then its associated curve of distributions $\{P_t\}_{t\in [0,1]}$ satisfies the so-called transport/continuity equation 
\begin{equation}
	\label{ContinuityEq}
	\partial_t P_t+{\rm div}( {\bf v}_t P_t )=0
\end{equation}
in an appropriate weak sense. The continuity equation is commonly used in fluid mechanics, where  ${\bf v}_t$ represents the flow velocity vector field.    However, given the curve $\{\X_t\}_{t\in [0,1]}$ there could exist several curves of velocity fields $\{{\bf v}_t\}_{t\in [0,1]}$ solving \eqref{ContinuityEq}, i.e., generating the same flow. Among all  of them, there exists only one belonging to 
\begin{equation}
	\label{eq:WassersteinVar}
	\argmin\left\{\int_0^1 \|{\bf v}_t\|_{L^2(P_t)}^2 \mathrm{d} t:  \quad \partial_t P_t+{\rm div}( {\bf v}_t P_t )=0\right\}.
\end{equation}
The  tangent bundle of $(\mathcal{P}_2(\R^d), \mathcal{W}_2)$ is 
$$ \mathcal{T}_{P}(\mathcal{P}_2(\R^d))= \overline{\{ \nabla \phi: \quad \phi\in \mathcal{C}^{\infty}_c(\R^d)\}}^{L^2(P)}, \quad  P\in \mathcal{P}_2(\R^d) $$
where $\mathcal{C}^{\infty}_c(\R^d)$ denotes the set of infinitely differentiable functions with compact support.
Above, $\overline{A}^{L^2(P)}$ denotes the closure of a subset $A$ in the Hilbert space $L^2(P)$. 
Given two probability measures $P$ and $Q$, a geodesic is any curve $\{\gamma_{t}^{P\to Q}\}_{t\in [0,1]}$ with endpoints $\gamma_0^{P\to Q}=P$ and $\gamma_1^{P\to Q}=Q$ with minimal velocity, i.e., any element of 
\begin{equation}
	\label{eq:WassersteinVar2}
	\argmin \left\{\int_{0}^1 \|{\bf v}_t\|_{L^2(\gamma_t)}^2 \mathrm{d} t:  \quad \partial_t \gamma_t+{\rm div}( {\bf v}_t \gamma_t )=0, \ \gamma_0=P \ {\rm and}\ \gamma_1=Q\right\}.
\end{equation}
If $P$ belongs to $ \mathcal{P}^{a.c}_2(\R^d)$, there exists a unique geodesic given by the relation 
\begin{equation} \label{eq:unique:geodesic}
	\gamma_t^{P\to Q}= ( (1-t) {\rm I}+t T_{P,Q})_{\#} P.
\end{equation}  
Its velocity field at $t=0$ is 
$ {\bf v}_0^{P\to Q}= T_{P,Q}- {\rm I} $ and the Riemannian inner product in $ \mathcal{T}_{P}(\mathcal{P}_2(\R^d))$ is $\langle \cdot, \cdot\rangle_{L^2(P)}$.
Therefore, the Wasserstein spatial depth of a probability measure $Q\in \mathcal{P}_2^{a.c}(\R^d)$ with respect to a probability measure ${\bf P}$ over $\mathcal{P}_2(\R^d)$ is defined as 
$$  {\rm WSD}(Q; {\bf P}):=1-   \left\|\E_{P\sim \mathbf{P}}\left[  \frac{ {\bf v}_0^{Q\to P} }{\|{\bf v}_0^{Q\to P}\|_{L^2(Q)}}\right]\right\|_{L^2(Q)}. $$
Since ${\bf v}_0^{Q\to P}= T_{Q,P}- {\rm I} $ we get the following definition of spatial depth.   
\begin{Definition}\label{Wdepth_definition}
	The Wasserstein spatial depth of a distribution $Q \in  \mathcal{P}^{a.c}_2(\R^d)$ with respect to a distribution of distributions $\mathbf{P}\in \mathcal{P}(\mathcal{P}_2(\R^d)) $ is defined as 
	$$  {\rm WSD}(Q; {\bf P}):=1-  \left(\int{ \left\|\E_{P\sim \mathbf{P}}\left[  \frac{\x-T_{Q,P}(\x) }{\mathcal{W}_2(P,Q)}\right]\right\|^2} \mathrm{d}Q(\x)\right)^{\frac{1}{2}}. $$
	When $\mathbb{P}_{P\sim {\bf P}}(\mathcal{W}_2(P,Q)=0)\neq 0$, we set $\frac{\x-T_{Q,P}(\x) }{\mathcal{W}_2(P,Q)}=0$ for all $\x$ when $\mathcal{W}_2(P,Q)=0$.
\end{Definition}

Note that the definition of $ {\rm WSD}(Q; {\bf P})$ is focused on absolutely continuous distributions $Q$, while the distributions  that ${\bf P}$ samples can be arbitrary (for instance, absolutely continuous, discrete, or a mixture of both).
We also refer to the discussion in 
Section~\ref{section:future:directions} on this point.

\section{Examples}\label{section:examples}
In this section we give several examples where the WSD can be computed explicitly.

\subsection{Univariate case}
In the case of univariate distributions, WSD reduces to quantile spatial depth.
The univariate Wasserstein distance has a flat structure since there is an isometric homeomorphism between distributions and the corresponding generalized quantile functions. Consequently, the Wasserstein distance between univariate distributions $P$ and $Q$ has the simple form
$$ \mathcal{W}_2^2(P,Q)=\int_{0}^1 (F^{-1}_P(u)- F^{-1}_Q(u) )^2 \mathrm{d}u,$$
where 
$ F^{-1}_P(u)=\inf\{x\in \R: \ u\leq P((-\infty, x])\}$.
Moreover, the univariate case is the unique case where the composition of optimal transport maps (here non-decreasing functions) is still an optimal transport map. Therefore, the spatial depth is just 
$$  {\rm WSD}(Q; {\bf P}):=1-  \left(\int_0^1{ \left(\E_{P\sim \mathbf{P}}\left[  \frac{F^{-1}_P(u)- F^{-1}_Q(u)}{\left(\int_{0}^1 (F^{-1}_P(u)- F^{-1}_Q(u) )^2 \mathrm{d}u\right)^{\frac{1}{2}}}\right]\right)^2} \mathrm{d}u\right)^{\frac{1}{2}}, $$
which in short notation stands
$$  {\rm WSD}(Q; {\bf P}):=1-  \left\|{ \E_{P\sim \mathbf{P}}\left[  \frac{F^{-1}_P- F^{-1}_Q}{\left\|F^{-1}_P- F^{-1}_Q \right\|_{L^2([0,1])} }\right]}\right\|_{L^2([0,1])}, $$
which is the spatial depth of the quantile functions in the Hilbert space $L^2([0,1])$ (see \cite{serfling2002depth,vardi2000multivariate,virta2023spatial}).

\subsection{Location families}
Consider that $\mathbf{P}$ is supported on a location family, a set of shifted distributions indexed by the location parameter $\boldsymbol{\theta}$, where $P \sim \mathbf{P}$ has parameter $\boldsymbol{\theta}_P$. 
In this case, $\mathbf{P}$ coincides with the distribution of $\boldsymbol{\theta}_P$. Then the WSD reduces to
\begin{equation}\label{spatial_depth}
	{\rm WSD}(Q; {\bf P})= 1- \left\|\E_{P\sim \mathbf{P}}  \Bigg( \frac{\boldsymbol{\theta}_P-\boldsymbol{\theta}_Q }{\|\boldsymbol{\theta}_P-\boldsymbol{\theta}_Q\|} \Bigg) \right\| = 1- \left\|\E_{\boldsymbol{\theta}}  \Bigg( \frac{\boldsymbol{\theta}-\boldsymbol{\theta}_Q }{\|\boldsymbol{\theta}-\boldsymbol{\theta}_Q\|} \Bigg) \right\|,
\end{equation}
which is the Euclidean spatial depth of $\boldsymbol{\theta}_Q$ with respect to the distribution of $\boldsymbol{\theta}$. This also includes the Gaussian location family (see below). 

\subsection{Gaussian families}\label{gaussian_exp}
It is well known that the optimal transport problem between Gaussian probability measures admits a closed form (see \cite{CuestaAlbertos1996OnLB}). In particular if $Q$ and $P$ are non degenerated Gaussian with means $\boldsymbol{\mu}_Q$ and $\boldsymbol{\mu}_P$ and (invertible) covariance matrices $ \bSigma_Q $ and $\bSigma_P$, respectively, the optimal transport map $T_{Q,P}$ is 
$$ \boldsymbol{\mu}_P+
\bA_{Q,P}
(\x-\boldsymbol{\mu}_Q) $$
with
$$
\bA_{Q,P} = 
\bSigma_Q^{-\frac{1}{2}}
\left( 
\bSigma_Q^{\frac{1}{2}}
\bSigma_P
\bSigma_Q^{\frac{1}{2}}
\right)^{\frac{1}{2}}
\bSigma_Q^{-\frac{1}{2}}.
$$
Therefore, if ${\rm supp}({\bf P})$ is a set of Gaussian probability measures and $Q$ is a non-degenerated Gaussian, 
then the WSD can be equivalently formulated as
\begin{multline*}
	{\rm WSD}(Q; {\bf P}) =\\
	1-  \left(\int{ \left\|\E_{P\sim \mathbf{P}}\left[  \frac{\x-\boldsymbol{\mu}_P-\bA_{Q,P}(\x-\boldsymbol{\mu}_Q) }{\left(\|\boldsymbol{\mu}_P-\boldsymbol{\mu}_Q\|^2+{\rm Tr}\left(\bSigma_P+\bSigma_Q-2\left(\bSigma^{\frac{1}{2}}_P \bSigma_Q \bSigma^{\frac{1}{2}}_P\right)^{\frac{1}{2}}\right)\right)^{\frac{1}{2}}}\right]\right\|^2} \mathrm{d}Q(\x)\right)^{\frac{1}{2}}. 
\end{multline*}
In the special case of a common $\bSigma$ for all $P\in {\rm supp}(\mathbf{P})$, and when $\bSigma_Q  = \bSigma$, the above formula reduces to
\begin{equation*}
	{\rm WSD}(Q; {\bf P})=
	1- \left\|\E_{P\sim \mathbf{P}}  \Bigg( \frac{\boldsymbol{\mu}_P-\boldsymbol{\mu}_Q }{\|\boldsymbol{\mu}_P-\boldsymbol{\mu}_Q\|} \Bigg) \right\|,
\end{equation*}
which is the Euclidean spatial depth function of $\boldsymbol{\mu}_Q$.
When $\mathbf{P} = \frac{1}{n}\sum_{i=1}^n \delta_{P_i}$, we obtain
\begin{multline*}
	{\rm WSD}(Q; {\bf P}):=\\
	1 - \left(\int{ \left\|\frac{1}{n} \sum_{i=1}^n \left[  \frac{\x-\boldsymbol{\mu}_{P_i}-\bA_{Q,P_i} (\x-\boldsymbol{\mu}_Q) }{\left(\|\boldsymbol{\mu}_{P_i}-\boldsymbol{\mu}_Q\|^2+{\rm Tr}\left(\bSigma_{P_i}+\bSigma_Q-2\left(\bSigma^{\frac{1}{2}}_{P_i} \bSigma_Q \bSigma^{\frac{1}{2}}_{P_i}\right)^{\frac{1}{2}}\right)\right)^{\frac{1}{2}}}\right]\right\|^2} \mathrm{d}Q(\x)\right)^{\frac{1}{2}}.  
\end{multline*}
Furthermore, if $\bSigma_{P_i}=\bSigma_{Q}$ for all $i=1, \dots, n$, the WSD is 
$$ 1- \left(\int{ \left\| \frac{1}{n}\sum_{i=1}^n   \frac{\boldsymbol{\mu}_{P_i}- \boldsymbol{\mu}_Q }{\|\boldsymbol{\mu}_{P_i}-\boldsymbol{\mu}_Q\|}\right\|^2} \mathrm{d}Q(\x)\right)^{\frac{1}{2}}=1- \left\|\frac{1}{n} \sum_{i=1}^n   \frac{\boldsymbol{\mu}_{P_i}- \boldsymbol{\mu}_Q }{\|\boldsymbol{\mu}_{P_i}-\boldsymbol{\mu}_Q\|}\right\|.$$

\section{Properties of Wasserstein spatial depth}
\label{section:properties}

Zuo and Serfling 
postulated in \cite{zuo2000general} 
the main four properties that a data depth  should satisfy in Euclidean spaces. Those properties are  {\it affine invariance}, meaning that the data depth function is invariant to affine transformations; {\it center-outward monotonicity}, meaning that the depth function decreases along rays arising from the deepest point;  {\it vanishing at infinity}, meaning that the depth function tends to $0$ as the distance to the deepest point tends to infinity; {\it maximality at the center}, meaning that for elliptic distributions, its geometric center is the unique deepest point. The Euclidean spatial depth satisfies some of these properties. In particular, it is invariant to isometric  transformations, it vanishes at infinity and, if the spatial median is unique it is the unique maximizer of the spatial  depth.    As $\R^d$ is trivially embedded on 
$\mathcal{P}_2(\R^d)$ by means of the mapping $\x\mapsto \delta_\x$, we cannot expect better properties for the Wasserstein space adaptation. 

\subsection{General properties}
In this section we prove that the WSD  shares the main properties of the Euclidean spatial depth, i.e., it belongs to the interval $[0,1]$, it decreases at infinity and it is transformation invariant.
\begin{Theorem}\label{Theo:Main}
	Set  $\mathbf{P}\in \mathcal{P}(\mathcal{P}_2(\R^d)) $. Then the following properties hold: 
	\begin{enumerate}
		\item \big(Values in $[0,1]$.\big) \label{th:pointIn01} $ {\rm WSD}(Q; {\bf P})\in [0,1]$ for all $Q\in  \mathcal{P}^{a.c}_2(\R^d)$.
		\item \big(Transformation invariance.\big) \label{TransformInv}  Assume that $d\geq 2$. Then for any  isometry $F:\mathcal{P}_2(\R^d)\to \mathcal{P}_2(\R^d)$, it holds that 
		$$   {\rm WSD}( F(Q); F_{\#} {\bf P})= {\rm WSD}(Q;  {\bf P}), \quad \text{for all $Q\in  \mathcal{P}^{a.c}_2(\R^d)$.} $$
		\item \big(Vanishing at infinity.\big) Let $\{Q_n\}_{n\in \N}\subset   \mathcal{P}^{a.c}_2(\R^d)$ be a sequence such that $\mathcal{W}_2(Q_n,Q)\to +\infty$, for one $Q\in \mathcal{P}_2(\R^d)$, then 
		$  {\rm WSD}(Q_n; {\bf P}) \to 0 . $
	\end{enumerate}
\end{Theorem}

Recall from \cite{bertrand2012geometric} that there are three types of isometries in $(\mathcal{P}_2(\R^d), \mathcal{W}_2)$.
Let  $F:\mathcal{P}_2(\R^d)\to \mathcal{P}_2(\R^d)$ be an isometry, i.e., 
$$ \mathcal{W}_2(F(P), F(Q))=\mathcal{W}_2(P,Q)\quad \text{for all }P,Q\in \mathcal{P}_2(\R^d).$$
Then $F$ is called {\it trivial}  if there exists an isometry  $  f:\R^d\to \R^d$ such that 
$ F(P)=f_\# P$ for all $P\in \mathcal{P}_2(\R^d)$; $F$  is said to {\it preserve shapes} if for all $P\in\mathcal{P}_2(\R^d) $  there exists an isometry  $  f=f_P:\R^d\to \R^d$ such that $ F(P)=f_\# P$; and if $F$ does not preserve shapes, it is said to be {\it exotic}.  An example of nontrivial isometry  on $(\mathcal{P}_2(\R^d), \mathcal{W}_2)$ that preserves shapes is given by the mapping $\Phi(\varphi):P\mapsto \Phi(\varphi)(P)$  where $\varphi:\R^d\to \R^d$ is a linear isometry and $\Phi(\varphi)(P)$ is the
law of the random variable 
$$ \varphi(\X-\E[\X])+\E[\X], \quad {\rm for
}\ \X\sim P . $$
Theorems 1.1 and 1.2 in \cite{bertrand2012geometric} prove that $(\mathcal{P}_2(\R^d), \mathcal{W}_2)$ admits exotic isometries if and only if $d=1$, which is the reason for which the invariance of the WSD holds for $d\geq 2$. 

An interesting property of the Euclidean spatial depth is the geodesic invariance: Set $\lambda\in (0,1)$ and let  $P_{\lambda,\x}$ be the distribution of $ \lambda  \X+ (1-\lambda ) \x $, where $\X\sim P$ and $\x\in \R^d$. Then
$$  \rm SD(\x; P_{\lambda,\x})=1-\left\|\E\left[\frac{\lambda  \X+ (1-\lambda ) \x-\x}{\|\lambda  \X+ (1-\lambda ) \x-\x\|}\right]\right\|= 1-\left\|\E\left[\frac{ \X-\x}{\| \X-\x\|}\right]\right\|= \rm SD(\x; P) .$$
In the following result we show that the analogous result holds for the WSD.
\begin{Lemma}\label{lemma:geo-inva}
	Set  $\mathbf{P}\in \mathcal{P}(\mathcal{P}_2(\R^d)) $,   $P\sim {\bf P}$ and $ Q\in  \mathcal{P}^{a.c}_2(\R^d)$. Let $\{\gamma_t^{Q\to P} \}_{t\in [0,1]}$ be a Wasserstein Geodesic with endpoints $\gamma^{Q\to P}_0 = Q $ and $\gamma^{Q\to P}_1 = P $. Let ${\bf P}_{\lambda,Q}$ be the distribution of  $\gamma^{Q\to P}_\lambda$. Then it follows that 
	$$   {\rm WSD}(Q; {\bf P}_{\lambda,Q}) =  {\rm WSD}(Q; {\bf P})\quad \text{for all }\lambda\in (0,1).$$
\end{Lemma}

\subsection{Maximality at the center}
The set of spatial medians of ${\bf P}$ is defined as
$$ \argmin_{Q\in \mathcal{P}_2(\R^d)}\E_{P\sim \mathbf{P}}[\mathcal{W}_2(P,Q)] .$$
The following result shows that, under some assumptions,  the set of spatial medians which are absolutely continuous with respect to Lebesgue measure has maximum depth.
\begin{Theorem}\label{Theo:RelationSpatial}
	Set  $\mathbf{P}\in \mathcal{P}(\mathcal{P}_2(\mathcal{K})) $ for a compact set $\mathcal{K}\subset \R^d$. Assume that $\mathbf{P}$ is supported on a finite set $\{P_1, \dots, P_n\}$. Then any 
	$$ Q\in \mathcal{P}^{a.c}_2(\mathcal{K})\cap \argmin_{Q'\in \mathcal{P}(\mathcal{K})}\E_{P\sim \mathbf{P}}[\mathcal{W}_2(P,Q')] $$
	such that $Q\neq P_i $ for all $i=1, \dots, n$ satisfies
	$ {\rm WSD}(Q; {\bf P})=1$.
\end{Theorem}
\begin{Remark}
	We do not know if the set of spatial medians which are absolutely continuous with respect to Lebesgue measure is nonempty. It is known that,  under the setting of Theorem~\ref{Theo:RelationSpatial}, if we assume that $P_i\in \mathcal{P}^{a.c}_2(\mathcal{K})$,  the set of geometric means (or barycenters) is a singleton and its unique element belongs to $\mathcal{P}^{a.c}_2(\mathcal{K})$ (see \cite{AlvarezEsteban2016}). However, the proof of \cite{AlvarezEsteban2016}, based on a fixed point argument which exploits the strict convexity of the squared Wasserstein distance, does not apply to our setting. 
\end{Remark}
\begin{Remark}
	The existence of a point with WSD equal to one has been shown in \cite{you2025wasserstein}.
\end{Remark}

\subsection{Robustness properties} \label{subsection:robustness}
Two measures of robustness are popular in the literature. The first one is the breakdown point and the second is the influence function (see \cite{Hampel.2005.book,Huber.2009.book}).

Let ${\bf P}_n$ be the empirical distribution of $P_1, \dots, P_n \stackrel{i.i.d.}{\sim} {\bf P}$. 
In this section we show that the empirical depth regions 
$$ \mathcal{R}(\alpha;{\bf P}_n)= \{ Q\in \mathcal{P}_2^{a.c}(\R^d):\ { \rm WSD}(Q;{\bf P}_n) \ge \alpha \}, \quad \alpha\in (0,1],$$
and the deepest points 
$$ {\rm m}(\alpha;{\bf P}_n)=  \{ Q\in \mathcal{P}_2^{a.c}(\R^d):\ { \rm WSD}(Q;{\bf P}_n)=1 \},$$
which correspond to the region for $\alpha=1$,
are robust in terms of the breakdown point. Then we show that the influence function of the depth functional is bounded. Now we define our notion of breakdown point and provide upper and lower bounds for it. We note that the deepest points have breakdown point above $1/2$ irrespective of the uniqueness of the spatial median and, even more, irrespective of the fact that the spatial median agrees with the deepest points.
\begin{Definition}
	The breakdown point of $\mathcal{R}(\alpha;{\bf P}_n)$ is the value 
	$$ {\rm BP}(\mathcal{R}(\alpha;{\bf P}_n)) = \frac{1}{n}\inf\left\{ \ell \in \{1, \dots, n\}:\  \sup_{{\bf Q}_{n}\in \mathcal{P}({\ell,n})}d_H \left(\mathcal{R}(\alpha; {\bf P}_n) ,\mathcal{R}(\alpha;{\bf Q}_{n})\right)=+\infty  \right\}, $$
	where $\mathcal{P}({\ell,n})$ denotes the set of probability  measures supported on $n$ points (with weights $1/n$)  sharing at least $n-\ell$ with ${\bf P}_n$. Here $d_H$ denotes the Hausdorff distance, defined for $A,B\subset \mathcal{P}_2(\R^d)$ as 
	\[
	d_H(A, B) = \max\left\{
	\sup_{P \in A} \inf_{Q \in B} \mathcal{W}_2(P, Q), \;
	\sup_{P \in B} \inf_{Q \in A} \mathcal{W}_2(P, Q)
	\right\}.
	\]
\end{Definition}
In the Euclidean setting, it was recently shown by Konen and  Paindaveine \cite{Dimitri}  that the spatial depth region (or contour) of order $\alpha\in (0,1] $ has breakdown point $ \alpha/2  $ (up to the order $1/n$). The following result shows that the same happens for the WSD. Note that, even though Lemma \ref{lemma:breakdown} provides a similar statement as in \cite{Dimitri}, the proof of \cite{Dimitri} cannot be adapted to our setting. Indeed, \cite{Dimitri} considers quantile contours defined as minimizers of a pinball-type loss, while our depth definition cannot be written as a minimization problem.

\begin{Lemma}\label{lemma:breakdown}
	For every $\alpha\in (0,1]$, we have
	\[ {\rm BP}(\mathcal{R}(\alpha;{\bf P}_n)) 
	\ge \frac{\alpha}{2}
	\]
	and for every $\alpha\in (0,1 - \frac{2}{n} ]$ (for $n \ge 3$), we have
	\[
	{\rm BP}(\mathcal{R}(\alpha;{\bf P}_n)) \le
	\frac{\alpha}{2} + \frac{1}{n}. 
	\]
\end{Lemma}

\begin{Remark}
	The breakdown point in the two--stage sampling model (see Section \ref{two_stg_sampling}) is arbitrary low in the subsample size $m$. Indeed, perturbing  one point of each of the columns of \eqref{eq:dataTensor} can imply that all of the empirical measures   $P_{i,m}:=\frac{1}{m}\sum_{j=1}^m\delta_{\X_{i,j} }$ diverge. In such a case, we recommend the reader the recent articles \cite{AvellaGonzalezSanz.2024} and \cite{paindaveine.2024.robustnesssemidiscreteoptimaltransport} to construct trimmed estimators of the optimal transport maps.  
\end{Remark}
The other important measure of robustness is the so-called influence function introduced by  Hampel \citep{Hampel1974infuence}. In our setting, the influence curve of the WSD at a point $Q \in  \mathcal{P}^{a.c}_2(\R^d) $ is the function 
$$\mathcal{P}_2(\R^d) \ni  \mu\mapsto  {\rm IC}( \mu, { \rm WSD}(Q;{\bf P} )) =\lim_{t\to 0^+}\frac{{ \rm WSD}(Q;{\bf P}+ t(\delta_\mu-{\bf P}) )- { \rm WSD}(Q;{\bf P} ) }{t}.$$
A functional is said to be robust if the gross error sensitivity is bounded, i.e., if 
$$ \sup_{\mu
	\in \mathcal{P}_2(\R^d)
} |{\rm IC}( \mu, {\rm WSD}(Q;{\bf P} ))| <\infty. $$
In the following result we show that the influence curve exists and is bounded. 
\begin{Lemma}\label{lemma:influence}
	Let  $\mathbf{P}\in \mathcal{P}(\mathcal{P}_2(\R^d)) $ be atomless.
	Let $Q\in  \mathcal{P}^{a.c}_2(\R^d)$.
	Then it follows that, if  ${\rm WSD}(Q;{\bf P})\neq 1$, 
	for any $\mu \in \mathcal{P}_2(\R^d)$,
	$${\rm IC}( \mu, {\rm WSD}(Q;{\bf P} ))= -\frac{\int
		\left\langle
		\E_{P\sim \mathbf{P}}\left[  \frac{\x-T_{Q,P}(\x) }{\mathcal{W}_2(P,Q)}\right],
		\left(  \frac{\x-T_{Q,\mu}(\x) }{\mathcal{W}_2(\mu,Q)}- \E_{P\sim \mathbf{P}}\left[  \frac{\x-T_{Q,P}(\x)  }{\mathcal{W}_2(P,Q)}\right] \right)
		\right\rangle
		dQ(\x)}{\left( \int \left\| \E_{P\sim \mathbf{P}}\left[  \frac{\x-T_{Q,P}(\x) }{\mathcal{W}_2(P,Q)}\right] \right\|^2 dQ(\x) \right)^{\frac{1}{2}}}  $$
	and if   ${\rm WSD}(Q;{\bf P})=1$,
	$${\rm IC}( \mu, {\rm WSD}(Q;{\bf P} ))= -\left(\int  \left\|  \frac{\x-T_{Q,\mu}(\x) }{\mathcal{W}_2(\mu,Q)}- \E_{P\sim \mathbf{P}}\left[  \frac{\x-T_{Q,P}(\x)  }{\mathcal{W}_2(P,Q)}\right] \right\|^2 dQ(\x)\right)^{\frac{1}{2}} . $$
	As a consequence, 
	$\displaystyle \sup_{\mu \in 
		\mathcal{P}_2(\R^d)   
	} |{\rm IC}( \mu, {\rm WSD}(Q;{\bf P} ))| \leq 2$. 
\end{Lemma}

\subsection{Continuity}\label{Sec:Continuity}
In this section we investigate some topological properties of the WSD.
We analyze separately the functions $Q\mapsto  {\rm WSD}(Q; {\bf P})$ and ${\bf P}\mapsto  {\rm WSD}(Q; {\bf P})$. The following result shows that the function $\mathcal{P}_2^{a.c}(\R^d)\ni Q\mapsto  {\rm WSD}(Q; {\bf P})$ is continuous.  
\begin{Theorem}\label{Theo:COnti}
	Let  $\mathbf{P}\in \mathcal{P}(\mathcal{P}_2(\R^d)) $ be atomless  and $\{Q_n\}_{n\in \N}\subset \mathcal{P}_2^{a.c}(\R^d)$  be a sequence such that $\mathcal{W}_2(Q_n,Q)\to 0$ for some $Q\in \mathcal{P}_2^{a.c}(\R^d)$. Then $$\lim_{n\to \infty}  {\rm WSD}(Q_n; {\bf P}) =  {\rm WSD}(Q; {\bf P}) .$$ 
\end{Theorem}
Next we show the continuity of ${\bf P}\mapsto  {\rm WSD}(Q; {\bf P})$ for fixed $Q$. As an intermediate step we need to show that for each $Q\in \mathcal{P}_{2}^{a.c}(\R^d)$ the function 
$$ T^Q: \mathcal{P}_2(\R^d)\ni P\mapsto T_{Q,P}\in L^2(Q)$$
is continuous. Recall that $T_{Q,P}$ is the optimal transport map from $Q$ to $P$. 

\begin{Lemma}[Continuity of $T^Q$]\label{Stabilty} Set $Q\in \mathcal{P}_2^{a.c}(\R^d)$. Let  $\{P_n\}_n\subset \mathcal{P}_2(\R^d) $ be a sequence of  probability measures such that  $\mathcal{W}_{2}(P_n, P)\to 0 $ for some $P\in \mathcal{P}_2(\R^d)$. Then $$\| T_{Q,P_n} - T_{Q,P}\|_{L^2(Q)} \to 0.$$ In words, $T^Q:\mathcal{P}_{2}(\R^d)\to L^2(Q)$ is continuous.
\end{Lemma}

Fix $Q\in \mathcal{P}_2^{a.c}(\R^d)$.  Lemma~\ref{Stabilty} implies that the function 
$$ \mathcal{P}_{2}(\R^d)\ni  P\mapsto  \frac{I-T_{Q,P} }{\mathcal{W}_2(P,Q)} \in  L^2(Q) $$
is  continuous around all $P\neq Q$. This observation enables to derive the continuity of the function 
$\mathcal{P}(\mathcal{P}_{2}(\R^d)) \ni {\bf P}\mapsto  {\rm WSD}(Q; {\bf P})$ around  atomless probability measures. 

\begin{Theorem}\label{Theorem:contP}
	Let $\mathbf{P}\in \mathcal{P}(\mathcal{P}_{2}(\R^d)) $ be atomless and let $Q\in \mathcal{P}_2^{a.c}(\R^d)$. 
	Then 
	$$\lim_{n\to \infty}  {\rm WSD}(Q; {\bf P}_n) = {\rm WSD}(Q; {\bf P}) $$
	for every sequence $\{{\bf P}_n\}_{n\in \N}\subset \mathcal{P}(\mathcal{P}_{2}(\R^d))$ such that 
	$ {\bf P}_n \rightarrow {\bf P}$ weakly in $\mathcal{P}(\mathcal{P}_{2}(\R^d))$.
\end{Theorem}

\section{Consistent estimation} \label{consistency}
In practice, we only observe sample datasets instead of knowing the true ${\bf P}$ or even the true $P_1, P_2, \ldots, P_n \sim {\bf P}$.
Two common scenarios in the literature of distributional data learning \citep{bachoc2023improved, meunier2022distribution,szabo2016learning} will be considered, namely, {\it one-stage sampling model} and {\it two-stage sampling model}.
One-stage sampling model assumes the observation of an \textit{i.i.d.}\ sample $P_1, \dots, P_n$ of ${\bf P}$. Two-stage sampling model assumes the observation of a data array
\begin{equation}\label{eq:dataTensor}
	\left(\begin{array}{ccc}
		\X_{1,1}  & \dots & \X_{1,m} \\
		\vdots  & \vdots & \vdots\\
		\X_{n,1}  & \dots & \X_{n,m}
	\end{array}\right),
\end{equation}
where $\X_{i,1}, \dots, \X_{i,m} \in \R^d$ is an \textit{i.i.d.} sample from the distribution $P_i$ for each $i=1, \dots, n$, and $P_1, \dots, P_n$ are \textit{i.i.d.} drawn from ${\bf P}$. 
The difference between the two models is that the sampled distributions $P_1, \ldots, P_n$ are known in one-stage sampling model, but unknown and to be estimated by the empirical distributions in two-stage sampling model. Note that the sample sizes drawn from different $P_i$'s can be different without creating technical or computational difficulties. We set them to be the same $m$ just for simpler notations.

In each scenario, we give the empirical counterpart to the population WSD in Definition~\ref{Wdepth_definition}.  
We also establish a point-wise central limit theorem for the empirical WSD under the {\it one-stage sampling model} and a consistency result
and rates of convergence 
for the   {\it two-stage sampling model}.

\subsection{One-stage sampling}\label{one_stg_sampling}
We describe the asymptotic behavior of the empirical WSD 
\begin{equation}\label{one-stage-sampling}
	{\rm WSD}(Q; {\bf P}_n):=1-  \left(\int{ \left\|\frac{1}{n}\sum_{i=1}^n \left[  \frac{\x-T_{Q,P_i}(\x) }{\mathcal{W}_2(Q, P_i)}\right]\right\|^2} \mathrm{d}Q(\x)\right)^{\frac{1}{2}}, \quad {\bf P}_n =\frac{1}{n}\sum_{i=1}^n \delta_{P_i} ,
\end{equation}
where $\mathbf{P}_n$ is the empirical counterpart to $\mathbf{P}$.  
The WSD is associated with the spatial distribution process 
$$ S_{{\bf P}}:  \mathcal{P}^{a.c}_2(\R^d)\ni Q \mapsto S_{{\bf P},Q}=\E_{P\sim \mathbf{P}}\left[  \frac{I-T_{Q,P} }{\mathcal{W}_2(P,Q)}\right] \in L^2(Q).$$
The representation 
$$ S_{{\bf P},Q} =\E_{P\sim \mathbf{P}}\left[  \frac{I-T_{Q,P} }{\|I-T_{Q,P}\|_{L^2(Q)}}\right]  $$
allows to use 
standard techniques to obtain the point-wise  strong law of large numbers and a central limit theorem for
the empirical spatial distribution process
$\left(S_{{\bf P}_n,Q }-S_{{\bf P},Q} \right) \in L^2(Q),  $
after showing that the random function $T_{Q,P} $ in $L^2(Q)$ is tight.
Recall that a probability measure  $\mu\in \mathcal{P}(\mathcal{X})$ over a separable topological space $(\mathcal{X},d)$  is said to be tight if for every $\epsilon>0$ there exists a compact (in the metric topology) set $K$  such that $\mu(K)\geq 1-\epsilon.$  A random variable is tight if its distribution is tight.     
The space $\mathcal{P}(\mathcal{P}_{2}(\R^d))$
endowed with the distance
$\mathcal{W}_2$
is Polish from \cite[Proposition 7.1.5]{AmbrosioGradient}.
Hence from Ulam's theorem, any
fixed ${\bf P}\in \mathcal{P}(\mathcal{P}_{2}(\R^d))$ is tight.
Therefore, Lemma~\ref{Stabilty} implies that $T_{Q,P}$ is tight in  $L^2(Q)$.  

As a consequence, $\left\{\frac{I-T_{Q,P_i} }{\|I-T_{Q,P_i}\|_{L^2(Q)}}\right\}_{i=1}^n$ is an \textit{i.i.d.} sequence of tight random elements in the separable Hilbert space $L^2(Q)$, with finite second order moments.  The strong law of large numbers and the central limit theorem in separable Hilbert spaces (cf.\ \cite[Corollary~10.9]{Ledoux1991}) yield the following result. 
Note that a random element $Z$ of a Hilbert space $\mathcal{H}$ is defined to be Gaussian when $h(Z)$ follows a (univariate) Gaussian distribution for all linear continuous mappings $h : \mathcal{H} \to \mathbb{R}$.

\begin{Theorem}\label{Theorem:Glivenk}
	Set  $\mathbf{P}\in \mathcal{P}(\mathcal{P}_2(\R^d)) $, $Q\in \mathcal{P}^{a.c}_2(\R^d) $. Then 
	\begin{equation} \label{eq:LLN:one}
		\|S_{{\bf P}_n,Q } -S_{{\bf P},Q }\|_{L^2(Q)} \xrightarrow{a.s} 0 
	\end{equation}
	and 
	$$ \sqrt{n}(S_{{\bf P}_n,Q } -S_{{\bf P},Q }) \xrightarrow{\mathcal{P}(L^2(Q))}\mathbb{G}_{{\bf P},Q}, $$
	for some centered Gaussian element $\mathbb{G}_{{\bf P},Q}\in L^2(Q) $. 
	As a consequence, it holds that
	\begin{equation} \label{eq:LLN:two}
		{\rm WSD}(Q; {\bf P}_n )\xrightarrow{a.s}  {\rm WSD}(Q; {\bf P} )
	\end{equation}
	and, if $ {\rm WSD}(Q; {\bf P} )<1$,  also
	\begin{equation} \label{eq:CLT}
		\sqrt{n}( {\rm WSD}(Q; {\bf P}_n )- {\rm WSD}(Q; {\bf P} ))\xrightarrow{\mathcal{P}(\mathbb{R})} \frac{\langle \mathbb{G}_{{\bf P},Q}, S_{{\bf P},Q } \rangle_{L^2(Q)} }{ {\rm WSD}(Q; {\bf P} )-1}. 
	\end{equation}
\end{Theorem}
The last two statements of Theorem~\ref{Theorem:Glivenk} are a mere application of the delta method.

Now, we provide a consistent estimator of the asymptotic variance in \eqref{eq:CLT}. We define
\[
\widehat{\sigma}_{n,Q}^2
=
\frac{
	\frac{1}{n}\sum_{i=1}^n
	\left( 
	\langle
	\xi_i
	,
	S_{{\bf P}_n,Q }
	\rangle_{L^2(Q)}
	\right)^2
	-
	\left\| 
	S_{{\bf P}_n,Q }
	\right\|_{L^2(Q)}^4
}{
	\left(
	{\rm WSD}(Q; {\bf P}_n ) - 1
	\right)^2
},
\]
where
\[
\xi_i = 
\frac{I-T_{Q,P_i} }{\|I-T_{Q,P_i}\|_{L^2(Q)}}
\in L^2(Q).
\]

\begin{Lemma} \label{lem:asympt:var:CLT}
	Let  $\mathbf{P}\in \mathcal{P}(\mathcal{P}_2(\R^d)) $ and $Q\in \mathcal{P}^{a.c}_2(\R^d) $ be such that $ {\rm WSD}(Q; {\bf P} )<1$. Then, 
	\[
	\widehat{\sigma}_{n,Q}^2
	\xrightarrow{a.s} 
	\sigma_{Q}^2
	: = 
	\Var
	\left( 
	\frac{\langle \mathbb{G}_{{\bf P},Q}, S_{{\bf P},Q } \rangle_{L^2(Q)} }{ {\rm WSD}(Q; {\bf P} )-1}
	\right).
	\]
\end{Lemma}

\subsection{Two-stage sampling}\label{two_stg_sampling}
Now we deal with the scenario where only the data array \eqref{eq:dataTensor}  is available. Recall that, in this case, the \textit{i.i.d.}  samples  $P_1, \dots,P_n $ of ${\bf P}$ are no longer observed but a sample $\{\X_{i,j}\}_{i,j=1}^{n,m}$ is, where $\X_{i,j}\sim P_i$ for $j\in \{1, \dots, m\}$ and each $i\in \{1, \dots, n\}$ .  
We denote 
\begin{equation}
	\label{ModelTwo}
	{\bf P}_{n,m}:=\frac{1}{n}\sum_{i=1}^n \delta_{P_{i,m}}, \quad \text{with} \quad    P_{i,m}:=\frac{1}{m}\sum_{j=1}^m\delta_{\X_{i,j} }\quad \text{for each}\quad  i\in \{1, \dots, n\}.
\end{equation}
Correspondingly, the empirical WSD is formulated as 
\begin{equation*}
	{\rm WSD}(Q_m; {\bf P}_{n,m}) = 1- \sqrt{ \frac{1}{m} \sum_{j=1}^m  \bigg\| \frac{1}{n}\sum_{i=1}^n \frac{\X_{q,j}-T_{Q_m, P_{i,m}}(\X_{q,j})}{\mathcal{W}_2(Q_m, P_{i,m})} \bigg\|^2 },
\end{equation*}
where $Q_m = \frac{1}{m}\sum_{j=1}^m \delta_{\X_{q,j}}$ with $\X_{q,1} , \ldots , \X_{q,m} \stackrel{\textit{i.i.d.}}{\sim} Q$, and the convention $\frac{\mathbf{0}}{0}=\mathbf{0}$ remains.
Here we use the convention $q=n+1$ for convenience.
Now we show that, as $n,m\to \infty$, 
$   {\bf P}_{n,m} $ converges in probability in  $  \mathcal{P}(\mathcal{P}_p(\R^d)) $ for all $p\geq 1$. We endow $\mathcal{P}(\mathcal{P}_{p}(\R^d))$ with the metric 
\begin{multline*}
	d_{{\rm BL}^{(p)}}({\bf P}, {\bf Q}) = \sup\bigg\{ \int f(P)\mathrm{d}({\bf P}- {\bf Q})(P): \   |f(P)|\leq 1\ {\rm and}\\ |f(P)-f(Q) |\leq \mathcal{W}_{p}(P,Q), \ \forall \, P,Q\in \mathcal{P}_{p}(\R^d)\bigg\}.
\end{multline*}
\begin{Lemma}\label{Lemma:ConvergenceBL}
	Let $\{{\bf P}_{n,m}\}_{n\in \N}$ be as in \eqref{ModelTwo} where $m=m(n)$ is such that $m\to \infty$ as $n\to \infty$. Assume that ${\bf P}\in \mathcal{P}(\mathcal{P}_{p}(\R^d))$ for $p\geq 1$.  Then 
	$$ \E[d_{{\rm BL}^{(p)}}({\bf P}_{n,m}, {\bf P})] \longrightarrow 0 \quad \text{as } n\to \infty.$$
\end{Lemma}
A combination of Lemma~\ref{Lemma:ConvergenceBL}, Theorem~\ref{Theorem:contP} 
and the continuous mapping theorem
yields the following consistency result for the two-stage sampling estimator. 
\begin{Theorem}\label{two_stage_consistency}
	Set $\mathbf{P}\in \mathcal{P}(\mathcal{P}_{2}(\R^d)) $ be atomless. Let $\{{\bf P}_{n,m}\}_{n\in \N}$ be as in \eqref{ModelTwo} where $m=m(n)$ is such that $m\to \infty$ as $n\to \infty$. Then,   for every $Q\in \mathcal{P}^{a.c}_2(\R^d)$, 
	$$ {\rm WSD}(Q; {\bf P}_{n,m}) \overset{\mathbb{P}}{\longrightarrow}  {\rm WSD}(Q; {\bf P}) \quad \text{as } n\to \infty. $$
\end{Theorem}

Theorems \ref{Theorem:Glivenk} and \ref{two_stage_consistency} state that the empirical WSD converges to the population version asymptotically. Given enough sample size, the empirical WSD is informative of the truth and has practical value. The simulation results in Section \ref{simulation_consistency} also verify the above theorems. 

The rest of the section is devoted to  the convergence rates under the two-stage sampling model. Here we need further additional assumptions. 
Note that the definition of a $\mathcal{C}^{1,1}$ boundary is provided for instance in \cite{gilbarg2001elliptic}.

\begin{Assumption}\label{Assumtion-Reg}
	Let $P\sim {\bf P}$ and let $Q\in \mathcal{P}^{a.c}_2(\R^d) $.
	There exist two sets $\Omega$ and $\Omega'$, included in $\R^d$, that are compact and strongly convex with $\mathcal{C}^{1,1}$ boundary such that the following hold.
	\begin{enumerate}
		\item {(\it Caffarelli regularity condition on ${\bf P}$.)} There exists $\Lambda>0$ such that 
		\begin{equation}
			\label{eq:Caffarelli:P}
			\mathbb{P}_{P\sim {\bf P}}\left( \text{$P\in \mathcal{P}(\Omega)$ has density $p$ with 
				$ \Lambda^{-1}\leq p(\x) \leq \Lambda \quad \text{for all }\x\in \Omega $} \right)=1.
		\end{equation}
		\item {(\it Caffarelli regularity condition on $Q$.)} The distribution $Q$ is in $ \mathcal{P}(\Omega')$ and has density $q$ with, for the same $\Lambda >0$ as above, 
		$$ \Lambda^{-1}\leq q(\y) \leq \Lambda \quad \text{for all }\y\in \Omega'.  $$
	\end{enumerate}
	
\end{Assumption}
Under the previous assumptions, the rate of convergence under the two-stage sampling model follows from the well-known rates of the optimal transport map (see \cite[Corollary~7]{Manole.2024-Aos})
\begin{equation}
	\label{Rates-OT}
	\E[\| T_{Q,P_{1,m}}-T_{Q,P_1}\|_{L^2(Q)}^2]\leq \alpha(d,m):=C\cdot \begin{cases}
		\frac{1}{m} & {\rm if }
		\ d=1,\\
		\frac{\log(m)}{m} & {\rm if }
		\ d=2,\\
		m^{-\frac{2}{d}} & {\rm if }
		\ d>2,
	\end{cases} 
\end{equation}
for any fixed $P_1$ satisfying the event in \eqref{eq:Caffarelli:P},
where the constant $C$ depends only on the dimension $d$, the supports  $\Omega$ and $\Omega' $ and the bound $\Lambda$ on the density.
\begin{Lemma} 
	\label{lemma:rate:two:stage}
	Let Assumption \ref{Assumtion-Reg} hold. Then it follows that 
	$$ \E[|  {\rm WSD}(Q; {\bf P}_{n,m})- {\rm WSD}(Q; {\bf P}_n)| ] \leq  2\left( \alpha(d,m)\right)^{\frac{1}{2}} \E_{P_1\sim {\bf P}}\left[ W_2^{-1}(Q,P_{1})\right] ,$$
	where $\alpha(d,m)$ is as in \eqref{Rates-OT}.
\end{Lemma}

Note that $\E_{P_1\sim {\bf P}}\left[ W_2^{-1}(Q,P_{1})\right]$ is finite except in degenerate cases, essentially if, so to speak, $P_{1}$ is supported on a one-dimensional space. This is because the integral $\int_{ [-1,1]^d} \frac{1}{\|\x\|}  \mathrm{d} \x$ is infinite only when $d=1$. 

\begin{Remark}
	Consider the case where with probability one $P\sim {\bf P}$ is supported on a discrete set $\{\Z_1, \dots, \Z_k\}$. Then, the optimal transport problem is no longer cursed by dimensionality, cf.~\cite{delbarrio2021centraldisc}. In such a case it is easy to see that, if $Q$ satisfies Assumption \ref{Assumtion-Reg},   then 
	$$ \E[\| S_{{\bf P}_{n,m},Q} -S_{{\bf P}_{n},Q} \|_{L^1(Q)}] \leq C m^{-\frac{1}{2}}  $$
	and 
	$$ \E[|  {\rm WSD}(Q; {\bf P}_{n,m})- {\rm WSD}(Q; {\bf P}_n)| ] \leq  C m^{-\frac{1}{4}}, $$
	for a constant $C$,
	follow directly from the rates of convergence of the optimal transport map in the semi-discrete setting (see \cite{Sadhu.Kato-2024}).
\end{Remark}

The next corollary is immediate.
\begin{Corollary}
	Let  $\mathbf{P}\in \mathcal{P}(\mathcal{P}_2(\R^d)) $ and $Q\in \mathcal{P}^{a.c}_2(\R^d) $ be such that $ {\rm WSD}(Q; {\bf P} )<1$.
	Let Assumption \ref{Assumtion-Reg} hold. 
	Assume that $\E_{P_1\sim {\bf P}}\left[ W_2^{-1}(Q,P_{1})\right] < \infty$
	and that $\alpha(d,m) = o \left( \frac{1}{n} 
	\right) $ as $n,m \to \infty$. Then
	$$ \sqrt{n}( {\rm WSD}(Q; {\bf P}_{n,m} )- {\rm WSD}(Q; {\bf P} ))\xrightarrow{\mathcal{P}(\mathbb{R})} \frac{\langle \mathbb{G}_{{\bf P},Q}, S_{{\bf P},Q } \rangle_{L^2(Q)} }{ {\rm WSD}(Q; {\bf P} )-1}, $$ 
	as $n,m \to \infty$.
\end{Corollary}

\section{Comparison with other possible depth notions}
\label{section:comparison}
In the field of nonparametric statistics, extending the concept of depth to non-Euclidean spaces remains a challenging task. It has garnered considerable attention in advanced statistical research during the past decade \citep{LopezTukeyMetric, Geenens2023, virta2023spatial}. 
Within the confines of linear functional spaces, such as Banach or Hilbert spaces, the application of Euclidean methodologies remains largely successful, attributed primarily to their inherent vectorial structures. Contrastingly, the landscape becomes markedly more complex when venturing into the domain of infinite-dimensional spaces devoid of a vectorial framework.

The statistical literature identifies a mere trio of propositions capable of addressing this complexity: lens depth \cite{Geenens2023}, Tukey depth \citep{LopezTukeyMetric}, and a novel approach of metric spatial depth \citep{virta2023spatial}, different from our proposal. Here we delve into a meticulous exploration of these methodologies, with a particular emphasis on their adaptability to Wasserstein space framework. 

We shall demonstrate that these methodologies do not possess all the favorable theoretical and computational properties that we have established for the WSD. The WSD is thus most beneficial in broad, complex statistical contexts, thereby yielding a significant advancement in the field of machine learning and statistical analysis. Later in Section \ref{subsection:vs:general_metric_depth}, the simulation results also empirically support the above point of view.

\subsection{Tukey depth}
We first examine the adaptation of metric Tukey (or halfspace) depth proposed in \cite{LopezTukeyMetric} to the Wasserstein space. 
\begin{Definition} [Adapted from \citep{LopezTukeyMetric}]
	The Wasserstein halfspace depth of a distribution $Q\in \mathcal{P}_2(\R^d)$ with respect to a probability measure over Wasserstein space $\mathbf{P} \in \mathcal{P}(\mathcal{P}_2(\R^d))$ is the value 
	$$ \mathrm{HSD}(Q; {\bf P})=\inf_{ \substack{P_1,P_2\in \mathcal{P}_2(\R^d) \\ \mathcal{W}_2(Q,P_1) \leq \mathcal{W}_2(Q,P_2)  }}\mathbb{P}_{P\sim \mathbf{P}} \Big\{ \mathcal{W}_2(P,P_1) \leq \mathcal{W}_2(P,P_2) \Big\}.  $$
\end{Definition}
According to \citep{LopezTukeyMetric}, the Wasserstein halphspace depth is transformation invariant and vanishes at infinity.  Moreover, center-outward monotonicity (the function   $t\mapsto \mathrm{HSD}(\gamma(t); {\bf P})$ is monotone decreasing for any geodesic $\gamma(t)$ with $\mathrm{HSD}(\gamma(0); {\bf P})=1/2$) holds if for any constant speed geodesic $\gamma$ of $(\mathcal{P}_2(\R^d), \mathcal{W}_2)$, the following geometric condition holds: 
\begin{multline}\label{eq:centerTukey}
	\text{ there exists $t\in [0,1] $ such that  }\mathcal{W}_2^2(\gamma(t), P)\leq 
	\mathcal{W}_2^2(\gamma(t), Q) \\
	\implies \left(\mathcal{W}_2^2(\gamma(0), P)\leq 
	\mathcal{W}_2^2(\gamma(0), Q) \right)\ {\rm or}\ \left(\mathcal{W}_2^2(\gamma(1), P)\leq 
	\mathcal{W}_2^2(\gamma(1), Q) \right)  . 
\end{multline}
Recall (Section~\ref{subsec:Geodesic}) that a constant speed geodesic in a metric space $(\mathcal{M},d)$ is a curve $\gamma : [0, 1] \to \mathcal{M}$  such that $d(\gamma(t), \gamma(s))=|t-s| d(\gamma(0), \gamma(1))$ for all $s, t\in [0,1]$.  In $(\mathcal{P}_2(\R^d), \mathcal{W}_2)$, a constant speed geodesic corresponds to interpolations obtained from optimal transport plans \cite[p. 158]{AmbrosioGradient}. More precisely, any constant speed geodesic connecting two absolutely continuous distributions $P_1$ and $P_2$ is of the form $ \gamma(t)=((1-t)I + tT_{P_1,P_2})_\# P_1 $, where $T_{P_1,P_2}$ is the unique optimal map pushing $P_1$ to $P_2$ (see Section~\ref{subsec:WassersteinGeom}).  

Center-outward monotonicity is widely regarded as a favorable attribute within the scope of statistical depth measures. However, it is an attribute not typically anticipated in the context of spatial depths, particularly within Euclidean spaces. Notably, neither the transport-based depth nor the lens depth exhibit this property.  Moreover, for the Wasserstein halphspace depth it is not clear if the geometric condition \eqref{eq:centerTukey}, and {\it a fortiori} the center-outward monotonicity, hold in general.

Despite the ostensibly advantageous traits of Tukey depths, they are encumbered by significant computational demands, particularly evident as the dimensionality of the data increases. This computational intensity escalates to the point of impracticality for exact calculations in dimensions exceeding five, already in the Euclidean case. Within the confines of Wasserstein spaces, which are characterized by infinite dimensions, approximating Tukey depths poses a substantial challenge, much more so than for the WSD. 

\subsection{Lens depth}
Let us now turn our attention to the adaptation of the metric lens depth, presented in \cite{Geenens2023}, to the Wasserstein space.
\begin{Definition}[Adapted from \cite{Geenens2023}]
	The Wasserstein lens depth of a distribution $Q\in \mathcal{P}_2(\R^d)$ with respect to  ${\bf P}\in \mathcal{P}(\mathcal{P}_2(\R^d)) $ is defined as 
	$$ \mathrm{LD}(Q;{\bf P} )=\mathbb{P}_{(P',P)\sim {\bf P} \otimes {\bf P} }\Big\{ \mathcal{W}_2( P,P')\geq \max\big(  \mathcal{W}_2( P,Q),  \mathcal{W}_2( P',Q) \big)  \Big\}.$$
\end{Definition}
The Wasserstein lens depth is transformation invariant and 
vanishes at infinity. The two-stage plug-in estimator of $\mathrm{LD}(Q;{\bf P} )$ can be computed exactly for a discrete distribution $Q$ within polynomial time. Nevertheless, as indicated in \cite{Geenens2023}, the lens depth fails to exhibit center-outward monotonicity in the linear case. Similarly, this property would be absent in Wasserstein spaces. Moreover, to compute the empirical $$\mathrm{LD}(Q;{\bf P}_{n,m}),$$ one needs two copies of \textit{i.i.d.} samples, $\mathbf{P}_{n,m}$ and $\mathbf{P}'_{n,m}$, which increases the sample size requirement.  

\subsection{Metric spatial Wasserstein depth}
A recent paper \cite{virta2023spatial} gave a definition of spatial depth for general metric spaces that does not agree with our definition in the particular case of Wasserstein space.  
To avoid confusion in terminology, the proposal from \cite{virta2023spatial} will be referred to as metric spatial Wasserstein depth.
\begin{Definition}[Adapted from \cite{virta2023spatial}]
	The metric spatial Wasserstein depth of $Q\in \mathcal{P}_2(\R^d)$  with respect to ${\bf P}\in \mathcal{P}(\mathcal{P}_2(\R^d))$ is defined as
	$$ \mathrm{MSD}(Q;{\bf P} )= 1-\frac{1}{2}\E_{(P',P)\sim {\bf P} \otimes {\bf P} }\left[\frac{\mathcal{W}_2^2( P,Q)+\mathcal{W}_2^2( P',Q)-\mathcal{W}_2^2( P,P')}{\mathcal{W}_2( Q,P)\mathcal{W}_2( Q,P')} \right]. $$
\end{Definition}
The function $Q\mapsto \mathrm{MSD}(Q;{\bf P} )$ takes values in the interval $[0,2]$.  It is transformation invariant and vanishing at infinity. The metric spatial depth presents a remarkably viable and effective solution that is widely applicable to general metric spaces. Nevertheless, when specialized to the Wasserstein space, it falls short of fulfilling all the desirable properties that we have established for the WSD.
In particular, also pointed out in \cite{virta2023spatial}, the question of the inclusion of spatial medians within the set of deepest points in terms of the metric spatial depth remains overall open.
Note that taking a directional derivative of $\mathrm{MSD}(Q;{\bf P} )$ with respect to $Q$, in the aim of studying deepest points, does not seem particularly fruitful.
This leads us to conjecture that, in general, spatial medians have no relation to the maximizers of $\mathrm{MSD}(Q;{\bf P} )$.
In contrast, our Theorem~\ref{Theo:RelationSpatial} establishes that spatial medians maximize the WSD, in more general situations.

In the following result we show that $\mathrm{MSD}(Q;{\bf P} )$ is upper bounded by $1-(1- {\rm WSD}(Q;{\bf P} ))^2$ so that   its maximum  value is one. Moreover, we show that $\mathrm{MSD}(Q;{\bf P} )= 1-(1- {\rm WSD}(Q;{\bf P} ))^2$ if and only if  $T_{P',P}=T_{Q,P} \circ T_{P',Q}$ for ${\bf P}\otimes {\bf P}$-a.e. $P,P'$. In particular, we can show that the same bound applies to any non-negatively curved manifold (defining the metric space in \cite{virta2023spatial}).   Hence, the metric spatial depth  in a strictly  positively curved manifold is strictly smaller than one except if the dataset lies in a geodesic.   


\begin{Lemma}\label{lemma:RelationDephts}
	For every ${\bf P}\in \mathcal{P}(\mathcal{P}_2(\R^d)) $ and $Q\in \Ptac(\R^d)$ it follows that 
	$$  {\rm MSD}(Q;{\bf P} ) \leq 1-(1- {\rm WSD}(Q;{\bf P} ))^2$$
	with equality if and only if 
	$ \mathcal{W}_2(P,P')=\|T_{Q,P}-T_{Q,P'}\|_{L^2(Q)} $ for all $P,P'\in {\rm supp}({\bf P})$.  
\end{Lemma}

Finally, beyond the Wasserstein space, a natural question remaining overall open in \cite{virta2023spatial} is whether the maximal possible value $2$ for the metric spatial depth can be reached, for some pairs of metric spaces and distributions on these spaces.
In particular Theorem 3 there 
provides the existence of such pairs. However, the conditions of this theorem
are very strict, as noted in \cite{virta2023spatial}, and only satisfied in arguably pathological metric spaces. Lemma \ref{lemma:RelationDephts} shows that for the Wasserstein space, it is not even possible for the metric depth to be strictly larger than one.

\section{Nonparametric two-sample tests based on Wasserstein spatial depth}\label{section:testing}
A natural application of WSD is to construct nonparametric testing procedures. 
Consider two distribution-valued random objects, $P\sim \mathbf{P}$ and $Q\sim \mathbf{Q}$ with $\mathbf{P}, \mathbf{Q} \in \mathcal{P}(\mathcal{P}_{2}(\R^d))$. Given two sample sets
\begin{equation*}
	\{P_1, P_2, \ldots, P_n\} \stackrel{\textit{i.i.d.}}{\sim} \mathbf{P} \text{ and } \{Q_1, Q_2, \ldots, Q_n\} \stackrel{\textit{i.i.d.}}{\sim} \mathbf{Q}, 
\end{equation*}
or, as in Section \ref{two_stg_sampling}, their empirical versions 
\begin{equation*}
	\mathbf{P}_{n,m} = 
	\frac{1}{n}
	\sum_{i=1}^n 
	\delta_{P_{i,m}}
	~ ~ ~
	\text{ and }
	~ ~ ~ 
	\mathbf{Q}_{n,m} =  \frac{1}{n}
	\sum_{i=1}^n 
	\delta_{Q_{i,m}},
\end{equation*}
we want to test whether they are from the same population, \textit{i.e.}, 
\begin{equation}\label{NPtest}
	H_0: \mathbf{P}=\mathbf{Q}, \qquad H_1: \mathbf{P}\neq\mathbf{Q}.
\end{equation}
This type of hypothesis testing is particularly useful for data in the form of multiple batches or ``bags'', which is prevalent in modern sciences.  
For examples, a same type of biomedical data is collected from different laboratories with one data batch from each laboratory; one animal species has several subpopulations and researchers collect one data batch from each subpopulation; some spatial-temporal data can be treated as batches in sequence where each batch contains data points across the spatial domain; multi-batch training/processing has emerged as a powerful paradigm in machine learning. 
In the above cases, batch information is also important because the distributional shifts among the data batches somehow reveal the underlying mechanism of interest. 

The two-stage sampling setting described in \eqref{eq:dataTensor} adapts perfectly to such multi-batch data structure: each batch of data points $P_{i,m}$ ($Q_{k, m}$) are \textit{i.i.d.} from $P_i$ ($Q_k$), and 
\begin{align*}
	P_i \stackrel{\textit{i.i.d.}}{\sim} \mathbf{P},\quad Q_k \stackrel{\textit{i.i.d.}}{\sim} \mathbf{Q},
\end{align*}
where $\mathbf{P}$ and $\mathbf{Q}$ are two populations of distributions. 
This model allows distributional shifts among the data batches, manifests the hierarchically varying data structure, and thus offers a flexible framework to different types of real data. 
Compared to traditional two-sample tests which pool all the data batches into one large sample set, the statistical test in \eqref{NPtest} incorporates the batch information, and better captures the distributional variation along the geodesics. As a toy example, suppose that we have two sample sets of data batches: the first set contains two batches with the first batch \textit{i.i.d.} drawn from $\text{Unif}[0,1]$ and the second batch \textit{i.i.d.} drawn from $\text{Unif}[1,2]$; the other set contains also two batches both of which are \textit{i.i.d.} drawn from $\text{Unif}[0,2]$. 
Throughout the paper, for a set $A$, $\text{Unif} A$ denotes the uniform distribution on $A$. 
Clearly, these two sample sets are not from the same population. However, we will probably draw a wrong conclusion if we disregard the batch information and mix the data points from different batches followed by traditional two-sample tests.

In view of the above, we propose the following nonparametric permutation test procedure for two-sample test based on WSD, which is inspired by the Liu-Singh depth-based nonparametric test \cite{liu1993quality} and by permutation tests \cite{pitman1937significance}. 
\begin{itemize}
	\item Step 1: Compute $\rm{WSD}(P_{i,m}; \widetilde{\mathbf{P}}_{n,m})$ for $1\leq i\leq n$ and $\rm{WSD}(Q_{k,m}; \widetilde{\mathbf{P}}_{n,m})$ for $1\leq k\leq n$, where $\widetilde{\mathbf{P}}_{n,m}$ is an independent copy of $\mathbf{P}_{n,m}$ and independent of $\mathbf{Q}_{n,m}$. 
	In practice, we would split a sample from $\mathbf{P}$
	as in \eqref{eq:dataTensor}
	with $2n$ rows into $\widetilde{\mathbf{P}}_{n,m}$ and $\mathbf{P}_{n,m}$. 
	For simpler notation, we denote 
	\begin{equation*}
		W^P_i = \rm{WSD}(P_{i,m}; \widetilde{\mathbf{P}}_{n,m}), \quad W^Q_k = \rm{WSD}(Q_{k,m}; \widetilde{\mathbf{P}}_{n,m}).
	\end{equation*}
	
	\item Step 2: Compute the observed Kolmogorov–Smirnov (KS) test statistic, 
	\begin{equation*}
		T_{\text{obs}} = \sup_{t \in [0,1]} | F_{n,m}(t) -G_{n,m}(t) |,
	\end{equation*}
	where     
	\begin{align*}
		F_{n,m}(t) = \frac{1}{n}\sum_{i=1}^n \mathbf{1}\{ W^P_i \leq t \}, \ G_{n,m}(t) = \frac{1}{n}\sum_{k=1}^n \mathbf{1}\{ W^Q_k \leq t \}.
	\end{align*}
	\item Step 3: Perform independent random permutations for $B\in \N$ times: for each $1\leq b\leq B$, randomly draw $n$ elements
	without replacement
	from the pooled WSDs,
	\begin{equation*}
		\{ W^P_i\}_{i=1}^n \cup  \{ W^Q_k \}_{k=1}^n 
	\end{equation*}
	as the new ``$\{ W^P_i \}_i$'' and the remaining $n$ elements as the new ``$\{ W^Q_k \}_k$''; compute the same statistic (KS statistic) as above, denoted as $T_b$;   
	$B$ independent repetitions generate $\{ T_b \}_{1\leq b\leq B}$. The permutations are independent of the samples.
	\item Step 4: If there are no ties in $ \{T_{\text{obs}}\} \cup \{ T_b \}_{1\leq b\leq B}$, compute the p-value via
	\begin{equation}\label{p-value}
		p = \frac{ 1+ \#\{b: T_b\geq T_{\text{obs}} \} }{ 1+B }.
	\end{equation}
	If there are ties in  $ \{T_{\text{obs}}\} \cup \{ T_b \}_{1\leq b\leq B}$, we break ties by randomization while preserving their relative orderings. For a group of ties $T_{b_1} = \cdots = T_{b_R} =a$ (letting $T_0 = T_{\text{obs}}$ by convention), we slightly change their values via
	\begin{equation*}
		\widetilde{T}_{b_r} = T_{b_r} + \frac{g}{4}\cdot U_{b_r}, \quad U_{b_r}\sim \rm{Unif}[-1,1],
	\end{equation*}
	where $g$ is the smallest gap between $a$ and any other values in $\{ T_b \}_{0\leq b\leq B}$. Then, we use the randomized test statistics $\widetilde{T}_0,\dots,\widetilde{T}_B$ to compute the p-value in \eqref{p-value}.
	\item Step 5: reject the null if $p < \alpha$ for some pre-determined significance level $\alpha$. 
\end{itemize}

Note that an alternative to the above permutation-based procedure is to break ties before computing $T_{
	\text{obs}}$. In this case, it is not necessary to perform the permutation steps 3 and 4. Instead, one can simply reject if $T_{
	\text{obs}}$ is above its quantile $1 - \alpha$ under the null hypothesis. This quantile does not depend on the underlying distribution of $\rm{WSD}(P_{i,m}; \widetilde{\mathbf{P}}_{n,m})$, since the KS test is distribution-free (in the absence of ties) \citep{massey1951distribution}. The benefit of the above permutation-based procedure is
not only that we do not need to break ties when computing $T_{
	\text{obs}}$, but also that 
any statistic can be chosen to compute $T_{\text{obs}}$, including those that are not distribution-free under the null hypothesis.
An alternative test statistic can be
\begin{equation*}
	T^{\text{alt}}_{\text{obs}} = \frac{1}{n}\sum_{i=1}^n \rm{WSD}(P_{i,m}; \widetilde{\mathbf{P}}_{n,m}) - \frac{1}{n}\sum_{k=1}^n \rm{WSD}(Q_{k,m}; \widetilde{\mathbf{P}}_{n,m}).
\end{equation*}
One may choose a test statistic that adapts to specific distributions of $\big\{ \rm{WSD}(P_{i,m}; \widetilde{\mathbf{P}}_{n,m}) \big\}_i$ and $\big\{ \rm{WSD}(Q_{k,m}; \widetilde{\mathbf{P}}_{n,m}) \big\}_k$.

The two propositions below provide theoretical guarantees on the Type I error and power of the proposed testing procedure. First we show that it achieves exact finite-sample control of the Type I error.

\begin{Proposition}
	Under $H_0:{\bf P}={\bf Q}$, the p-value computed in Step 4  follows a uniform distribution on $\{ \frac{1}{B+1}, \frac{2}{B+1}, \ldots, \frac{B}{B+1}, 1 \}$ for any fixed $(n, m, B)$. 
	Additionally, when $B\rightarrow \infty$, 
	\begin{equation*}
		p \stackrel{w}{\longrightarrow} \rm{Unif}[0, 1].
	\end{equation*} 
	\label{prop:level:two:sample}
\end{Proposition}

Next, Proposition \ref{testing_power} shows that the power of the testing procedure goes to one asymptotically.
There, we address one-stage sampling for simplicity. 
\begin{Proposition}\label{testing_power}
	Assume $\mathbf{P}, \mathbf{Q} \in \mathcal{P}(\mathcal{P}_2^{a.c}(\R^d))$. 
	Let $P \sim \mathbf{P}$ and $Q \sim \mathbf{Q}$.
	Let $F$ be the 
	cumulative distribution function
	(CDF) of the distribution $\rm{WSD}(P; \mathbf{P})$; let $G$ be the CDF of the distribution $\rm{WSD}(Q; \mathbf{P})$. 
	Consider the proposed testing procedure adapted to the one-stage sampling, with $(P_{i,m},Q_{k,m},\widetilde{\mathbf{P}}_{n,m})$ replaced by 
	$(P_{i},Q_{k},\widetilde{\mathbf{P}}_{n})$,
	with $\widetilde{\mathbf{P}}_{n}$ an independent copy of $\mathbf{P}_{n}$.
	If $F$ and $G$ are continuous and different, then
	the power of the proposed testing procedure goes to one when $n,B\rightarrow\infty$, that is $ p \to 0$ in probability.
\end{Proposition}

Note that in Proposition \ref{testing_power}, we assume that the two WSD distributions are different, which is stronger than assuming that $\mathbf{P}$ and $ \mathbf{Q} $ are different. 
Nevertheless, this type of assumption is frequent in the two-sample testing literature to establish power properties. In particular, a similar assumption is made in \cite{dubey2024metric} where the power is established under the assumption that distance profiles are different between the two distributions of the two samples.

The proofs of Propositions \ref{prop:level:two:sample} and \ref{testing_power} use standard arguments for permutation-based tests,
combined with Theorem \ref{Theorem:Glivenk},
but we provide all the details for completeness.

\section{Numerical simulations}\label{section:simulation}
In this section, we carry out extensive numerical simulations to validate our notion of Wasserstein spatial depth and support its theoretical properties and practical utility. Specifically, we confirm the consistency of the empirical WSD, examine its relationship with conventional spatial depth in certain cases, evaluate its effectiveness in outlier detection and show its benefit compared to applying functional depths 
and general metric depths
to distributions. 
We also show the merits of the two-sample test based on the WSD. 
Throughout this section, we use the \texttt{R} package \texttt{transport} to compute all the Wasserstein distances and optimal transport maps from data clouds. Based on the two-stage sampling model in \eqref{eq:dataTensor}, the empirical WSDs are calculated via the formula below. For $Q_m = \frac{1}{m}\sum_{j=1}^m \delta_{\X_{q,j}}$ with $\X_{q,1} , \ldots , \X_{q,m} \stackrel{\textit{i.i.d.
}}{\sim} Q$,
\begin{equation}\label{empirical_WSD_1}
	{\rm WSD}(Q_m; {\bf P}_{n,m}) = 1- \sqrt{ \frac{1}{m} \sum_{j=1}^m  \bigg\| \frac{1}{n_q}\sum_{i\neq q} \frac{\X_{q,j}-T_{Q_m, P_{i,m}}(\X_{q,j})}{\mathcal{W}_2(Q_m, P_{i,m})} \bigg\|^2 },
\end{equation}
where $Q_m$ could be outside (with the convention $q = n+1$ and $n_q=n$) or within (with the convention $q \in \{1,\ldots,n\}$ and $n_q= n-1$) the sampled distributions $\{ P_{1,m}, \ldots, P_{n,m} \}$, and where we recall the convention $\mathbf{0}/0 = \mathbf{0}$. 
The code for all simulations is publicly available at \url{https://github.com/YishaYao/Wasserstein-Spatial-Depth/tree/main}.

Since computing the optimal transport map between any pair of empirical distributions costs $O(m^2)$ \cite{peyre2019computational}, and once the optimal transport map between a pair of empirical distributions is available, the corresponding Wasserstein distance immediately follows with almost zero extra cost,
the computational complexity of $ {\rm WSD}(Q_m, \mathbf{P}_{n,m})$ is of order $O(nm^2)$.

\subsection{Consistency of the empirical Wasserstein spatial depth}\label{simulation_consistency}
The simulation results below support the theoretical results in Section \ref{consistency}. That is, the empirical WSD, formulated in \eqref{empirical_WSD_1}, is close to the theoretical value $ {\rm WSD}(Q; {\bf P})$ in Definition \ref{Wdepth_definition}. Hence, the WSD can be inferred accurately from sample data and has practical value. 
Four cases are considered and described below.
\begin{itemize}
	\item Case 1: ${\bf P}$ is supported on a family of exponential distributions indexed by the rate parameter $\lambda$ which follows a $\mathrm{Beta}(2,2)$ distribution. The theoretical WSD of the exponential distribution with rate parameter $\lambda_Q \in (0,1]$, denoted as $\mathrm{exp}(\lambda_Q)$, with respect to ${\bf P}$ is
	\begin{align*}
		{\rm WSD}(Q; {\bf P}) &= 1- \sqrt{ \int_0^\infty \Big( \mathbb{E}_{\lambda\sim \mathrm{Beta}(2,2)} \frac{x-(\lambda_Q/\lambda)x}{\mathcal{W}_2(F_{\lambda_Q}, F_\lambda)}  \Big)^2 \lambda_Q e^{-\lambda_Q x} \mathrm{d} x}\\
		&= 1- \sqrt{ \int_0^\infty \frac{\lambda_Q^2 x^2}{2} \Big( \mathbb{E}_{\lambda\sim \mathrm{Beta}(2,2)} \frac{1/\lambda_Q - 1/\lambda}{|1/\lambda_Q - 1/\lambda|} \Big)^2 \lambda_Q e^{-\lambda_Q x} \mathrm{d} x}\\
		&= 1- \sqrt{ \int_0^\infty \frac{\lambda_Q^2}{2} \Big( 4\lambda_Q^3 -6\lambda_Q^2 +1 \Big)^2 x^2 \lambda_Q e^{-\lambda_Q x} \mathrm{d}x}\\
		&= 1- \Big| 1+4\lambda_Q^3 - 6\lambda_Q^2 \Big| ,
	\end{align*}
	where $F_\lambda$ is the CDF of the exponential distribution with rate parameter $\lambda$, the optimal map from $\mathrm{exp}(\lambda_Q)$ to $\mathrm{exp}(\lambda)$ is 
	\begin{equation*}
		T_{\lambda_Q, \lambda} (x) = F_\lambda^{-1} \circ F_{\lambda_Q}(x) = \frac{\lambda_Q x}{\lambda} ,
	\end{equation*}
	and $\mathcal{W}_2(F_{\lambda_Q}, F_\lambda)$ is derived by
	\begin{equation*}
		\mathcal{W}_2(F_{\lambda_Q}, F_\lambda) = \sqrt{\int_0^\infty \Big(x-(\lambda_Q/\lambda)x \Big)^2 \lambda_Q e^{-\lambda_Q x} \mathrm{d}x} = \sqrt{2}\Big| \frac{1}{\lambda_Q}-\frac{1}{\lambda} \Big| .
	\end{equation*}
	Note that the WSD is equal to $1$ (maximal) for $\lambda_Q = 1/2$ which is the mean of the $\mathrm{Beta}(2,2)$ distribution.
	
	\item Case 2: ${\bf P}$ is supported on a family of Weibull  distributions with fixed scale parameter $1$ and varying shape parameter $k$. This family of Weibull distributions is indexed by the shape parameter $k$ which takes value either $1$ or $2$ with equal probabilities, i.e., $k\sim \mathrm{Unif}\{1,2\}$. Let $Q$ be the Weibull distribution with shape parameter $k_Q$ ($k_Q$ equating either $1$ or $2$). Its theoretical WSD with respect to ${\bf P}$ is
	\begin{align*}
		{\rm WSD}(Q; {\bf P}) &= 1- \sqrt{ \int_0^\infty \Big( \mathbb{E}_{k\sim \mathrm{Unif}\{1,2\} } \frac{x-x^{k_Q/k}}{\mathcal{W}_2(k_Q, k)}  \Big)^2 k_Q x^{k_Q-1} e^{-x^{k_Q}} \mathrm{d}x}\\
		&= 1- \sqrt{ \int_0^\infty \Big( \frac{x-x^{k_Q/\overline{k}_{Q}}}{2\mathcal{W}_2(k_Q, \overline{k}_{Q})} \Big)^2  k_Q x^{k_Q-1} e^{-x^{k_Q}} \mathrm{d}x}\\
		&= 1- \frac{1}{2\mathcal{W}_2(k_Q, \overline{k}_{Q})} \sqrt{\int_0^\infty \big(x-x^{k_Q/\overline{k}_{Q}} \big)^2 k_Q x^{k_Q-1} e^{-x^{k_Q}} \mathrm{d}x} \\
		&= 1/2,
	\end{align*}
	where the optimal map from $\mathrm{Weibull}(k_Q)$ to $\mathrm{Weibull}(k)$ is 
	\begin{equation*}
		T_{k_Q, k} (x) = F_k^{-1} \circ F_{k_Q}(x) = x^{k_Q/k},
	\end{equation*}
	using $\overline{k}_{Q} = 3 - k_Q$, the convention $0/0 = 0$,
	and where $\mathcal{W}_2(k_Q, \overline{k}_{Q})$ is derived by
	\begin{align*}
		\mathcal{W}_2(k_Q, \overline{k}_{Q}) &= \sqrt{\int_0^\infty \big(x-x^{k_Q/\overline{k}_{Q}} \big)^2 k_Q x^{k_Q-1} e^{-x^{k_Q}} \mathrm{d}x} .
	\end{align*}
	
	\item Case 3: ${\bf P}$ is supported on a family of isotropic bivariate Gaussian distributions with varying centers. The distribution of the Gaussian centers is supported on four points $\{ \boldsymbol{\mu}_1=(1,0)^\top, \boldsymbol{\mu}_2=(-1,0)^\top, \boldsymbol{\mu}_3=(0,1)^\top, \boldsymbol{\mu}_4=(0,-1)^\top \}$ with equal probabilities $1/4$. Let $Q$ be  $\mathcal{N}(\boldsymbol{\mu}_q, \boldsymbol{I})$ for $q \in \{1, \ldots, 4\}$. The theoretical WSD is computed as, see Section~\ref{gaussian_exp}, 
	\begin{align*}
		{\rm WSD}(Q; {\bf P}) &= 1- \Bigg\| \frac{1}{4}\sum_{k\neq q} \frac{\boldsymbol{\mu}_q-\boldsymbol{\mu}_k}{\|\boldsymbol{\mu}_q-\boldsymbol{\mu}_k\|} \Bigg\| = \frac{3-\sqrt{2}}{4} .
	\end{align*}
	
	\item Case 4: ${\bf P}$ is supported on a family of bivariate uniform distributions $\mathrm{Unif} [0, c]^2 $ with $c\sim \mathrm{Unif}[1,2]$. Let $Q$ be $\mathrm{Unif} [0, c_q]^2 $.  Its theoretical WSD with respect to ${\bf P}$ is
	\begin{align*}
		{\rm WSD}(Q; {\bf P}) &= 1- \sqrt{ \int_{[0,c_q]^2} \bigg\| \mathbb{E}_{c\sim \mathrm{Unif}[1,2]}\ \frac{\mathbf{x}-(c/c_q)\mathbf{x}}{\mathcal{W}_2(c_q, c)}  \bigg\|^2 \frac{1}{c_q^2} \mathrm{d} \mathbf{x} } \\
		&= 1- \sqrt{ \int_{[0,c_q]^2} \bigg\| \sqrt{3/2} \mathbf{x} \mathbb{E}_{c\sim \mathrm{Unif}[1,2]} \frac{1-(c/c_q)}{|c_q-c|}  \bigg\|^2 \frac{1}{c_q^2} \mathrm{d} \mathbf{x} } \\
		&= 1- \sqrt{ 3(2-3/c_q)^2/2 \int_{[0,c_q]^2} \| \mathbf{x} \|^2 \frac{1}{c_q^2} \mathrm{d} \mathbf{x} } \\
		&= 1- |2c_q-3| ,
	\end{align*}
	where the optimal map from $\mathrm{Unif} [0, c_q]^2 $ to $\mathrm{Unif} [0, c]^2 $ is the dilation 
	\begin{equation*}
		T_{c_q, c}(\mathbf{x}) = \frac{c}{c_q}\mathbf{x},
	\end{equation*}
	and $\mathcal{W}_2 (c_q, c)$ is computed as 
	\begin{equation*}
		\mathcal{W}_2 (c_q, c) = \sqrt{ \int_{[0,c_q]^2} \|\mathbf{x}-(c/c_q)\mathbf{x}\|^2 (1/c_q^2) \mathrm{d} \mathbf{x} }= \sqrt{2/3} |c_q-c|.
	\end{equation*}
\end{itemize} 
For each of the above cases, we repeat independently the following experiment for $100$ times: generate the data array $\mathbf{X}$ via the two-stage sampling procedure in Section \ref{two_stg_sampling}; then compute the empirical WSDs of the sampled distributions; finally, compare the theoretical WSD and the ensemble of $100$ empirical WSDs.  We choose $m=1000,\ n=2000$.
As shown in Figure~\ref{consis}, the empirical estimates are gathering tightly around the corresponding theoretical values. 
\begin{figure}[!htbp]
	\centering
	\includegraphics[width=1\linewidth]{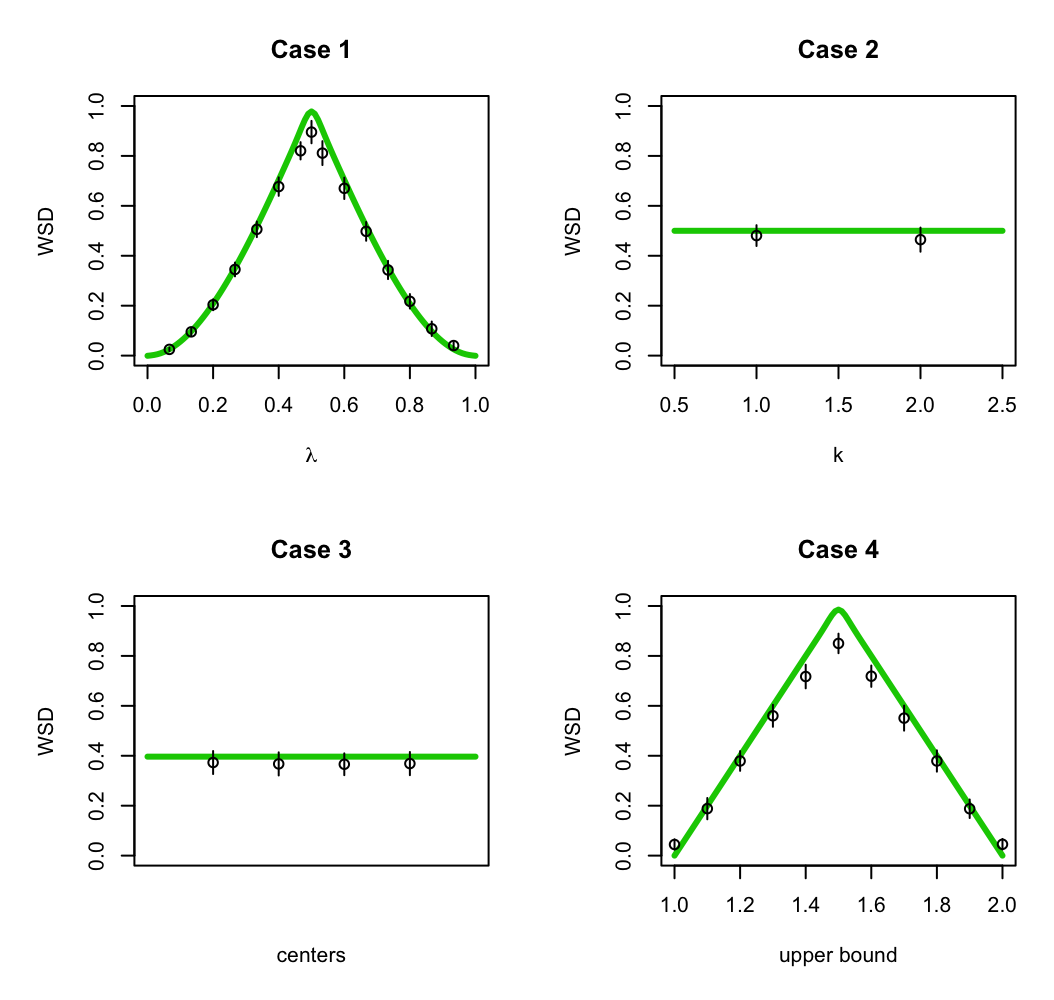}
	\caption{The green solid lines depict the change of theoretical WSD along the parameter indexing ${\bf P}$. The black circles represent the distribution of empirical WSDs, with error bars indicating one standard deviation above and below the mean.}
	\label{consis}
\end{figure}

\subsection{Wasserstein spatial depth vs. conventional spatial depth}
\label{subsection:WSd:vs:conv}
As discussed in Section \ref{section:examples}, when $\mathbf{P}$ is supported on a location family, 
the WSD coincides with the spatial depth of the location parameter. We verify the equivalence between WSD and spatial depth in the four cases described below.
\begin{itemize}
	\item Case 1: $\mathbf{P}$ is supported on a set of $d=10$-dimensional Gaussian distributions with identity covariance matrix. The Gaussian centers are \textit{i.i.d.} from $\mathrm{Unif}[-2,2]^d$. 
	\item Case 2: $\mathbf{P}$ is supported on a set of $d=10$-dimensional Gaussian distributions with a common covariance matrix. The common covariance matrix is chosen as $\Sigma_{i,j} = 0.2^{|i-j|}$. The Gaussian centers are drawn in the same way as in Case 1.
	\item Case 3: The support of $\mathbf{P}$ is a set of uniform distributions on $d=10$-dimensional unit cubes with varying centers. The centers of the cubes are identically independently drawn from $\mathcal{N}(\mathbf{0}, \boldsymbol{I})$.
	\item Case 4: $\mathbf{P}$ is supported on a set of univariate double exponential distributions with fixed rate equaling $1$ and varying locations. The locations are identically independently drawn from $\mathcal{N}(0,1)$.
\end{itemize}
The simulation procedure is as follows. First, $n=500$ distributions are drawn as described above in each case. Second, $m=500$ data points are randomly drawn for each sampled distribution. Third, the empirical WSD of each empirical distribution is computed according to \eqref{empirical_WSD_1}, and the empirical spatial depths of the locations are computed as in \eqref{spatial_depth}. Finally, we check whether the empirical WSDs and spatial depths are approximately equal. As shown in Figure \ref{WSD_Sdepth}, there are nice equality relationships between the two depths.
\begin{figure}[!htbp]
	\centering
	\includegraphics[width=1\textwidth]{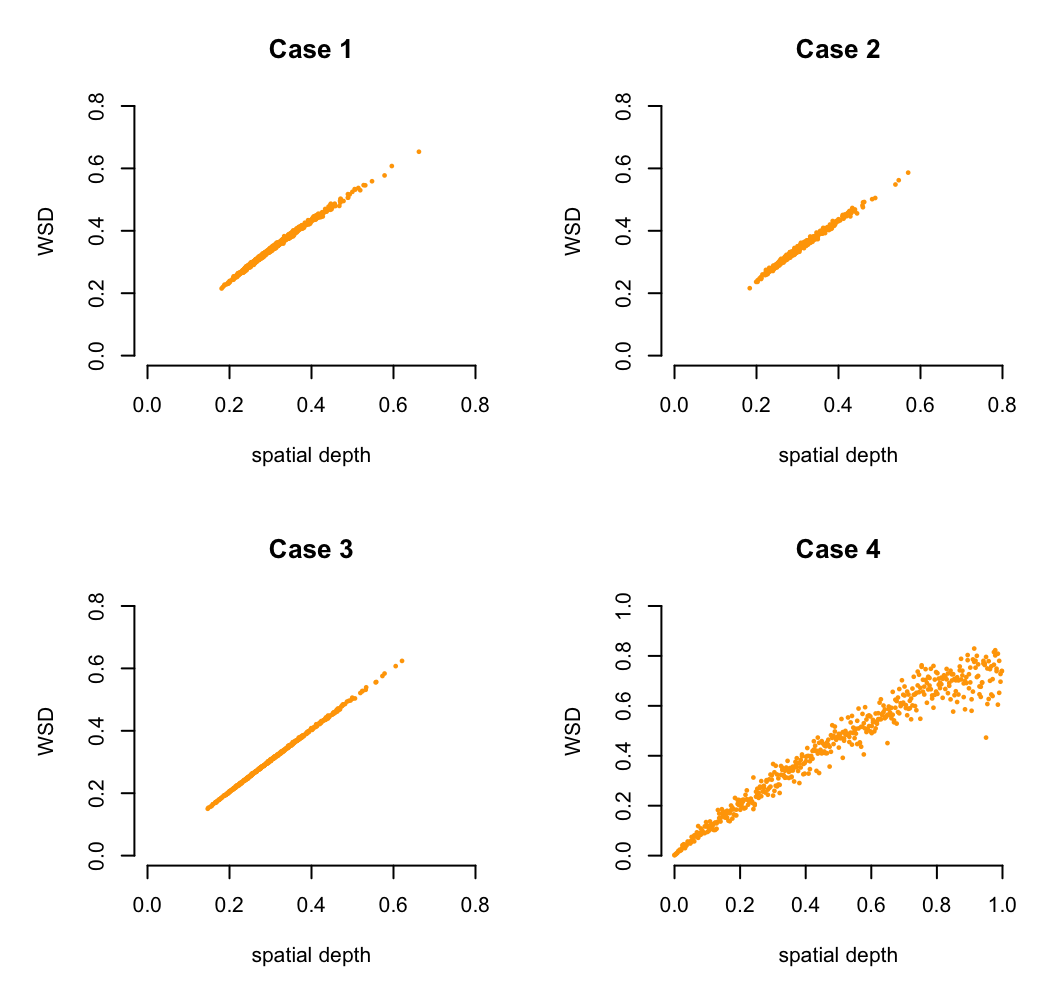}
	\caption{
		The relationships between the WSD and conventional spatial depth in the four cases of Section \ref{subsection:WSd:vs:conv}. 
	}
	\label{WSD_Sdepth}
\end{figure}

\subsection{Outlier detection}\label{subsection:outlier_detect}
Like conventional statistical depth, WSD can be used to detect outlier distributions. 
We demonstrate its utility for outlier detection in two cases. In each case, we draw $n=500$ distributions from a population ${\bf P}$ and six outlier distributions which are relatively far away from the population. For each sampled distribution, we draw $m=500$ data points.  All the distributions are on $\R^d$ with $d=10$. 
\begin{itemize}
	\item Case 1: the population is a collection of Gaussian distributions with common identity covariance matrix and random centers, where the centers follow \textit{i.i.d.} $\mathcal{N}(\mathbf{0}, \boldsymbol{I})$; 
	the six outlier distributions are 
	\begin{align*}
		&\mathcal{N}((4, \ldots, 4)^\top, \boldsymbol{I}),\quad \mathcal{N}((4, \ldots, 4)^\top, \bSigma)\ \text{with}\ \Sigma_{i,j}=0.5^{|i-j|},\\
		&\big[ \mathrm{Gamma}(3,2) \big]^d,\quad \big[ \mathrm{Unif}[-6, 6] \big]^d, \quad \big[ 8\cdot\mathrm{Beta}(0.1, 0.1)-4 \big]^d,\\
		&\mathrm{Multinomial}\big(2d, (0.25, 0.25, 0.15, 0.15, 0.15, 0.01, 0.01, 0.01, 0.01, 0.01)\big).
	\end{align*}
	Here for a distribution $\mu$, $[\mu]^d$ is the distribution such that for
	$\mathbf{Z} = (\mathbf{Z}_1 , \ldots , \mathbf{Z}_d) \sim [\mu]^d $ we have $\mathbf{Z}_1 , \ldots , \mathbf{Z}_d 
	\stackrel{\textit{i.i.d.}}{\sim} 
	\mu$.
	\item Case 2: the population is a collection of uniform distributions $\big[ \mathrm{Unif}[0, u] \big]^d$ with $u\sim \mathrm{Unif}[1,2]$; the outlier distributions are 
	\begin{align*}
		&\mathcal{N}\big( (3, \ldots, 3)^\top, \boldsymbol{I} \big), \quad \mathcal{N}\big( (-1, \ldots, -1)^\top, \bSigma \big)\ \text{with}\ \Sigma_{i,j}=0.5^{|i-j|},\\
		&\big[\mathrm{Poisson}(4)\big]^d, \quad \big[2\cdot\mathrm{Binomial}(d, 0.2)-1\big]^d, \quad \big[\chi^2_{10}\big]^d ,\\
		&\mathrm{Multinomial}\big(2d, (0.25, 0.15, 0.1, 0.1,\linebreak[1] 0.15, 0.05, 0.05, 0.05, 0.05, 0.05)\big).
	\end{align*}
\end{itemize}

For each case, we repeat the same experiment for $200$ times. In each replica: we draw the data array $\mathbf{X}$ according to Case 1 or Case 2; compute their empirical WSD according to \eqref{empirical_WSD_1}; detect the outlier distributions whose empirical WSDs are smaller than the $1\%$ quantile of all the empirical WSDs.
As shown in Table \ref{tab:outlier}, each outlier distribution in fact has abnormally small WSD values (compared to the population) in all $200$ replicas of the experiment. There are  gaps between the maxima of the WSDs of the outliers and the minima of the WSDs of the population.
Therefore, the outlier distributions can be easily separated from the population based on the magnitude of WSD. 
Figure \ref{outlier_fig} shows the result of one randomly chosen experiment. 

\begin{table}[h]
	\caption{Ranges of the empirical WSD values in $200$ replicas}
	\label{tab:outlier}
	\begin{tabular}{|l|c|c|}
		\hline
		& Case 1           & Case 2           \\ \hline
		population & (0.1413, 0.7126) & (0.2513, 0.7254) \\ \hline
		outlier 1  & (0.0452, 0.0489) & (0.0167, 0.0176) \\ \hline
		outlier 2  & (0.0451, 0.0492) & (0.0143, 0.0165) \\ \hline
		outlier 3  & (0.0171, 0.0190) & (0.0091, 0.0095) \\ \hline
		outlier 4  & (0.0836, 0.0910) & (0.0313, 0.1416) \\ \hline
		outlier 5  & (0.0862, 0.0932) & (0.0196, 0.0213) \\ \hline
		outlier 6  & (0.0771, 0.0830) & (0.0024, 0.0025) \\ \hline
	\end{tabular}
\end{table}

\begin{figure}[!htbp]
	\centering
	\includegraphics[width=0.95\textwidth]{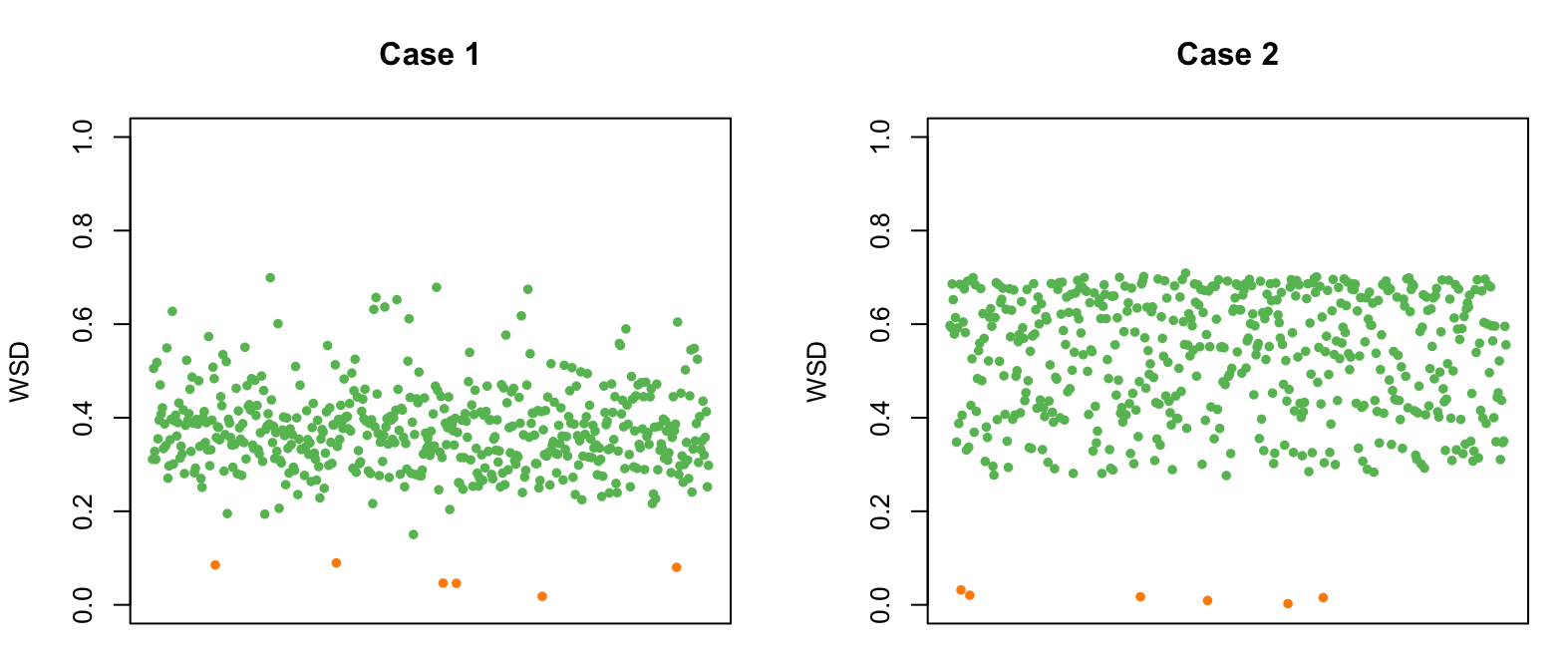}
	\caption{
		Left panel: the distributions are drawn according to Case 1. Right panel: the distributions are drawn according to Case 2. The green dots represent regular distributions from the population $\mathbf{P}$, and the orange dots represent the outlier distributions.
	}
	\label{outlier_fig}
\end{figure}

\subsection{Wasserstein spatial depth vs. functional depth}
\label{subsection:vs:functional}

Hilbertian embedding of probability measures (into a RKHS) is a powerful technique in machine learning and statistics  \citep{sriperumbudur2010hilbert}, which allows for a functional representation of probability measures. 
This approach maps a probability distribution $\mu \in \mathcal{P}(\R^d)$ to an element $f_{\mu}$ in the RKHS $\mathcal{H}_K$ via kernel mean embedding, 
\begin{equation*}
	f_{\mu}(t) = \int_{\R^d} K(x,t) d\mu(x).
\end{equation*}
Here $K : \R^d \times \R^d \to \R$ is a kernel on $\R^d$, yielding the Hilbert space $\mathcal{H}_K$ (the RKHS) of functions from $\R^d$ to $\R$ \cite{berlinet2011reproducing}. 
Given this embedding machinery and an available notion of depth for functional data \citep{fraiman2001trimmed, lopez2009concept}, one could first transform a distribution into a functional data point and then compute its functional depth. However, such an approach neglects the rich geodesic structure of the Wasserstein space. The relative distance and ``ordering'' of pairs of distributions are probably distorted after Hilbertian embedding. The simulation results in this subsection support the above point of view. 

We consider two cases here. In each case, $n=100$ similar distributions (denoted as regular distributions) and four exotic distributions are drawn. By ``similar'' we mean that these $n$ distributions are of the same parametric family and are close to each other in terms of Wasserstein distance. All the distributions are on $\R^3$ so that visualization is possible. We draw $m=300$ data points for each distribution. 
\begin{itemize}
	\item Case 1: the regular distributions are from a collection of spherical Gaussian distributions $\mathcal{N}(\boldsymbol{\mu}, \sigma^2\boldsymbol{I})$ with varying centers $\boldsymbol{\mu} \sim  \mathcal{N}(\boldsymbol{0}, \boldsymbol{I})$ and varying variances $\sigma
	\sim
	\mathrm{Unif}[0.8,1]$;   
	the four exotic distributions are 
	\begin{align*}
		&\big[ \mathrm{Gamma}(3,2) \big]^d,\quad \big[ \mathrm{Weibull}(2,1) \cdot 3\mathrm{Bernoulli}(-1,1,1/2)\big]^d,   \\
		&\big[ \mathrm{Unif}\{-3.5, -2.5, 2.5, 3.5\} \big]^d, \quad
		\mathcal{N}((-3, 3, -3)^\top, \bSigma)\ \text{with}\ \Sigma_{i,j}=0.5^{|i-j|} .
	\end{align*}
	
	\item Case 2: the regular distributions are from a collection of uniform distributions $\big[ \mathrm{Unif}[0, u] \big]^d$ with $u\sim \mathrm{Beta}(2,2) + 1$; the exotic distributions are 
	\begin{align*}
		&\big[\mathrm{Poisson}(1)\big]^d, \quad \big[\mathrm{Exponential}(2) \cdot \mathrm{Bernoulli}(-1,1,1/2)\big]^d, \\
		&\big[\mathrm{Unif}\{ 1,2,3 \}\big]^d, \quad \mathrm{Multinomial}\big(2d, (0.1, 0.2, 0.7)\big).
	\end{align*}
\end{itemize}
Here $\mathrm{Bernoulli}(-1,1,1/2)$ means an independent Bernoulli random variable taking value $-1$ or $+1$ with probability $1/2$.
In each case, the regular distributions are close to each other in the Wasserstein space because
\begin{align*}
	&\mathcal{W}_2\Big(N(\boldsymbol{\mu}_1, \sigma_1^2\boldsymbol{I}), N(\boldsymbol{\mu}_2, \sigma_2^2\boldsymbol{I}) \Big) =\sqrt{\|\boldsymbol{\mu}_1 - \boldsymbol{\mu}_2\|^2 + d(\sigma_1-\sigma_2)^2}
	\lesssim 1.5\sqrt{d},
\end{align*}
\begin{align*}
	\mathcal{W}_2\Big(\big[ \mathrm{Unif}[0, u_1] \big]^d, \big[ \mathrm{Unif}[0, u_2] \big]^d \Big) 
	&= \sqrt{\int_{[0,u_1]^d} \frac{\|\x- (u_2/u_1)\x\|^2}{u_1^d} \mathrm{d}\x} \\
	&= \sqrt{d/3} \ |u_2-u_1| \leq \sqrt{d/3} .
\end{align*}
Also shown in Figure \ref{WSD_Fdepth_1} (a) and  Figure \ref{WSD_Fdepth_2} (a), the regular distributions (represented by green triangles) tend to form a data cloud and are not visually distinguishable, while the exotic distributions are visually distant from the regular distributions. 

We compare the WSD with two types of functional depth, Modified Band Depth (MBD) \citep{lopez2009concept} and Functional Spatial Depth (FSD) \citep{Chakraborty2014spatial} in terms of detecting those exotic distributions. To compute the functional depth of a distribution, we first embed the distribution into a RKHS via a Gaussian kernel, and then compute the functional depth of the embedded function. The MBD and FSD are computed, respectively, by the \texttt{R} packages \texttt{depthTools} and \texttt{fda.usc}.  
As shown in Figures \ref{WSD_Fdepth_1} and \ref{WSD_Fdepth_2}, the WSD is able to discriminate exotic distributions in both cases, while the functional depths are not informative on the ``ordering'' of the distributions. 
\begin{figure}[!htbp]
	\centering
	\includegraphics[width=1\textwidth]{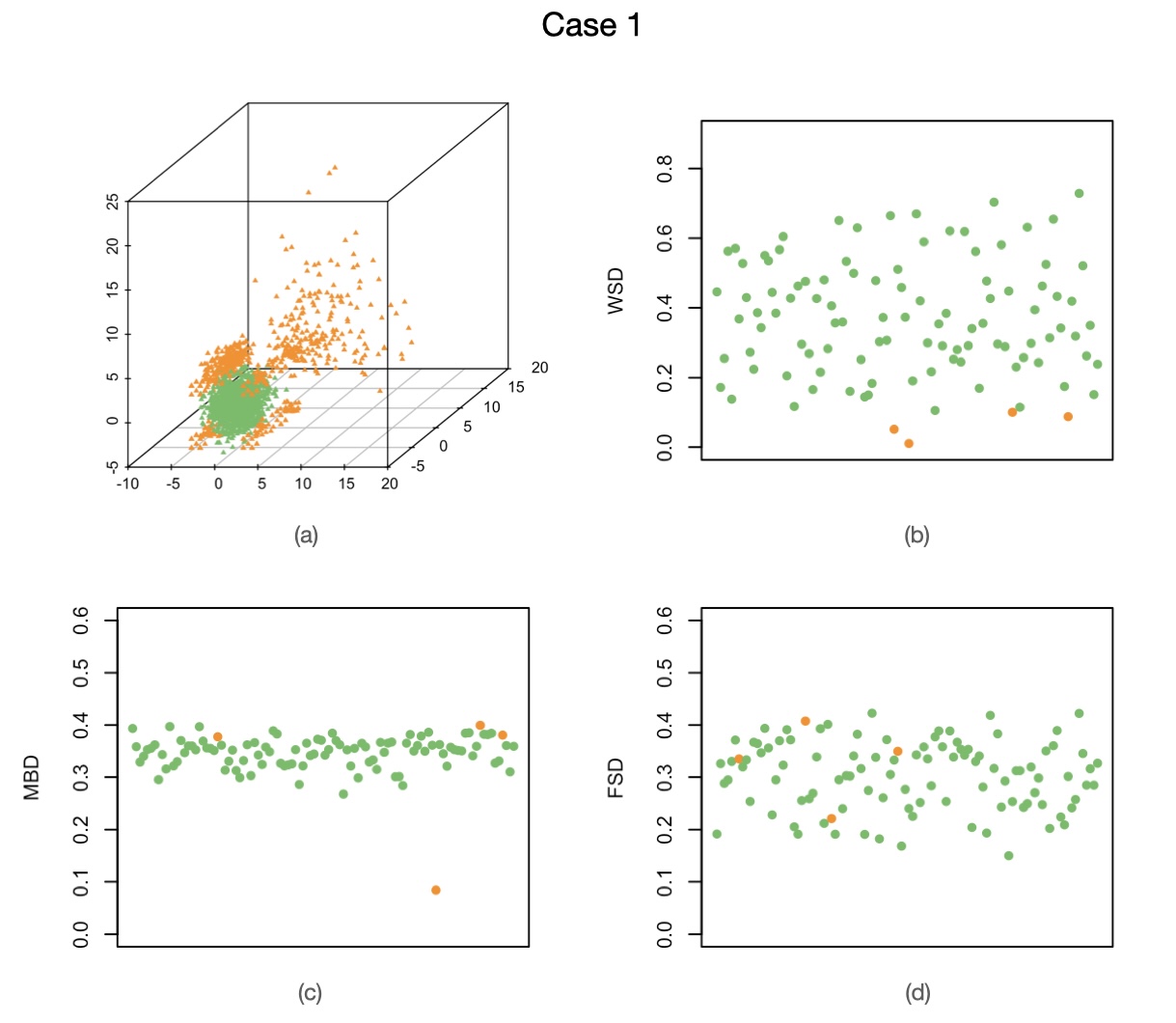}
	\caption{
		(a): The data points are drawn from the distributions of Case 1. The green triangles represent data points from the regular distributions, while orange triangles represent data points from the exotic distributions. (b): The green dots represent the WSD values of the regular distributions, while the orange dots represent the WSD values of the exotic distributions. (c): Each dot represents the MBD of a distribution. The coloring pattern is the same as before. (d):  Each dot represents the FSD of a distribution. The coloring pattern remains the same.
	}
	\label{WSD_Fdepth_1}
\end{figure}

\begin{figure}[!htbp]
	\centering
	\includegraphics[width=1\textwidth]{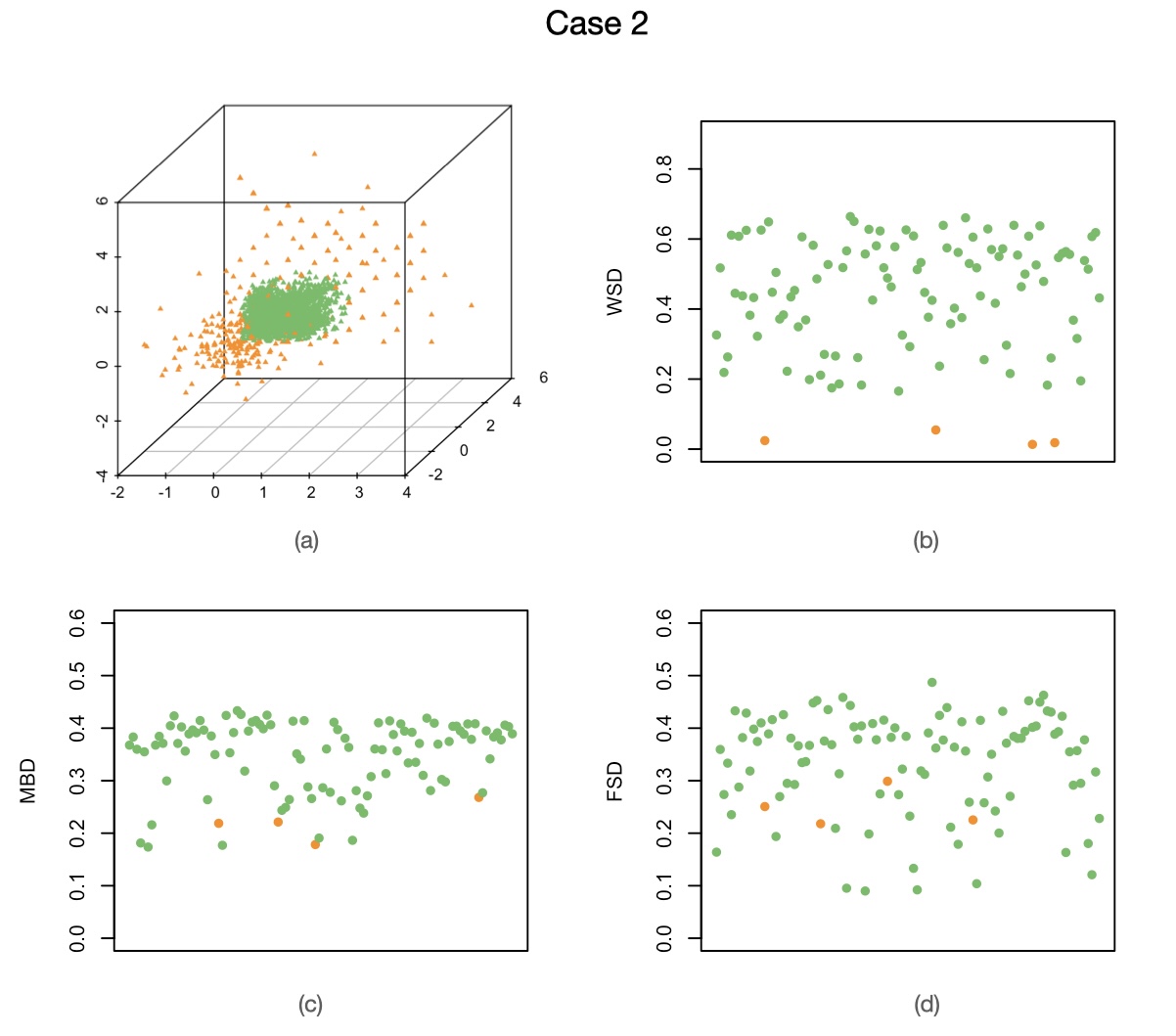}
	\caption{
		(a): The data points are drawn from the distributions of Case 2. The green triangles represent data points from the regular distributions, while orange triangles represent data points from the exotic distributions. (b): The green dots represent the WSD values of the regular distributions, while the orange dots represent the WSD values of the exotic distributions. (c): Each dot represents the MBD of a distribution. The coloring pattern is the same as before. (d):  Each dot represents the FSD of a distribution. The coloring pattern remains the same.
	}
	\label{WSD_Fdepth_2}
\end{figure}
The numerical results show the superiority of the proposed WSD when applied to distribution-valued data objects, which is expected since the WSD is specially designed for distribution-valued data objects and adapts to the geometry of the Wasserstein space.

\subsection{Wasserstein spatial depth vs. general metric depths}
\label{subsection:vs:general_metric_depth}
As discussed in Section \ref{section:comparison}, WSD enjoys more desirable theoretical properties than several general metric depths when adapted to the Wasserstein space. Here we also show empirically that WSD is more informative of the relative ``orderings'' of the distributions when compared to metric Lens depth \cite{Geenens2023} and metric spatial depth \cite{virta2023spatial} adapted to the Wasserstein space. We do not compare with the metric Tukey depth \cite{LopezTukeyMetric} because it is computationally too expensive. 
The comparison is carried out under the same two cases as in Section \ref{subsection:outlier_detect}. 
As shown in Figure \ref{compare_gegeral_metric_depths}, neither metric Lens depth nor metric spatial depth is able to detect outlier distributions. The outlier distributions are well embedded within the regular population. 
\begin{figure}[!htbp]
	\centering
	\includegraphics[width=1\textwidth]{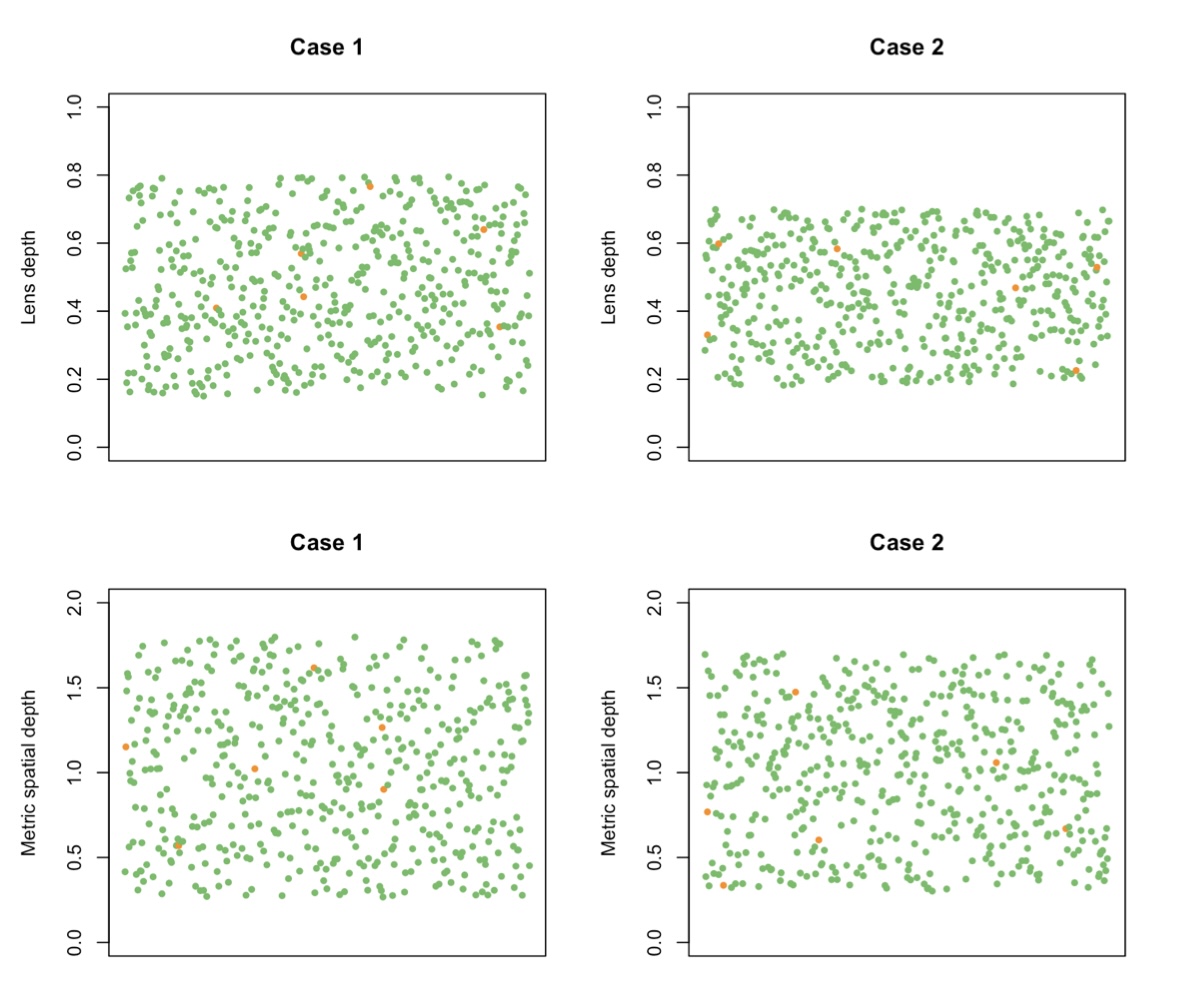}
	\caption{
		Upper left panel: empirical metric Lens depths of the distributions drawn according to Case 1 in Section \ref{subsection:outlier_detect}. Upper right panel:  empirical metric Lens depths of the distributions drawn according to Case 2 in Section \ref{subsection:outlier_detect}. 
		Lower left panel: empirical metric spatial depths of the distributions drawn according to Case 1. Lower right panel:  empirical metric spatial depths of the distributions drawn according to Case 2.
		The green dots represent regular distributions from the population $\mathbf{P}$, and the orange dots represent the outlier distributions.
	}
	\label{compare_gegeral_metric_depths}
\end{figure}

\subsection{Nonparametric testing based on WSD} \label{subsection:simulation:test} 
Here we assess the empirical performance of the testing procedure proposed in Section \ref{section:testing}. The goal is to test $H_0$ against $H_1$ 
\begin{equation*}
	H_0: \mathbf{P}=\mathbf{Q}  \qquad H_1: \mathbf{P}\neq\mathbf{Q}
\end{equation*}
given two sets of empirical distributions, $\mathbf{P}_{n,m}, \mathbf{Q}_{n,m}$. 
Four different cases are considered below. In each case we test varying alternatives $\{ \mathbf{Q}^k\}_{k=0,1,2,\ldots}$ against one null population $\mathbf{P}$. As $k$ gets larger, $\mathbf{Q}^k$ is increasingly different from $\mathbf{P}$.
In each case, we repeat the testing procedure $T=200$ times with $n=200$ distributions sampled from $\mathbf{P}$ and $\mathbf{Q}$, respectively, and $m=200$ data points sampled from each sampled distribution.  
The nominal level is chosen to be $\alpha = 0.05$ across four cases. 

\begin{itemize}
	\item Case 1: ${\bf P}$ is supported on a family of isotropic Gaussian distributions on $\R^2$, $\mathcal{N}(\boldsymbol{\mu}, \sigma^2\boldsymbol{I})$ with varying centers $\boldsymbol{\mu} \sim [\mathrm{Beta}(3,3) ]^4$ and varying variances $\sigma^2 \sim \mathrm{Unif}[3/16, 5/16]$. 
	The alternative populations $\{ \mathbf{Q}^k\}_{k=0,1,2,\ldots}$ are also families of isotropic Gaussian distributions on $\R^2$, $\mathcal{N}(\boldsymbol{\mu}_k, \sigma^2\boldsymbol{I})$ with varying centers $\boldsymbol{\mu}_k \sim [\mathrm{Beta}(3-0.2k, 3-0.2k) ]^4$ and varying variances $\sigma^2 \sim \mathrm{Unif}[3/16, 5/16]$. 
	
	\item Case 2: ${\bf P}$ is supported on a family of coordinate-wise Gamma distributions on $\R^2$, that is, each coordinate follows a $\rm{Gamma}(2, r)$ distribution with fixed shape parameter $2$ and varying rate parameters $r \sim
	\mathrm{Unif}[0,0.4]$, and two coordinates are independent. Each of the alternative populations $\{ \mathbf{Q}^k \}_{k=0, 1,2,\ldots}$ is supported on a family of coordinate-wise Gamma distributions with fixed shape parameter $2$ and varying rate parameters $r_k \sim 
	\mathrm{Unif}[0.03k, 0.4+0.03k]$. 
	
	\item Case 3: ${\bf P}$ is supported on a family of Poisson distributions with varying means $\lambda \sim \mathrm{Binomial}(17, 0.5)$. Each of the alternative populations $\{ \mathbf{Q}^k \}_{k=0, 1,2,\ldots}$ is supported on a family of Poisson distributions with varying means $\lambda_k \sim \mathrm{Binomial}(17-2k, 0.5)$. 
	
	\item Case 4: ${\bf P}$ is supported on a family of ``irregular'' distributions on $\R^3$, denoted by the random vector $Z \in \R^3$. $Z_1 \sim \mathrm{Unif}[0, c]$ with $c \sim  \mathrm{Weibull}(1,2)$; $Z_2 \sim \mathrm{Exp}(\lambda)$ with $\lambda \sim \mathrm{Unif}[1.4, 1.6]$ and $Z_1 \independent Z_2$; $Z_3 = 0.2 Z_1 + 0.1 Z_2 + 0.7 \mathcal{N}(0,1)$.  The $k$-th alternative population $\mathbf{Q}^k$ is distributed as $Z_1^k \sim \mathrm{Unif}[0, (1+0.1k)c_k]$ with $c_k \sim \mathrm{Weibull}(1,2)$; $Z_2^k \sim \mathrm{Exp}(\lambda_k)$ with $\lambda_k \sim \mathrm{Unif}[1.4-0.05k, 1.6-0.05k]$ and $Z_1^k \independent Z_2^k$; $Z_3^k = 0.2 Z_1^k + 0.1 Z_2^k + 0.7 \mathcal{N}(0,1)$.
\end{itemize} 

Shown in Figures \ref{nonparam_testing}, the Type I error is well controlled and the empirical power increases to one as the alternative deviates more and more from the null population. 
\begin{figure}[!htbp]
	\centering
	\includegraphics[width=1\linewidth]{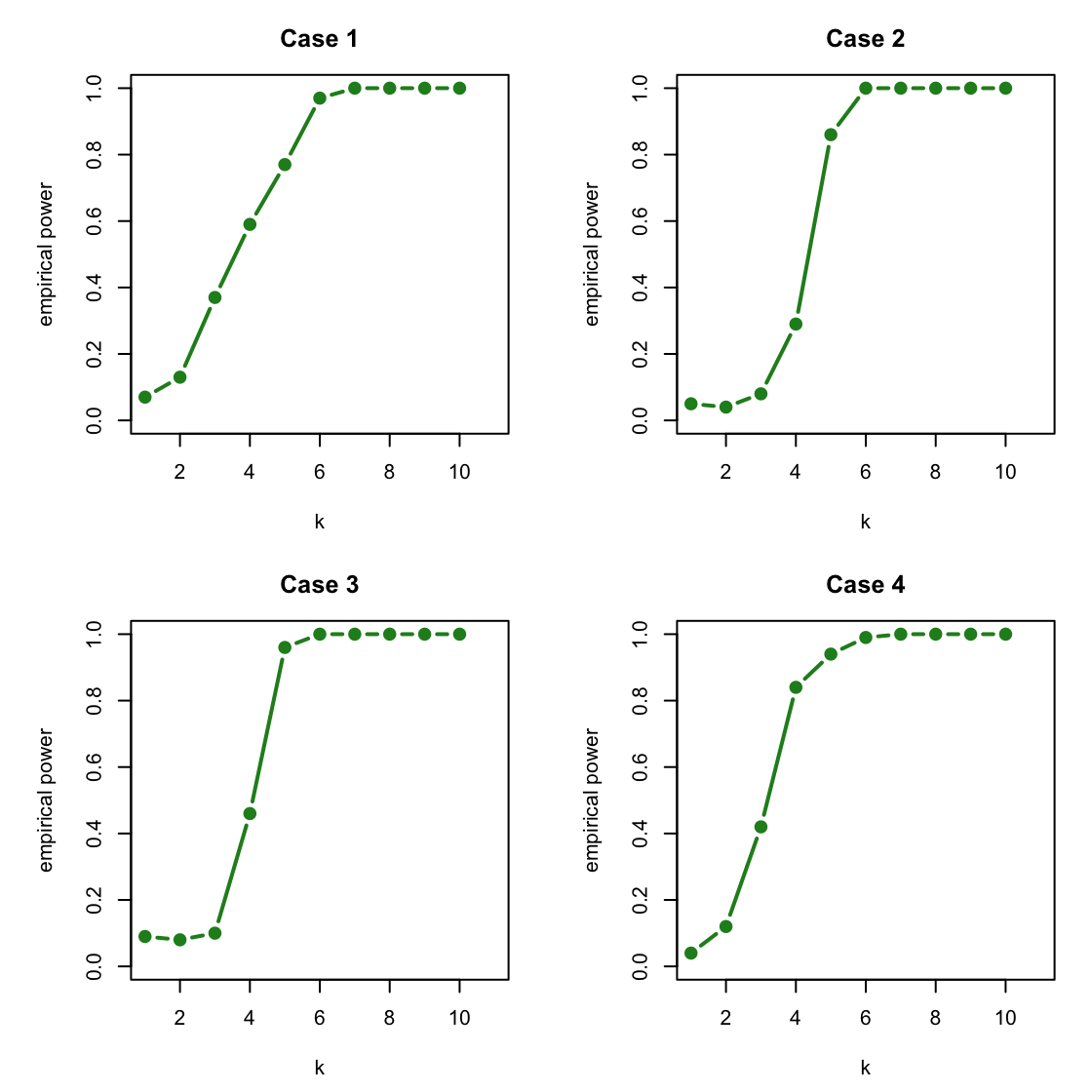}
	\caption{The Type I errors and empirical powers of the testing procedure proposed in Section \ref{section:testing}. Each plot corresponds to one case. In each plot, the first dot represents the Type I error because $\mathbf{P}=\mathbf{Q}$ at this point; the other dots represent empirical powers as $\mathbf{Q}^k$ deviates from $\mathbf{P}$.   }
	\label{nonparam_testing}
\end{figure}

\subsection{Wasserstein spatial depth vs. distance profiles}
\label{subsection:vs:distance_profile}
Distance Profile (DP) \cite{dubey2024metric} is a recently proposed metric to measure the centrality of a point with respect to a distribution in a general metric space. The DP of any point $\omega$ with respect to a distribution $\mu$ in a metric space is a univariate CDF that captures the relative position of $\omega$ with respect to $\mu$. 
The authors of \cite{dubey2024metric} discuss that if the metric space is of strong negative type (which is not known to be the case of the Wasserstein space), the collection of DPs at all points uniquely characterize the distribution $\mu$. 
Then DP was used to construct a permutation test for comparing two distributions in Wasserstein space. 
Here we compare the empirical performances of WSD-based (Section \ref{section:testing}) and DP-based two sample tests.
DP-based testing is conducted using the \texttt{R} package \texttt{ODP}, which is publicly available at \url{https://github.com/yqgchen/ODP}.
Among the four cases below, Case 1 and Case 2 are  the same as those in Section 6.3 of \cite{dubey2024metric} for fair comparison. The simulations below are conducted with $n=400$, $m=200$, and number of repetitions $T=200$ to estimate the empirical power. The nominal level is chosen to be $\alpha = 0.05$ across all four cases. 
\begin{itemize}
	\item Case 1: ${\bf P}$ is supported on a family of two-dimensional Gaussian distributions with fixed covariance matrix $0.25\boldsymbol{I}$ and varying centers, that is, $\mathcal{N}(\boldsymbol{\mu}, 0.25\boldsymbol{I})$ with $\boldsymbol{\mu}\sim \mathcal{N}(\mathbf{0}, 0.25\boldsymbol{I})$. ${\bf Q}$ is also supported on $\{ \mathcal{N}(\boldsymbol{\mu}, 0.25\boldsymbol{I}) \}$ with $\boldsymbol{\mu}\sim \mathcal{N}((\delta,0)^\top, 0.25\boldsymbol{I})$ and an increasing $\delta$.
	\item Case 2: ${\bf P}$ is same as in Case 1 except that $\boldsymbol{\mu}\sim \mathcal{N}(\mathbf{0}, 0.16\boldsymbol{I})$. ${\bf Q}$ is same as in Case 1 except that $\boldsymbol{\mu}\sim \mathcal{N}(\mathbf{0}, (0.4+\delta)^2\boldsymbol{I})$ with an increasing $\delta$.
	\item Case 3: ${\bf P}$ is supported on a family of two-dimensional centered Gaussian distributions with covariance matrix $\sigma_P^2\boldsymbol{I}$.  ${\bf Q}$ is also supported on a family of two-dimensional centered Gaussian distributions but with covariance matrix $\sigma_Q^2\boldsymbol{\Sigma}$, where $\Sigma_{ij} = \delta^{|i-j|}$ with an increasing $\delta$.  
	Here $\sigma_P, \sigma_Q \sim \rm{Unif}[0.1, 0.3]$.
	\item Case 4: ${\bf P}$ and $\mathbf{Q}$ are supported on  families of Gaussian mixtures on $\R^2$, \textit{i.e.}, 
	\begin{align*}
		&w_1 \mathcal{N}(\mathbf{c}_1, 0.15^2 \boldsymbol{I}) +  w_2 \mathcal{N}(\mathbf{c}_2, 0.15^2 \boldsymbol{I}) +  w_3 \mathcal{N}(\mathbf{c}_3, 0.15^2 \boldsymbol{I}) +  w_4 \mathcal{N}(\mathbf{c}_4, 0.15^2 \boldsymbol{I}), \\
		& \mathbf{c}_1 = (0, 0)^\top, \quad \mathbf{c}_2 = \big(0.4\cos(-\pi/6), 0.4\sin(-\pi/6)\big)^\top,\\ &\mathbf{c}_3 = \big(0.4\cos(7\pi/6), 0.4\sin(7\pi/6)\big)^\top, \quad \mathbf{c}_4 = (0, 0.4)^\top
	\end{align*}
	with varying weights $w_1, w_2, w_3, w_4$. The weights $\mathbf{w}^P= (w_1^P, w_2^P, w_3^P, w_4^P)$ of $P\sim \mathbf{P}$ follow a Dirichlet distribution with $\boldsymbol{\alpha}^P = (20, 20, 20, 20)^\top$; the weights $\mathbf{w}^Q$ of $Q\sim \mathbf{Q}$ follow a Dirichlet distribution with $\boldsymbol{\alpha}^Q = (20-\delta, 20-\delta, 20-\delta, 20-\delta)^\top$.
\end{itemize}
\begin{figure}[!htbp]
	\centering
	\includegraphics[width=1\linewidth]{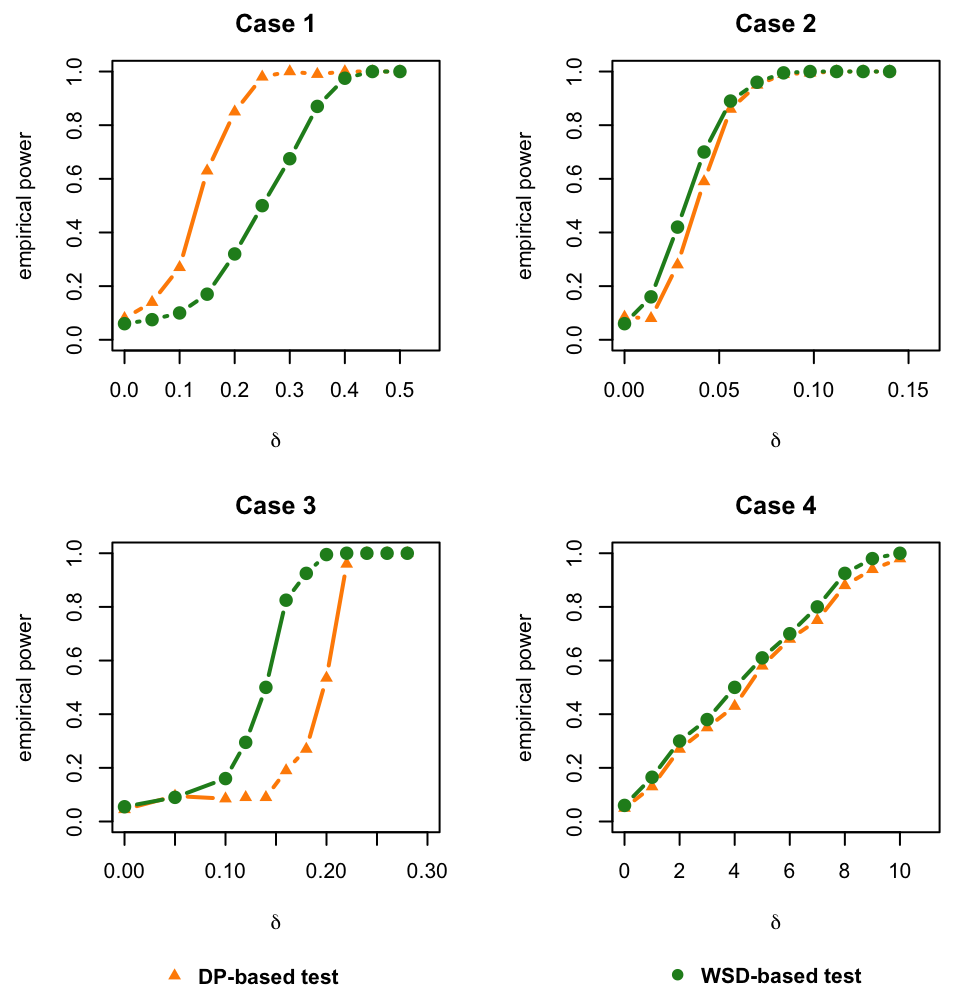}
	\caption{The empirical powers of the nonparametric two-sample tests based on WSD and on DP. Each plot corresponds to one case. The orange triangles correspond to the Type I errors and empirical powers of DP-based tests. The green dots correspond to the Type I errors and empirical powers of WSD-based tests. }
	\label{WSD_vs_DP}
\end{figure}
Shown in Figure \ref{WSD_vs_DP}, WSD has slightly better performances than DP in Cases 2 and 4.
WSD significantly outperforms DP in case 3, while the converse holds in Case 1.
Furthermore, we observed in extensive simulations that WSD-based tests maintain greater robustness across different levels of point dispersion in Wasserstein space. They are also more stable when the pairwise Wasserstein distances of the points vary, \textit{i.e.}, scaled invariance.

\section{Application}
\label{section:application}
Nowadays, climate change is a major concern across the society.  Considerable amount of information can be extracted from longitudinal series of daily temperatures. We apply the notion of WSD to explore a dataset recording European daily temperatures during the past two centuries.  

The data is collected from the public database ``European Climate Assessment and Dataset''\footnote{ https://www.ecad.eu/dailydata/index.php.}. 
It contains the daily average temperatures collected at $40$ meteorological stations located across Europe, including Austria, Croatia, Czech Republic, Denmark, Finland, Germany, Sweden, and United Kingdom, from year 1874 to 2023.  These $40$ meteorological stations cover a broad range of Europe and are representative of the region. The goal is to explore the temperature change over the years. 

We consider monthly temperatures obtained by averaging daily temperatures per month. For each weather station, we obtain a $12$ monthly-average temperature curve,  represented by a vector in $\R^{12}$. Hence, the monthly temperatures of each year correspond to one distribution on $\R^{12}$. For a particular year, the $12$ monthly temperatures (forming a vector in $\R^{12}$) collected at each station act as a sample point drawn from this distribution. Finally, we gather $150$ distributions (from year 1874 to year 2023) with each distribution associated to $40$ sample points (for the $40$ meteorological stations), and where each sample point is a $12$-dimensional real vector. In the following we assume that the distributions are drawn each year independently. \vskip .1in
Contrary to other work, we do not consider the annual evolution of the temperatures for a particular place but rather analyze the different temperature curves at all locations at the same time. We aim at understanding weather change at a global scale by considering the $40$ different locations as representatives of the European climate.

Within this framework, we compute the empirical WSDs of these $150$ distributions as in \eqref{empirical_WSD_1}. Several outlier years are identified based on their excessively small WSDs. As we discuss next, these identified ``abnormal years'' are consistent with historical records, which further validate the practical utility of the WSD.\\ For the reproducibility of our research, the code for data analysis is publicly available at \url{https://github.com/YishaYao/Wasserstein-Spatial-Depth/tree/main}.

\begin{figure}[!htbp]
	\centering
	\includegraphics[width=1\textwidth]{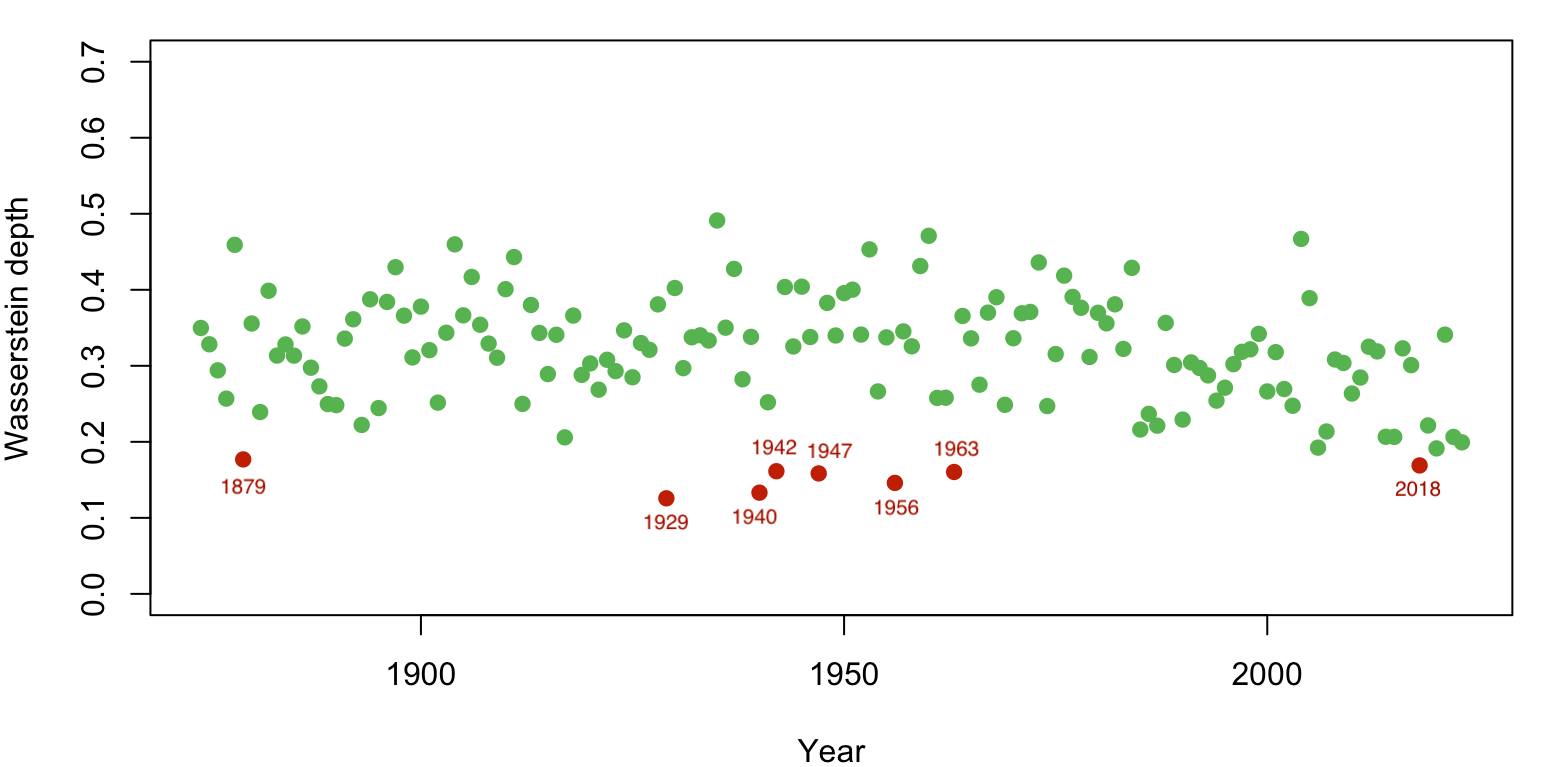}
	\caption{The green dots represent regular/representative/central distributions, while the red dots correspond to the distributions near the outskirt and “far” from the center.}
	\label{application-1}
\end{figure}
The values of the $150$ empirical WSDs are shown in Figure \ref{application-1}.  The lowest $5\%$ values, which we consider as outliers, are colored red, and the corresponding years are also marked. Based on empirical WSDs, the temperatures at years 1879, 1929, 1940, 1942, 1947, 1956, 1963, and 2018 are more ``exotic'' or near outskirt. After searching among historical documentations, we indeed fond evidences to support this discovery. 
Year 1879 was an extremely cold year, featured with a unusually snowy winter (November and December). 
The first two months of 1929 were recorded as one of the coldest winters in Europe during the past century with temperature reaching down to -30°C in central Europe. 
Both year 1940 and year 1942 were marked by severe winters with dramatic ice storms, and year 1942 had a cool summer.
The weather in year 1947 was unusually cold in winter and record-breaking hot in summer.
Europe experienced severe cold waves in both winters of 1956 and 1963. The well-known 2018 European drought and heat wave led to record-breaking temperatures and wildfires in many parts of Europe. 

To get a better view on how these years' temperatures differ from other regular years', we compare the four most ``exotic'' years with the most regular years. We pick the two years with the largest WSDs as our ``regular years'', year 1935 and year 1960. 
In each plot of Figure \ref{application2}, the bundle of green curves represents the temperature trends of the $40$ locations in the regular years (1935 and 1960), while the bundle of red curves corresponds to one particular outlier year. The green bundle and red bundle do exhibit clear visual differences in temperature trends over the months. 
\begin{figure*}[!htbp] 
	\subfloat[outlier year 1929]{%
		\includegraphics[width=.45\linewidth]{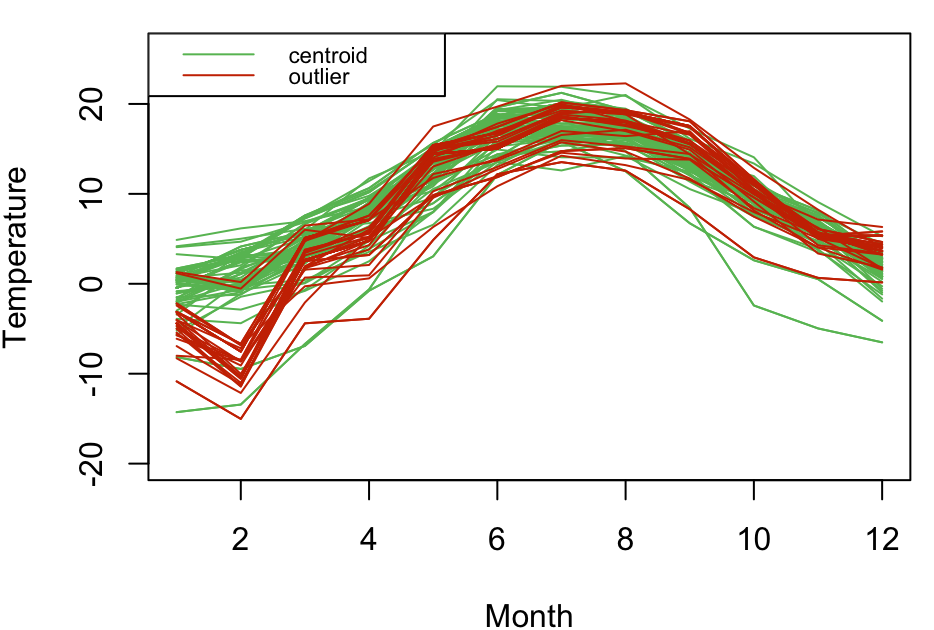}%
	}\hfill
	\subfloat[outlier year 1940]{%
		\includegraphics[width=.45\linewidth]{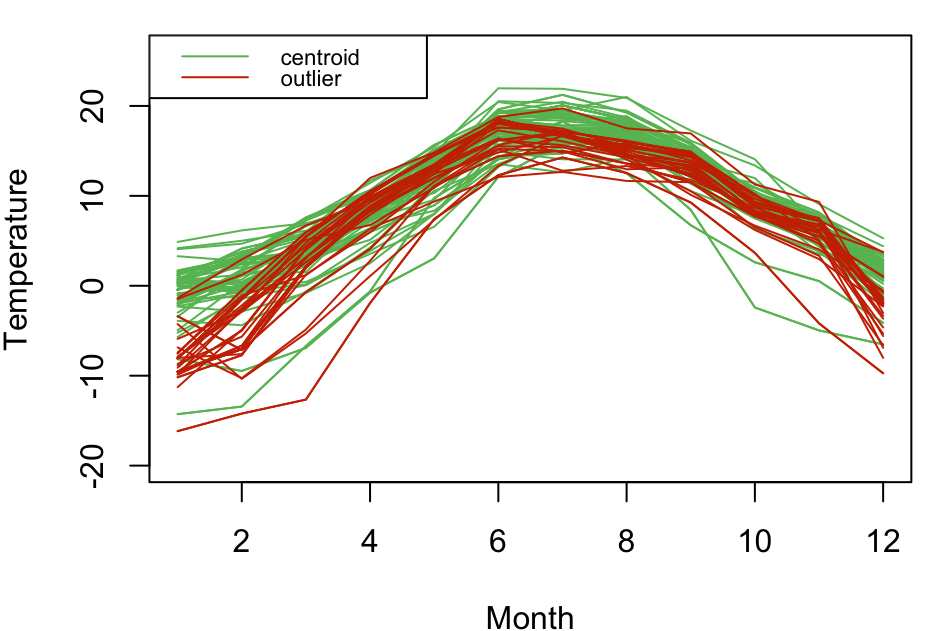}%
	}\\
	\subfloat[outlier year 1956]{%
		\includegraphics[width=.45\linewidth]{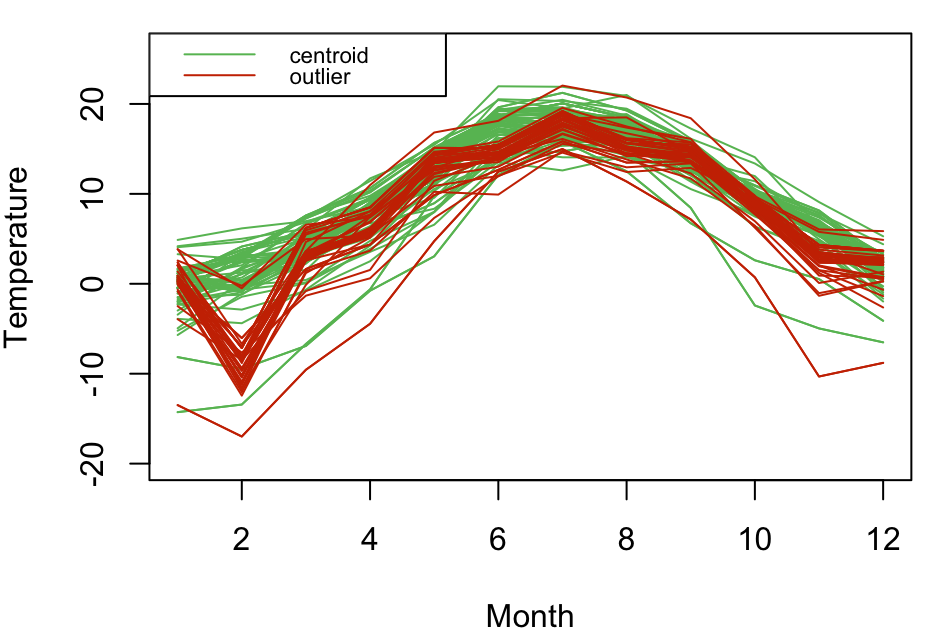}%
	}\hfill
	\subfloat[outlier year 2018]{%
		\includegraphics[width=.45\linewidth]{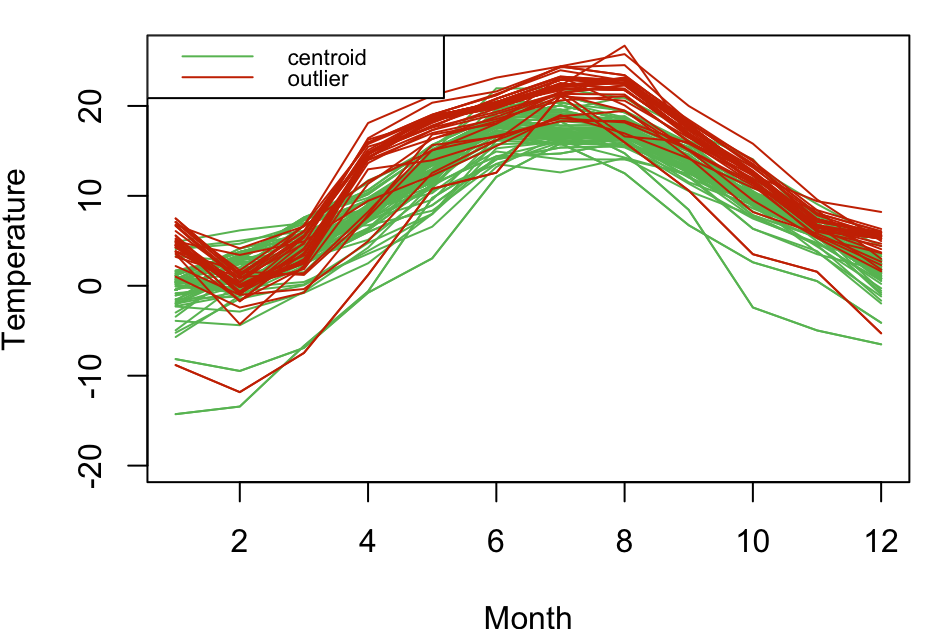}%
	}
	\caption{Comparisons between the most regular years 1935 and 1960 (the two years with the largest WSDs) and four outlier years. In each plot, the bundle of green curves represents the temperature trends in years 1935 and 1960 at 40 locations in Europe (totally $80$ green curves); the bundle of red curves represents the temperature trends in an outlier year at the same 40 locations (totally $40$ red curves).  }
	\label{application2}
\end{figure*}

\section{Further directions and future work}
\label{section:future:directions}

In this work, we propose a new notion of depth on the Wasserstein space. We demonstrate that it preserves critical properties of conventional statistical depths. Additionally, it has a straightforward empirical counterpart that can be easily computed from sample data and is asymptotically consistent.  Numerical simulations and real data analysis further support its practical utility. Importantly, in Section \ref{subsection:vs:functional}, we demonstrate that simply embedding distributions into linear Hilbert spaces, and relying on existing FDA methods, is not satisfying. In contrast, the WSD proves to be very informative in this section.

Note that we have defined the new notion of WSD,
$ {\rm WSD}(Q; {\bf P})$,
for absolutely continuous distributions $Q$ and where ${\bf P}$ can be arbitrary. This is because our approach exploits the definition of the geodesics in the Wasserstein space (see Section~\ref{subsec:WassersteinGeom}). 

When $Q$ is not absolutely continuous, the geodesic between $Q$ and another distribution $P$ might be not unique.
In this case, the set of geodesics is given by the laws of the random vectors $(1-t)\X+ t\Y$ where the law of the random vector $(\X,\Y)$, namely $\pi_{P,Q}$, is an optimal transport plan, as in \eqref{Wassertesin} with $p=2$.
Hence, uniqueness of the geodesics is equivalent to uniqueness of the transport plans.

Thus, if with ${\bf P}$-probability one  $P \sim {\bf P}$ is absolutely continuous, even if $Q$ is not absolutely continuous, the geodesics are unique and, following the route of Section~\ref{subsec:WassersteinGeom},  we can still define a notion of depth as follows: 
\begin{equation} \label{eq:new:depth}
	\rm WSD^{\text{discr}}(Q; {\bf P})
	:=
	1-
	\left( 
	\underset{{(P,P')\sim \mathbf{P}\otimes \mathbf{P}}}{ \E}
	\left[ \int  {\left\langle
		\frac{\x-\y }{W_2(P,Q)} , 
		\frac{\x-\y' }{W_2(P',Q)} \right\rangle }d\pi_{Q,P,P'}(\x,\y, \y' ) \right]\right)^{1/2},
\end{equation} 
where $\pi_{Q,P,P'}(\x,\y, \y' )$ is the distribution of a vector $(\X,\Y, \Y')$ with $(\X,\Y)\sim \pi_{Q,P}$, $(\X,\Y')\sim \pi_{Q,P'}$ and $\Y$ and $\Y'$ are independent given $\X$.
Here $\pi_{Q,P}$ (resp. $\pi_{Q,P'}$) is the unique optimal transport plan from $Q$ to $P$ (resp. $P'$).
This provides a  definition of WSD for any distribution $Q$ when ${\bf P}$ samples a.s. absolutely continuous distributions.
It can be seen, similarly as the proof of Theorem~\ref{Theo:Main}, that $ \rm WSD^{\text{discr}}(Q; {\bf P})$ would be $[0,1]$-valued (and the quantity in the square root being non-negative). 

We leave for future exploration the practical utility of this complementary WSD, along with the task of establishing analogous favorable mathematical properties as those demonstrated in this paper. Note that the depth in \eqref{eq:new:depth} coincides with $ {\rm WSD}(Q; {\bf P})$ in the special case where both $Q$ and (a.s.) the samples from {\bf P} are absolutely continuous. This can be seen from the arguments leading to \eqref{eq:An:inner:product} in the Appendix. 

Finally, for computational reasons, the statistics and machine learning community has also focused on regularized optimal transport \citep{cuturi2013sinkhorn,peyre2019computational}. It is an interesting prospect as well to extend the WSD to regularized optimal transport.



\appendix

\section{Proof of Theorem~\ref{Theo:Main}}

\subsection{Values in $[0,1]$}\label{subsection:liesIn01}   Here we prove that $\mathrm{SD}(Q; {\bf P})\in [0,1]$ for all $Q\in  \mathcal{P}^{a.c}_2(\R^d)$, which is probably the easiest statement to prove. To prove the upper bound we realize that 
$$ \left(\int{ \left\|\E_{P\sim \mathbf{P}}\left[  \frac{\x-T_{Q,P}(\x) }{\mathcal{W}_2(P,Q)}\right]\right\|^2} \mathrm{d} Q(\x)\right)^{\frac{1}{2}} \geq 0,$$
so that 
$$ \mathrm{SD}(Q; {\bf P})=1-  \left(\int{ \left\|\E_{P\sim \mathbf{P}}\left[  \frac{\x-T_{Q,P}(\x) }{\mathcal{W}_2(P,Q)}\right]\right\|^2} \mathrm{d} Q(\x)\right)^{\frac{1}{2}} \leq 1 . $$
To prove the lower bound we observe that 
\begin{align*}
	& \left(\int{ \left\|\E_{P\sim \mathbf{P}}\left[  \frac{\x-T_{Q,P}(\x) }{\mathcal{W}_2(P,Q)}\right]\right\|^2} \mathrm{d} Q(\x)\right)^{\frac{1}{2}} \\
	&=\sup_{\|G\|_{L^2(Q)}\leq 1} \left(\int{ \left\langle \E_{P\sim \mathbf{P}}\left[  \frac{\x-T_{Q,P}(\x)  }{\mathcal{W}_2(P,Q)}\right],  G(\x)\right\rangle} \mathrm{d} Q(\x)\right) \\
	&= \sup_{\|G\|_{L^2(Q)}\leq 1} \left(\int{  \E_{P\sim \mathbf{P}}\left[  \frac{\left\langle\x-T_{Q,P}(\x),G(\x)\right\rangle  }{\mathcal{W}_2(P,Q)}\right] } \mathrm{d} Q(\x)\right) \\
	& \leq   \E_{P\sim \mathbf{P}}\left[  \frac{ \sup_{\|G\|_{L^2(Q)}\leq 1} \int{ \left\langle\x-T_{Q,P}(\x),G(\x)\right\rangle } \mathrm{d} Q(\x) }{\mathcal{W}_2(P,Q)}\right] 
	\\
	&=   \E_{P\sim \mathbf{P}}\left[  \frac{ \|I-T_{Q,P}\|_{L^2(Q)}}{\mathcal{W}_2(P,Q)}\right] =  \E_{P\sim \mathbf{P}}\left[  \frac{ \mathcal{W}_2(P,Q)}{\mathcal{W}_2(P,Q)}\right] =1,
\end{align*}
so that 
$$ \mathrm{SD}(Q; {\bf P})=1-  \left(\int{ \left\|\E_{P\sim \mathbf{P}}\left[  \frac{\x-T_{Q,P}(\x) }{\mathcal{W}_2(P,Q)}\right]\right\|^2} \mathrm{d} Q(\x)\right)^{\frac{1}{2}} \geq 0 . $$

\subsection{Transformation invariance}
Theorem~1.2 in \cite{kloeckner2010geometric} describes the group of isometries of $(\mathcal{P}_2(\R^d),\mathcal{W}_2)$ for $d \ge 2$. Any isometry $F$ can be written as the composition of   $\Phi(\varphi)$
and a trivial isometry. Recall that $\Phi(\varphi):P\mapsto \Phi(\varphi)(P)$  where $\varphi:\R^d\to \R^d$ is a linear isometry and $\Phi(\varphi)(P)$ is the
law of the random variable 
$$ \varphi(\X-\E[\X])+\E[\X], \quad {\rm for
}\ \X\sim P . $$

Therefore, it is enough to show that the WSD is invariant with respect to trivial isometries and isometries of type $\Phi(\varphi)$ for some linear isometry $\varphi:\R^d\to \R^d$.

\paragraph*{Invariance under trivial isometries}
Let  $\A$ be a $d\times d$ orthogonal matrix and ${\bf b}\in \R^d$. We write 
$$ f_{\A,{\bf b}}(\x)=\A\x+{\bf b}.   $$
The mapping
$$ S_P=f_{\A,{\bf b}} \circ T_{Q,P} \circ (f_{\A,{\bf b}})^{-1}:\x\mapsto  \A\, T_{Q,P}(\A^T(\x-{\bf b}) ) +{\bf b}$$
is the a.s.\ defined gradient of a convex function and (by construction) pushes $(f_{\A,{\bf b}})_\# Q $ forward to $(f_{\A,{\bf b}})_\# P $. 
Therefore, $S_P$ is the optimal transport map from  $(f_{\A,{\bf b}})_\# Q $ forward to $(f_{\A,{\bf b}})_\# P $ (cf.\ \cite{McCann}). 
Hence, the following holds for the induced isometry $F:P\mapsto F(P)=(f_{\A,{\bf b}})_{\#} P$:
\begin{align*}
	\mathrm{SD}(F(Q); F_\# {\bf P})&=1-  \left(\int{ \left\|\E_{P}\left[  \frac{\x-S_P(\x)}{\mathcal{W}_2(P,Q)}\right]\right\|^2} \mathrm{d}((f_{\A,{\bf b}})_{\#} Q)(\x)\right)^{\frac{1}{2}}\\
	&=1-  \left(\int{ \left\|\E_{P}\left[  \frac{f_{\A,{\bf b}}(\x)-f_{\A,{\bf b}} \circ T_{Q,P}(\x) }{\mathcal{W}_2(P,Q)}\right]\right\|^2} \mathrm{d}Q(\x)\right)^{\frac{1}{2}}\\
	&=1-  \left(\int{ \left\|\E_{P}\left[  \frac{\A(\x- T_{Q,P}(\x)) }{\mathcal{W}_2(P,Q)}\right]\right\|^2} \mathrm{d}Q(\x)\right)^{\frac{1}{2}}\\
	&=1-  \left(\int{ \left\|\A\,\E_{P}\left[  \frac{\x- T_{Q,P}(\x) }{\mathcal{W}_2(P,Q)}\right]\right\|^2} \mathrm{d}Q(\x)\right)^{\frac{1}{2}}\\
	&=1-  \left(\int{ \left\|\E_{P}\left[  \frac{\x- T_{Q,P}(\x) }{\mathcal{W}_2(P,Q)}\right]\right\|^2} \mathrm{d}Q(\x)\right)^{\frac{1}{2}}=\mathrm{SD}(Q;  {\bf P}).
\end{align*}
This proves the invariance  under trivial isometries.

\paragraph*{Invariance under  isometries of type $\Phi(\varphi)$} Let $\varphi$ be a linear isometry. Then the mapping $S_P$ solving  
$$ S_P(\varphi(\x-\E_{\X\sim Q}[\X])+\E_{\X\sim Q}[\X])=  \varphi(T_{Q,P}(\x)-\E_{\Y\sim P}[\Y])+\E_{\Y\sim P}[\Y] $$
is, as in the previous case, the optimal transport map from $\Phi(\varphi)(Q)$ to $\Phi(\varphi)(P)$. Then it holds that 
\begin{align*}
	\mathrm{SD}(\Phi(\varphi)(Q); (\Phi(\varphi))_\# {\bf P})&=1-  \left(\int{ \left\|\E_{P\sim  {\bf P} }\left[  \frac{\x-S_P(\x)}{\mathcal{W}_2(P,Q)}\right]\right\|^2} \mathrm{d}\Phi(\varphi) (Q))(\x)\right)^{\frac{1}{2}}\\
	&=1-  \bigg(\int   \bigg\|\E_{P\sim {\bf P}}  \bigg[  \frac{\varphi(\x-\E_{\X\sim Q}[\X])+\E_{\X\sim Q}[\X]}{\mathcal{W}_2(P,Q)}  \\
	&\qquad \qquad -\frac{\varphi(T_{Q,P}(\x)-\E_{\Y\sim P}[\Y])+\E_{\Y\sim P}[\Y]}{\mathcal{W}_2(P,Q)}\bigg]  \bigg\|^2 \mathrm{d} Q(\x)  \bigg)^{\frac{1}{2}}.  
\end{align*}
As $\varphi$ is linear, we get the equality 
\begin{align*}
	\mathrm{SD}(\Phi(\varphi)(Q); (\Phi(\varphi))_\# {\bf P})=
	&1-  \bigg(\int   \bigg\|\E_{P\sim {\bf P}}  \bigg[  \frac{\varphi(\x-T_{Q,P}(\x)-\E_{\X\sim Q}[\X]+\E_{\Y\sim P}[\Y])}{\mathcal{W}_2(P,Q)}\\
	&\qquad \qquad\qquad \qquad +\frac{\E_{\X\sim Q}[\X]-\E_{\Y\sim P}[\Y]}{\mathcal{W}_2(P,Q)}\bigg]  \bigg\|^2 \mathrm{d} Q(\x)  \bigg)^{\frac{1}{2}}\\
	=&1-  \bigg(\int   \bigg\|\varphi \left(\E_{P\sim {\bf P}}  \bigg[  \frac{\x-T_{Q,P}(\x)-\E_{\X\sim Q}[\X]+\E_{\Y\sim P}[\Y])}{\mathcal{W}_2(P,Q)}\bigg]\right)
	\\
	&\qquad \qquad\qquad \qquad +\E_{P\sim {\bf P}}  \bigg[\frac{\E_{\X\sim Q}[\X]-\E_{\Y\sim P}[\Y]}{\mathcal{W}_2(P,Q)}\bigg]  \bigg\|^2 \mathrm{d} Q(\x)  \bigg)^{\frac{1}{2}}.
\end{align*}
Develop the squares and use the fact that $\varphi$ is an isometry to obtain 
\begin{align*}
	&\mathrm{SD}(\Phi(\varphi)(Q); (\Phi(\varphi))_\# {\bf P})\\
	=
	&1-  \bigg(\int \bigg\{  \bigg\| \E_{P\sim {\bf P}}  \bigg[  \frac{\x-T_{Q,P}(\x)}{\mathcal{W}_2(P,Q)}\bigg] \bigg\|^2+ 2\bigg\|\E_{P\sim {\bf P}}  \bigg[  \frac{\E_{\X\sim Q}[\X]-\E_{\Y\sim P}[\Y]}{\mathcal{W}_2(P,Q)} \bigg] \bigg\|^2\\
	& +2\bigg\langle \E_{P\sim {\bf P}}  \bigg[  \frac{\x-T_{Q,P}(\x)}{\mathcal{W}_2(P,Q)}\bigg], \E_{P\sim {\bf P}}  \bigg[  \frac{\E_{\Y\sim P}[\Y]-\E_{\X\sim Q}[\X]}{\mathcal{W}_2(P,Q)}\bigg]\bigg\rangle \\
	&  -2\bigg\langle \varphi\left(\E_{P\sim {\bf P}}  \bigg[  \frac{\x-T_{Q,P}(\x)}{\mathcal{W}_2(P,Q)}\bigg]\right), \E_{P\sim {\bf P}}  \bigg[  \frac{\E_{\Y\sim P}[\Y]-\E_{\X\sim Q}[\X]}{\mathcal{W}_2(P,Q)}\bigg]\bigg\rangle \\
	& -2\bigg\langle \varphi\left(\E_{P\sim {\bf P}}  \bigg[  \frac{\E_{\Y\sim P}[\Y]-\E_{\X\sim Q}[\X]}{\mathcal{W}_2(P,Q)}\bigg]\right), \E_{P\sim {\bf P}}  \bigg[  \frac{\E_{\Y\sim P}[\Y]-\E_{\X\sim Q}[\X]}{\mathcal{W}_2(P,Q)}\bigg]\bigg\rangle \bigg\}\mathrm{d}Q(\x)  \bigg)^{\frac{1}{2}}.
\end{align*}
The second term of the sum cancels with the third and the fourth with the last one as a consequence of Fubini's theorem, the linearity of $\varphi$ and the fact that $(T_{Q,P})_\#Q=P$. 
Therefore, the result follows.

\subsection{Vanishing at infinity}
The goal of this section is to prove that
$ \mathrm{SD}(Q_n; {\bf P}) \to 0 $ as $ \mathcal{W}_2(Q_n,P)\to \infty$
for one $P\in \mathcal{P}_2(\R^d)$. 
\begin{Remark}\label{RemarkOneToAll}
	Note that $\mathcal{W}_2(Q_n,P)\to \infty$ implies that for any other $P'\in \mathcal{P}_2(\R^d)$, 
	$$ \mathcal{W}_2(Q_n,P') \geq \mathcal{W}_2(Q_n,P)-\mathcal{W}_2(P',P) \to + \infty. $$
	Moreover, for any compact set $K$, 
	$$ \inf_{P\in K} \mathcal{W}_2(Q_n,P)\to + \infty. $$
\end{Remark}
Let $\{Q_n\}_{n\in \N}\subset \mathcal{P}^{a.c}_2(\R^d)$ be such that $ \mathcal{W}_2(Q_n,P)\to \infty$ for all $P\in \mathcal{P}_2(\R^d)$. Recall that 
$$ \mathrm{SD}(Q_n; {\bf P}):=1-  \left(\int{ \left\|\E_{P\sim \mathbf{P}}\left[  \frac{\x-T_{Q_n ,P}(\x) }{\mathcal{W}_2(P,Q_n)}\right]\right\|^2} \mathrm{d}Q_n (\x)\right)^{\frac{1}{2}} $$
with the convention 
$ \frac{\x-T_{Q_n ,P}(\x) }{\mathcal{W}_2(P,Q_n)} =\boldsymbol{0}$
if ${\mathcal{W}_2(P,Q_n)} =0$. First we want to get rid of this last pathological case. 
Let
\begin{equation}
	\label{eq:An}
	A_n:=\int{\left\|\E_{P\sim \mathbf{P}}\left[  \frac{\x-T_{Q_{n} ,P}(\x) }{\mathcal{W}_2(P,Q_{n})}\right]\right\|^2} \mathrm{d}Q_{n} (\x).
\end{equation}
Let $E_n = \{Q_{n}\}$.
Note that when $P \in E_n$, a $0$ appears in the expression of $A_n$
(recall the convention $\boldsymbol{0}/0 = \boldsymbol{0}$).
For each $n$, we modify ${\bf P}={\bf P}_1+ {\bf P}_2$, where ${\bf P}_1$ is a measure on $\mathcal{P}_2(\R^d) \backslash E_n$ and ${\bf P}_2$ is a measure on $E_n$,
by ${\bf P}'={\bf P}_1+{\bf \widetilde{P}}_2$, 
where ${\bf \widetilde{P}}_2$ is an arbitrary measure on $\mathcal{P}_2(\R^d) \backslash E_n$ such that ${\bf \widetilde{P}}_2(\mathcal{P}_2(\R^d) ) = {\bf P}(E_n)$. Note that ${\bf P}'$ is also a probability measure. 

Since the measure ${\bf P}$ is tight
and $Q_n$ diverges, it is clear that  ${\bf P}(E_n) \to 0$ as $n \to \infty$. Moreover, 
\begin{align*}
	& \left|
	\left(
	\int{\left\|\E_{P\sim \mathbf{P}'}\left[  \frac{\x-T_{Q_{n} ,P}(\x) }{\mathcal{W}_2(P,Q_{n})}\right]\right\|^2} \mathrm{d}Q_{n} (\x)
	\right)^{1/2}
	-(A_n)^{1/2}
	\right| 
	\\ 
	& \le{\bf P}(E_n)
	\left(
	\int{\left\|
		\E_{P\sim \frac{{\bf \widetilde{P}}_2}{{\bf P}(E_n)}}
		\left[ 
		\frac{\x-T_{Q_{n} ,P}(\x) }{\mathcal{W}_2(P,Q_{n})}\right]\right\|^2} \mathrm{d}Q_{n} (\x)
	\right)^{1/2}.
\end{align*}
Since the spatial depth lies in $[0,1]$, we can upper bound this quantity by 
${\bf P}(E_n)$ and obtain that the limit of $\mathrm{SD}(Q_{n}; {\bf P})$ is that of $\mathrm{SD}(Q_{n}; {\bf P}')$. Therefore, we can feel free to assume that ${\mathcal{W}_2(P,Q_n)} =0$ does not happen for $n$ big enough and for $P \sim {\bf P}$. 

We prove that $A_n \to 1$, where $A_n$ is defined in \eqref{eq:An}.
To do so, let $P'$ be an independent copy of $P$, so that 
\begin{align} \label{eq:An:inner:product}
	A_n&=\int{\left\langle\E_{P\sim \mathbf{P}}\left[  \frac{\x-T_{Q_{n} ,P}(\x) }{\mathcal{W}_2(P,Q_{n})}\right], \E_{P'\sim \mathbf{P}}\left[  \frac{\x-T_{Q_{n} ,P'}(\x) }{\mathcal{W}_2(P',Q_{n})}\right] \right\rangle} \mathrm{d}Q_{n} (\x) \notag \\
	& =\int{\E_{P, P'\sim \mathbf{P}}\left[  \frac{ \left\langle\x-T_{Q_{n} ,P}(\x), \x-T_{Q_{n} ,P'}(\x)  \right\rangle}{\mathcal{W}_2(P,Q_{n}) \mathcal{W}_2(P',Q_{n})}\right] } \mathrm{d}Q_{n} (\x).
\end{align}
In order to reduce the size of the formulas we call 
$ B_{P,n}(\x)=\x-T_{Q_{n} ,P}(\x)$ and $ B_{P',n}(\x)=\x-T_{Q_{n} ,P'}(\x)$. Then 
$$  A_n=\int{\E_{P, P'\sim \mathbf{P}}\left[  \frac{ \left\langle B_{P,n}(\x), B_{P',n}(\x)  \right\rangle}{\|B_{P,n}\|_{L^2(Q_n)} \|B_{P',n}\|_{L^2(Q_n)}}\right] } \mathrm{d}Q_{n} (\x),$$
and, via Fubini's theorem, 
$$  A_n=\E_{P, P'\sim \mathbf{P}}\left[  \frac{ \left\langle B_{P,n}, B_{P',n} \right\rangle_{L^2(Q_n)}}{\|B_{P,n}\|_{L^2(Q_n)} \|B_{P',n}\|_{L^2(Q_n)}}\right].$$
Since 
\begin{equation*}
	\left|   C_n(P,P') \right|
	:=
	\left|
	\frac{ \left\langle B_{P,n}, B_{P',n} \right\rangle_{L^2(Q_n)}}{\|B_{P,n}\|_{L^2(Q_n)} \|B_{P',n}\|_{L^2(Q_n)}}
	\right|
	\leq \frac{ \| B_{P,n}\|_{L^2(Q_n)}\| B_{P',n} \|_{L^2(Q_n)}}{\|B_{P,n}\|_{L^2(Q_n)} \|B_{P',n}\|_{L^2(Q_n)}} = 1,
\end{equation*}
the dominated convergence theorem can be applied and we only need to show that 
\begin{equation}
	\label{eq:Cn}
	C_n(P,P') \longrightarrow   1,\quad \text{for } {\bf P }\otimes {\bf P}-{\rm a.e.}\  (P,P').
\end{equation}
We decompose $C_n(P,P')$ in two terms: $C_n(P,P')=C_{n,1}(P,P')+C_{n,2}(P,P')$ with
$$ C_{n,1}(P,P')= \frac{ \|B_{P,n}\|_{L^2(Q_n)}^2} {\|B_{P,n}\|_{L^2(Q_n)}^2}=1 , $$
and
$$ C_{n,2}(P,P')=  \left\langle \frac{  B_{P,n}}{\|B_{P,n}\|_{L^2(Q_n)}},  \frac{  B_{P',n}}{\|B_{P',n}\|_{L^2(Q_n)}} -\frac{  B_{P,n}}{\|B_{P,n}\|_{L^2(Q_n)}}  \right\rangle_{L^2(Q_n)}  . $$
The goal, of course, is to show that $C_{n,2}(P,P')\to 0$, for ${\bf P }\otimes {\bf P}$-a.e. $(P,P')$.  Since 
\begin{align*}
	&C_{n,2}(P,P')\\
	&= \left\langle \frac{  B_{P,n}}{\|B_{P,n}\|_{L^2(Q_n)}},  \frac{  B_{P',n}-B_{P,n}}{\|B_{P',n}\|_{L^2(Q_n)}} +  B_{P,n}\left(\frac{1}{\|B_{P',n}\|_{L^2(Q_n)}}-\frac{1}{\|B_{P,n}\|_{L^2(Q_n)}}  \right)\right\rangle_{L^2(Q_n)} \\
	&=  \left\langle \frac{  B_{P,n}}{\|B_{P,n}\|_{L^2(Q_n)}},  \frac{  T_{Q_n,P}-T_{Q_n, P'}}{\|B_{P',n}\|_{L^2(Q_n)}} +  B_{P,n}\left(\frac{1}{\|B_{P',n}\|_{L^2(Q_n)}}-\frac{1}{\|B_{P,n}\|_{L^2(Q_n)}}  \right)\right\rangle_{L^2(Q_n)} ,
\end{align*}
we can upper bound $|C_{n,2}(P,P')|$ by 
\begin{multline} \label{AN2}
	\frac{  \|B_{P,n}\|_{L^2(Q_n)}}{\|B_{P,n}\|_{L^2(Q_n)}}  \frac{  \|T_{Q_n,P}\|_{L^2(Q_n)}+\|T_{Q_n, P'}\|_{L^2(Q_n)}}{\|B_{P',n}\|_{L^2(Q_n)}}  \\
	+  \left\vert   \left\langle \frac{  B_{P,n}}{\|B_{P,n}\|_{L^2(Q_n)}},B_{P,n}\left(\frac{1}{\|B_{P,n}\|_{L^2(Q_n)}}-\frac{1}{\|B_{P',n}\|_{L^2(Q_n)}}  \right)\right\rangle_{L^2(Q_n)} \right\vert. 
\end{multline}
The first term of \eqref{AN2} tends to $0$ for ${\bf P }\otimes {\bf P}$-a.e. $(P,P')$. Indeed, using the equality $ \|T_{Q_n,P}\|_{L^2(Q_n)}^2 = \int \|\x\|^2 \mathrm{d}P(\x) $, the first term of \eqref{AN2} is equal to 
\begin{equation}
	\label{UpperFirstTerm}
	\frac{ \sqrt{\int \|\x\|^2 \mathrm{d}P(\x)}+\sqrt{\int \|\x\|^2 \mathrm{d}P'(\x)}}{\|B_{P',n}\|_{L^2(Q_n)}}.
\end{equation}
The latter clearly tends to $0$ since $\|B_{P',n}\|_{L^2(Q_n)}= \mathcal{W}_2(Q_n,P')$. 

To show that  the second term of \eqref{AN2} also tends to $0$ we use the bound 
\begin{multline*}
	\left\vert   \left\langle \frac{  B_{P,n}}{\|B_{P,n}\|_{L^2(Q_n)}},B_{P,n}\left(\frac{1}{\|B_{P,n}\|_{L^2(Q_n)}}-\frac{1}{\|B_{P',n}\|_{L^2(Q_n)}}  \right)\right\rangle_{L^2(Q_n)} \right\vert\\
	\leq \left\vert    \|B_{P,n}\| \left(\frac{\|B_{P,n}\|_{L^2(Q_n)}-\|B_{P',n}\|_{L^2(Q_n)} }{\|B_{P,n}\|_{L^2(Q_n)} \|B_{P',n}\|_{L^2(Q_n)}}  \right) \right\vert
\end{multline*}
followed by the  triangle inequality
\begin{align*}
	\left\vert    \|B_{P,n}\| \left(\frac{\|B_{P,n}\|_{L^2(Q_n)}-\|B_{P',n}\|_{L^2(Q_n)} }{\|B_{P,n}\|_{L^2(Q_n)} \|B_{P',n}\|_{L^2(Q_n)}}  \right) \right\vert &=   \left|\frac{\|B_{P,n}\|_{L^2(Q_n)}-\|B_{P',n}\|_{L^2(Q_n)} }{\|B_{P',n}\|_{L^2(Q_n)}}  \right|\\
	&\leq   \left(\frac{\|B_{P,n}-B_{P',n}\|_{L^2(Q_n)} }{\|B_{P',n}\|_{L^2(Q_n)}}  \right)\\
	&=   \left(\frac{\|T_{Q_n,P}-T_{Q_n,P'}\|_{L^2(Q_n)} }{\|B_{P',n}\|_{L^2(Q_n)}}  \right).
\end{align*}
The latter can be upper bounded by \eqref{UpperFirstTerm}, so that the second term of \eqref{AN2} also tends to $0$ for ${\bf P }\otimes {\bf P}$-a.e. $(P,P')$. This implies $C_{n,2}(P,P')$ tends to $0$ for ${\bf P }\otimes {\bf P}$-a.e. $(P,P')$. Hence, \eqref{eq:Cn} holds. 

\section{Proof of Lemma \ref{lemma:geo-inva}}
Since, from \eqref{eq:unique:geodesic},
we have
$T_{Q,\gamma^{Q\to P}_\lambda}(\x)= (1-\lambda)\x+\lambda T_{Q,P}(\x) $, we have
\begin{align*}
	\int{ \left\|\E_{P\sim \mathbf{P}_{\lambda,Q} }\left[  \frac{\x-T_{Q,P}(\x) }{\mathcal{W}_2(P,Q)}\right]\right\|^2} dQ(\x) &=\int{ \left\|\E_{P\sim \mathbf{P}}\left[  \frac{\x-T_{Q,\gamma^{Q\to P}_\lambda}(\x) }{\|I-T_{Q,\gamma^{Q\to P}_\lambda}\|_{L^2(Q)}}\right]\right\|^2} dQ(\x)  \\
	&=\int{ \left\|\E_{P\sim \mathbf{P}}\left[  \frac{ \x- T_{Q,P}(\x) }{\|I-T_{Q,P}\|_{L^2(Q)}}\right]\right\|^2} dQ(\x) \\
	&= \int{ \left\|\E_{P\sim \mathbf{P} }\left[  \frac{\x-T_{Q,P}(\x) }{\mathcal{W}_2(P,Q)}\right]\right\|^2} dQ(\x) ,
\end{align*}
which concludes the proof.

\section{Proof of Theorem~\ref{Theo:RelationSpatial}}
Consider $P \in \{ P_1 , \ldots , P_n \}$ and $Q \in \mathcal{P}_{2}^{a.c}(\R^d)$ such that $Q \not \in \{ P_1 , \ldots , P_n \}$.
We recall from \cite{villani2003topics} that 
\begin{equation}\label{WassertesinAppen}
	\mathcal{W}_2^2(P,Q)=\inf_{\pi\in \Pi(P,Q) } \frac{1}{2} \int \|\x-\y\|^2 \mathrm{d}\pi(\x,\y)
\end{equation}
admits a dual formulation 
\begin{equation}
	\label{dual:}
	\mathcal{W}_2^2(P,Q) =\sup_{(f,g)\in \Phi} \left\{ \int f(\x) \, \mathrm{d}Q(\x) + \int g(\y) \, \mathrm{d}P(\y) \right\},
\end{equation}
where $ \Phi=\{ (f,g)\in \mathcal{C}(\R^d)\times \mathcal{C}(\R^d): 
\  f(\x)+g(\y)\leq\frac{1}{2}\|\x-\y\|^2 \}$.
Here $\mathcal{C}(\R^d)$ is the set of continuous functions on $\R^d$. 
We denote as $(f_{Q,P},f_{P,Q})$ the solutions of \eqref{dual:}. It is well-known that $\nabla f_{Q,P}(\x)=\x-T_{Q,P}(\x)$.  
Now we argue by contradiction. 
We assume first that there exists 
$$ Q\in  \mathcal{P}_2^{a.c} (\R^d) \cap \argmin_{Q'}  \E_{P\sim \mathbf{P}}[\mathcal{W}_2(P,Q')] 
$$
with $Q \not \in \{ P_1 , \ldots , P_n \}$
and we assume that the set $\mathcal{K}'$ of all $\x$ such that 
$$
s(\x):=\E_{P\sim \mathbf{P}}\left[  \frac{\x-T_{Q,P}(\x) }{\mathcal{W}_2(P,Q)}\right]\neq 0 
$$
has positive measure $Q(\mathcal{K}')>0$. 
As $T_{Q,P}$ is the gradient of a lower semi continuous convex function, it is continuous $Q$-a.e., so that $s$ is also continuous $Q$-a.e. Therefore, there exists a compact convex set with non-empty interior $U$ such that   $U\subset \mathcal{K}'$.   
Consider the signed measure $h$ such that 
$\frac{d h}{dQ}=-{\bf 1}_U \left(f_{Q,P} -
\frac{1}{Q(U)}
\int_U f_{Q,P}(\z) \mathrm{d}Q(\z) \right) $, where ${\bf 1}_{U} $ is the indicator function of the set $U$.

Note that $h(\R^d)=0$ and $Q+t h$ is a probability measure with  finite second order moment for all $t$ in a neighborhood of zero. Since $(\cdot)^{1/2}$ is concave, 
$$  \mathcal{W}_2(P,Q+th)\leq  \mathcal{W}_2(P,Q) + \frac{\mathcal{W}_2^2(P,Q+th)-\mathcal{W}_2^2(P,Q)}{2\mathcal{W}_2(P,Q)}.  $$
Using the dual formulation \eqref{dual:} we obtain for $t$ in a neighborhood of zero,
$$  \frac{\mathcal{W}_2(P,Q+th)-\mathcal{W}_2(P,Q)}{t}\leq    \frac{\int f_{Q+th,P}(\x) \, \mathrm{d}h(\x)}{2\mathcal{W}_2(P,Q)}.  $$
Since $h(\R^d)=0$, we have for $t$ in a neighborhood of zero
\begin{multline*}
	\frac{\mathcal{W}_2(P,Q+th)-\mathcal{W}_2(P,Q)}{t}\\
	\leq    -\frac{ \int_U \left( f_{Q,P}(\x)-\frac{1}{Q(U)} \int_U f_{Q,P}(\z) \mathrm{d}Q(\z) \right) \left(f_{Q+th,P}(\x)-\frac{1}{Q(U)} \int_U f_{Q+th,P}(\z) \mathrm{d}Q(\z)\right) \mathrm{d}Q(\x) }{2\mathcal{W}_2(P,Q)}.   
\end{multline*}
Set 
$$ M(P):=\frac{1}{2}\frac{\int_U \left(f_{Q,P}(\x)-\frac{1}{Q(U)}\int_U f_{Q,P}(\z) \mathrm{d}Q(\z) \right)^2 \mathrm{d}Q(\x) }{\mathcal{W}_2(P,Q)}$$
and the norm 
$$ \|\phi\|_{U}:= \left(\int_U \left(\phi(\x)-\frac{1}{Q(U)}\int_U \phi(\z) \mathrm{d}Q(\z)\right)^2 \mathrm{d}Q(\x)\right)^{\frac{1}{2}}. $$
Then 
\[\frac{\mathcal{W}_2(P,Q+th)-\mathcal{W}_2(P,Q)}{t}\leq    - M(P)+\frac{ \| f_{Q,P}\|_U \| f_{Q,P}-f_{Q+th,P}\|_U }{2\mathcal{W}_2(P,Q)}.\]
Since $s(\x)\neq 0$ for $\x \in U$, the function $U\ni \x\mapsto \E_{P \sim {\bf P}}[f_{Q,P}(\x)]$ is non constant, which implies that  
$$  \E_{P\sim {\bf P}}[M(P)]:=\frac{1}{2}\E_{P\sim {\bf P}}\left[\frac{\int_U \left(f_{Q,P}(\x)-\frac{1}{Q(U)} \int_U f_{Q,P}(\z) \mathrm{d}Q(\z) \right)^2 \mathrm{d}Q(\x) }{\mathcal{W}_2(P,Q)}\right]>0.  $$
The theorem follows upon showing that 
\begin{equation}
	\E_{P\sim {\bf P}}\left[ \frac{ \| f_{Q,P}\|_U \| f_{Q,P}-f_{Q+th,P}\|_U }{\mathcal{W}_2(P,Q)} \right]\to 0\quad {\rm as} \ t\to 0,
\end{equation}
which is a trivial consequence of the main result of \cite{Segers2022GraphicalAU} and the assumption $Q \not \in \{ P_1 , \ldots , P_n \}$. 

\section{Proofs of Section \ref{subsection:robustness}}

\begin{proof}[Proof of Lemma~\ref{lemma:breakdown}]
	First we show the lower bound. Fix $\alpha\in (0,1]$ and assume that $\ell$ is such that
	\begin{equation}
		\label{Escaping:horizon}
		d_H \left(\mathcal{R}(\alpha; {\bf P}_n) ,\mathcal{R}(\alpha;{\bf Q}_{m})\right) \to \infty
	\end{equation}
	for some sequence $\{{\bf Q}_{m}\}_m\subset \mathcal{P}({\ell,n}) $. Without loosing generality we assume that the perturbed points are $P_{n-\ell+1}, \dots, P_n$. Since the  WSD vanishes at infinity  (cf.~Theorem~\ref{Theo:Main}), it follows that $\mathcal{R}(\alpha; {\bf P}_n)$ is bounded. Hence,  \eqref{Escaping:horizon} is equivalent to the existence of a  sequence $Q_m\in \mathcal{R}(\alpha;{\bf Q}_{m})$ containing a divergence subsequence, i.e., $\mathcal{W}_2(Q_{m},Q)\to \infty $ for some $Q\in \mathcal{P}_2(\R^d)$. (Note here that $Q_{m}$ diverges if and only if   $\mathcal{W}_2(Q_{m},Q)\to \infty $ for all $Q\in \mathcal{P}_2(\R^d)$.) We use the same notation for the subsequence.
	Let $\widetilde{P}_1,\ldots,\widetilde{P}_n$ be the support points of ${\bf Q}_m$ with $\widetilde{P}_1=P_1,\ldots,\widetilde{P}_{n-\ell} = P_{n - \ell}$.
	By triangle inequality, we get 
	\begin{align*}
		1-\alpha & \ge \left\|\frac{1}{n}\sum_{i=1}^n \left[  \frac{I-T_{Q_m,\widetilde{P}_i} }{\mathcal{W}_2(Q_m, \widetilde{P}_i)}\right]\right\|_{L^2(Q_m)}\\
		&\geq \left\|\frac{1}{n}\sum_{i=1}^{n-\ell} \left[  \frac{I-T_{Q_m,P_i} }{\mathcal{W}_2(Q_m, P_i)}\right]\right\|_{L^2(Q_m)}- \left\|\frac{1}{n}\sum_{i = n-\ell +1}^n \left[  \frac{I-T_{Q_m,\widetilde{P}_i} }{\mathcal{W}_2(Q_m, \widetilde{P}_i)}\right]\right\|_{L^2(Q_m)}.
	\end{align*}
	Since $Q_m$ diverges, Theorem~\ref{Theo:Main} (vanishing at infinity and values in $[0,1]$ properties) implies 
	$$ \lim_{m \to \infty} \left\|\frac{1}{n}\sum_{i=1}^{n-\ell} \left[  \frac{I-T_{Q_m,P_i} }{\mathcal{W}_2(Q_m, P_i)}\right]\right\|_{L^2(Q_m)}= \frac{n-\ell}{n} $$
	and 
	$$ \left\|\frac{1}{n}\sum_{i = n-\ell +1}^n \left[  \frac{I-T_{Q_m,\widetilde{P}_i} }{\mathcal{W}_2(Q_m, \widetilde{P}_i)}\right]\right\|_{L^2(Q_m)} \leq \frac{\ell}{n}.$$
	Hence, we get 
	$$ 1-\alpha \geq  \frac{n-\ell-\ell}{n}=1-\frac{2\ell}{n}, $$
	which yields the lower bound.

	Consider now the lower bound and fix $\alpha\in (0,1 - \frac{2}{n} ]$.
	We define the group $\{S_{m}\}_{m\in \R}$ of Euclidean isometries, with $S_{m}(\x)=\x+m \u$, for some $\u\in \R^d$ with $\|\u\|=1$. Note that for each $m\in \mathbb{R}$, it induces the isometry $\mathbb{S}_{m}(P)=(S_{m})_\# P $ over $\mathcal{P}_2(\R^d)$. Next, fix $\ell\in \{1, \dots, n\}$, with $\ell \le n/2$.
	Fix also $Q\in \Ptac(\R^d)$. Now we construct the sequence of adversarial samples. For each $m\in \N$, we exchange  $P_{n-\ell+1}, \dots, P_n$ by $P^{m}_{1}, \dots, P^{m}_\ell$, where $P^m_j= \mathbb{S}_{m}(P_{n-\ell+j})$ for $j=1, \dots, \ell$. We call $\mathbf{P}^{m}_n$ the empirical measure of the data $$ P_1, \dots, P_{n-\ell},  P^{m}_{1}, \dots, P^{m}_\ell.$$
	We will increase $ m $ and show that $ \mathbb{S}_{\frac{m}{2}}(Q) \in \mathcal{R}(\alpha;{\bf P}_n^m)$ for $m$ large enough and $\alpha <   \frac{ 2\ell}{n} $. Since $\{\mathbb{S}_{\frac{m}{2}}(Q)\}_{m\in \N}$ diverges, this will finish the proof. The fact that WSD is transformation invariant (Theorem \ref{Theo:Main}) implies that, for every $Q\in \mathcal{P}_2^{a.c}(\R^d)$ and $m\in \N$, 
	$$  \rm WSD(\mathbb{S}_{\frac{m}{2}}(Q); {\bf P}_n^m) =   \rm WSD(Q; (\mathbb{S}_{-\frac{m}{2}})_\# {\bf P}_n^m),  $$
	where, by definition,  $(\mathbb{S}_{-\frac{m}{2}})_\# {\bf P}_n^m$ is the empirical measure of 
	$$ (\mathbb{S}_{-\frac{m}{2}}(P_1), \dots, \mathbb{S}_{-\frac{m}{2}}(P_{n-\ell}),  \mathbb{S}_{\frac{m}{2}}(P_{n - \ell +1}), \dots, \mathbb{S}_{\frac{m}{2}}(P_n)).$$
	Since 
	$ T_{Q, \mathbb{S}_{\frac{m}{2}}(P) } (\x)  =  T_{Q,P } (\x)+ \frac{m}{2}\u $ 
	and 
	$ T_{Q, \mathbb{S}_{-\frac{m}{2}}(P) } (\x)  =  T_{Q,P } (\x)-\frac{m}{2}\u, $
	we get  
	\begin{align*}
		1- & \rm WSD(\mathbb{S}_{\frac{m}{2}}(Q); {\bf P}_n^m)&\\
		& =  \frac{1}{n}\left\|\sum_{i=1}^{n-\ell}  \frac{I+\frac{m}{2} \u -T_{Q, P_i }  }{\| I+\frac{m}{2} \u -T_{Q, P_i } \|_{L^2(Q)} }+ \sum_{i=1}^\ell  \frac{I-\frac{m}{2} \u -T_{Q, P_i } }{\| I-\frac{m}{2} \u -T_{Q, P_i } \|_{L^2(Q)}}\right\|_{L^2(Q)}\\
		& \leq \frac{1}{n}\left\|\sum_{i=1}^\ell  \frac{I+\frac{m}{2} \u -T_{Q, P_i }  }{\| I+\frac{m}{2} \u -T_{Q, P_i } \|_{L^2(Q)} }+ \frac{I-\frac{m}{2} \u -T_{Q, P_i } }{\| I-\frac{m}{2} \u -T_{Q, P_i } \|_{L^2(Q)}}\right\|_{L^2(Q)}+ \frac{n-2\ell}{n}.
	\end{align*}
	Note that, for every $i$, 
	\begin{multline*}
		\left\|  \frac{I+\frac{m}{2} \u -T_{Q, P_i }  }{\| I+\frac{m}{2} \u -T_{Q, P_i } \|_{L^2(Q)} }+ \frac{I-\frac{m}{2} \u -T_{Q, P_i } }{\| I-\frac{m}{2} \u -T_{Q, P_i } \|_{L^2(Q)}}\right\|_{L^2(Q)} \\
		= \left\|  \frac{\frac{2}{m}I+ \u -\frac{2}{m}T_{Q, P_i }  }{\| \frac{2}{m}I+ \u -\frac{2}{m}T_{Q, P_i } \|_{L^2(Q)} }+ \frac{\frac{2}{m}I-\u -\frac{2}{m}T_{Q, P_i } }{\| \frac{2}{m}I- \u -\frac{2}{m}T_{Q, P_i } \|_{L^2(Q)}}\right\|_{L^2(Q)} \to 0,
	\end{multline*}
	so that, for every $\epsilon >0$,
	$ \rm WSD(\mathbb{S}_{\frac{m}{2}}(Q); {\bf P}_n^m) \geq  \frac{ 2\ell}{n}  - \epsilon $ for $m$ large enough and the result follows.
\end{proof}

\begin{proof}[Proof of Lemma~\ref{lemma:influence}]
	First note that 
	\begin{multline*}
		\E_{P\sim \mathbf{P}+t (\delta_\mu -{\bf P})}\left[  \frac{\x-T_{Q,P}(\x) }{\mathcal{W}_2(P,Q)}\right]=\underbrace{\E_{P\sim \mathbf{P}}\left[  \frac{\x-T_{Q,P}(\x) }{\mathcal{W}_2(P,Q)}\right]}_{=:A} \\
		+ t \underbrace{\left(  \frac{\x-T_{Q,\mu}(\x) }{\mathcal{W}_2(\mu,Q)}- \E_{P\sim \mathbf{P}}\left[  \frac{\x-T_{Q,P}(\x) }{\mathcal{W}_2(P,Q)}\right] \right)}_{=:H}.
	\end{multline*}
	Since the norm in a Hilbert space admits directional derivatives (except at zero), if  ${\rm WSD}(Q;{\bf P})\neq 1$ (or equivalently, if $A\neq 0$), we get 
	$${\rm IC}( \mu, {\rm WSD}(Q;{\bf P} ))= - \frac{d}{d t } \bigg\vert_{t=0}\|A+ t H\|_{L^2(Q)} =-\frac{\langle A,H\rangle_{L^2(Q)}}{\|A\|_{L^2(Q)}} $$
	and, if ${\rm WSD}(Q;{\bf P})= 1$ (or equivalently, if $A=0$), 
	$${\rm IC}( \mu, {\rm WSD}(Q;{\bf P} ))= - \frac{d}{d t } \bigg\vert_{t=0} t\| H\|_{L^2(Q)} = -\| H\|_{L^2(Q)} .$$
	Hence, the result follows.
\end{proof}

\section{Proofs of Section \ref{Sec:Continuity}}
\begin{proof}[Proof of Theorem~\ref{Theo:COnti}] 
	As $\mathbf{P}\in \mathcal{P}(\mathcal{P}_{2}(\R^d)) $ is atomless there exists an open  Wasserstein ball
	$$ \mathbb{B}_{\mathcal{W}_2}(Q,\beta)=\{P\in \mathcal{P}_{2}(\R^d): \ \mathcal{W}_2(P,Q)< \beta\}$$
	with ${\bf P}(\mathbb{B}_{\mathcal{W}_2}(Q,\beta))\leq \epsilon/2$. Since ${\bf P} \in  \mathcal{P}(\mathcal{P}_{2}(\R^d)) $ is tight, there exists a compact set $K\subset \mathcal{P}_{2}(\R^d)$ such that 
	${\bf P}(\mathcal{P}_{2}(\R^d)\setminus K) \leq \epsilon/2 $.
	Set 
	$\mathcal{V}_\beta= K\cap (\mathcal{P}_{2}(\R^d)\setminus \mathbb{B}_{\mathcal{W}_2}(Q,\beta))$ and $\mathcal{V}_\beta^c= \mathcal{P}_{2}(\R^d)\setminus \mathcal{V}_\beta$. In summary, it holds that 
	\begin{equation}
		\label{epsilonResidual}
		{\bf P}(\mathcal{V}_\beta^c) \leq \epsilon.
	\end{equation}
	Moreover, as $\mathcal{W}_2(Q_n,Q) \to 0 $, we can assume that $n$ is large enough such that $\mathcal{W}_2(Q_n,Q)\leq \beta/2$, which implies that
	\begin{equation}
		\label{eq:farfromBeta}
		\mathcal{W}_2(Q_n,P)\geq \mathcal{W}_2(P,Q) -\mathcal{W}_2(Q_n,Q) \geq  \beta/2,
	\end{equation}
	for all $P\in \mathcal{V}_\beta$.  Next, call 
	$$ A_n^2=\int{ \left\|\E_{P\sim \mathbf{P}}\left[  \frac{\x-T_{Q_n,P}(\x) }{\mathcal{W}_2(P,Q_n)}\right]\right\|^2} \mathrm{d}Q_n(\x) $$
	and 
	$$ A^2=\int{ \left\|\E_{P\sim \mathbf{P}}\left[  \frac{\x-T_{Q,P}(\x) }{\mathcal{W}_2(P,Q)}\right]\right\|^2} \mathrm{d}Q(\x). $$
	The result follows by showing that $ A_n^2\to A^2$. Triangle inequality implies that 
	\begin{multline*}
		\left\vert A_n-\left(\underbrace{\int{ \left\|\E_{P\sim \mathbf{P}}\left[ {\bf 1}_{\mathcal{V}_\beta}(P) \frac{\x-T_{Q_n,P}(\x) }{\mathcal{W}_2(P,Q_n)}\right]\right\|^2} \mathrm{d}Q_n(\x)}_{=:B_n^2} \right)^{\frac{1}{2}} \right|\\
		\leq \left(\int{ \left\|\E_{P\sim \mathbf{P}}\left[ {\bf 1}_{\mathcal{V}_\beta^c}(P) \frac{\x-T_{Q_n,P}(\x) }{\mathcal{W}_2(P,Q_n)}\right]\right\|^2} \mathrm{d}Q_n (\x) \right)^{\frac{1}{2}},
	\end{multline*}
	so that, arguing as in Section~\ref{subsection:liesIn01} and using \eqref{epsilonResidual}, we derive the bound 
	$ |A_n-B_n|\leq \epsilon $ for all $n\in \N$. By the same means 
	$ |A-B|\leq \epsilon $ where 
	$$ B^2=\int{ \left\|\E_{P\sim \mathbf{P}}\left[ {\bf 1}_{\mathcal{V}_\beta}(P) \frac{\x-T_{Q,P}(\x) }{\mathcal{W}_2(P,Q)}\right]\right\|^2} \mathrm{d}Q(\x).$$
	Therefore, since   $ \epsilon$ is arbitrary,  the result follows after showing that $B_n\to B$. To do so, we set $\X_n\sim Q_n$ for $n\in \N$, $\X \sim Q$ 
	and $P, P'\in \mathcal{P}_2(\R^d)$.
	Arguing as in the proof of Theorem~2.1 in \cite{BodhisattvaDeb} we get for every $P,P' \in \mathcal{P}_2(\R^d)$, 
	\begin{equation} \label{eq:from:thm:coupling}
		(\X_n,T_{Q_n,P}(\X_n),T_{Q_n,P'}(\X_n))
		\xrightarrow{w} 
		(\X,T_{Q,P}(\X),T_{Q,P'}(\X)).
	\end{equation}
	Indeed a straightforward adaptation of the arguments there shows first that there is a limit in distribution which is the distribution of the random vector
	\[
	( \Z_1,\Z_2,\Z_3),
	\]
	where of course we have $\Z_1 \sim Q$. Then the arguments there show that $(\X_n,T_{Q_n,P}(\X_n)) \xrightarrow{w}  (\X,T_{Q,P}(\X))$ and thus a.s.
	\[
	\Z_2 = T_{Q,P}(\Z_1).
	\]
	Similarly,
	\[
	\Z_3 = T_{Q,P'}(\Z_1)
	\]
	and thus \eqref{eq:from:thm:coupling} holds. The continuous mapping theorem with the function 
	$ (\x, \y,\z)\mapsto (\y-\x, \z-\x)$ implies that
	$$ 
	\begin{pmatrix}
		T_{Q_n,P}(\X_n) -\X_n \\
		T_{Q_n,P'}(\X_n) -\X_n
	\end{pmatrix}
	\xrightarrow{w} 
	\begin{pmatrix}
		T_{Q,P}(\X) -\X  \\
		T_{Q,P'}(\X) -\X 
	\end{pmatrix}.
	$$
	Since for all $P\in \mathcal{P}_2(\R^d)$, it holds that $\mathcal{W}_2(Q_n,P)\to \mathcal{W}_2(Q,P) $, Slutsky's theorem
	yields 
	\begin{equation}
		\label{weakLimitQn}
		\begin{pmatrix}
			\frac{T_{Q_n,P}(\X_n) -\X_n}{\mathcal{W}_2(Q_n,P)} 
			\\
			\frac{T_{Q_n,P'}(\X_n) -\X_n}{\mathcal{W}_2(Q_n,P')} 
		\end{pmatrix}
		\xrightarrow{w}
		\begin{pmatrix}
			\frac{T_{Q,P}(\X) -\X}{\mathcal{W}_2(Q,P)} \\
			\frac{T_{Q,P'}(\X) -\X}{\mathcal{W}_2(Q,P')}
		\end{pmatrix}
	\end{equation}
	for  all $ P , P' $ such that $\mathcal{W}_2(Q,P)>0$ and $\mathcal{W}_2(Q,P')>0$. As a consequence, \eqref{weakLimitQn} holds for ${\bf P}$-a.e. $P,P'$.  
	
	Let ${\bf P}_\beta$ be the probability measure  $A\mapsto {\bf P}_\beta(A)= \frac{{\bf P}(\mathcal{V}_\beta\cap A)}{ {\bf P}(\mathcal{V}_\beta)}$.
	Therefore, for $(P,P')\sim {\bf P}_\beta\otimes {\bf P}_\beta$ with $({P},P')$ independent of $ \{\X_n\}_{n\in \N} $, we obtain 
	$$
	{\Y}_n:= \frac{ \left\langle\X_n-T_{Q_{n} ,P}(\X_n), \X_n-T_{Q_{n} ,P'}(\X_n)  \right\rangle}{\mathcal{W}_2(P,Q_{n}) \mathcal{W}_2(P',Q_{n})}\xrightarrow{w} {\Y}:=\frac{ \left\langle \X-T_{Q ,P}(\X), \X-T_{Q ,P'}(\X)  \right\rangle}{\mathcal{W}_2(P,Q) \mathcal{W}_2(P',Q)}.
	$$
	Indeed, for a bounded continous function $F: \R \to \R$,
	\begin{align*}
		\E \left[ 
		F({\Y}_n)
		\right]
		&= 
		\E \left[ 
		\E \left[ 
		\left.
		F({\Y}_n)
		\right| 
		P,P'
		\right]
		\right]
		\\
		&=  \int 
		\int 
		\E \left[ 
		\left.
		F({\Y}_n)
		\right| 
		P=\tilde{P},P'=\tilde{P}'
		\right]
		\mathrm{d} {\bf P}_\beta(\tilde{P})
		\mathrm{d} {\bf P}_\beta(\tilde{P}')
		\\ 
		&= 
		\int 
		\int 
		\E \left[ 
		F
		\left(  
		\frac{ \left\langle\X_n-T_{Q_{n} ,\tilde{P}}(\X_n), \X_n-T_{Q_{n} ,\tilde{P}'}(\X_n)  \right\rangle}{\mathcal{W}_2(\tilde{P},Q_{n}) \mathcal{W}_2(\tilde{P}',Q_{n})}
		\right)
		\right]
		\mathrm{d} {\bf P}_\beta(\tilde{P})
		\mathrm{d} {\bf P}_\beta(\tilde{P}')
		\\
		&\underset{n \to \infty}{\longrightarrow} 
		\int 
		\int 
		\E \left[ 
		F
		\left(  
		\frac{ \left\langle\X-T_{Q ,\tilde{P}}(\X), \X-T_{Q ,\tilde{P}'}(\X)  \right\rangle}{\mathcal{W}_2(\tilde{P},Q) \mathcal{W}_2(\tilde{P}',Q)}
		\right)
		\right]
		\mathrm{d} {\bf P}_\beta(\tilde{P})
		\mathrm{d} {\bf P}_\beta(\tilde{P}')
		\\
		&= 
		\E[F(\Y)],
	\end{align*}
	where the above limit holds due to dominated convergence.
	
	Skorokhod's representation theorem yields the existence of a sequence of random variables $\{\tilde{\Y}_n\}$ defined on a common probability space $(\Omega', \mathcal{A}',\mathbb{P}')$ taking values in $\R$  with $\tilde{\Y}_n\overset{d}{=}{\Y}_n$ converging $\mathbb{P}'$-a.e. to a random variable  $\tilde{\Y}:\Omega'\to \R^d$ with  $\tilde{\Y}\overset{d}{=} \Y$. 
	Since
	\begin{align*}
		B_n^2&={\bf P}(\mathcal{V}_\beta)^2 {\E\left[  \frac{ \left\langle\X_n-T_{Q_{n} ,P}(\X_n), \X_n-T_{Q_{n} ,P'}(\X_n)  \right\rangle}{\mathcal{W}_2(P,Q_{n}) \mathcal{W}_2(P',Q_{n})}\right] }= {\bf P}(\mathcal{V}_\beta)^2 \E[\tilde{\Y}_n]
	\end{align*}
	and 
	\begin{align*}
		B^2&={\bf P}(\mathcal{V}_\beta)^2 {\E\left[  \frac{ \left\langle \X-T_{Q ,P}(\X), \X-T_{Q ,P'}(\X)  \right\rangle}{\mathcal{W}_2(P,Q) \mathcal{W}_2(P',Q)}\right] }={\bf P}(\mathcal{V}_\beta)^2 \E[\tilde{\Y}],
	\end{align*}
	we only need to prove that $\Y_n$ is uniformly integrable. The bound \eqref{eq:farfromBeta} implies that it is enough to show that each of the terms of the right hand side of 
	\begin{multline}\label{eq:decompositionUnfiormInt}
		|\left\langle\X_n-T_{Q_{n} ,P}(\X_n), \X_n-T_{Q_{n} ,P'}(\X_n)  \right\rangle|\\
		\leq \|\X_n\|^2+ \|T_{Q_{n} ,P}(\X_n)\|\|\X_n\|+\|T_{Q_{n} ,P'}(\X_n)\|\|\X_n\|+\|T_{Q_{n} ,P}(\X_n)\|\|T_{Q_{n} ,P'}(\X_n)\|
	\end{multline}
	are uniformly integrable. Recall that a set $S$ of random variables is uniformly integrable if 
	$$ \lim_{R\to +\infty}\sup_{U\in S}\E[|U|{\bf 1}_{|U|>R}] =0 . $$
	Since $\mathcal{V}_\beta$ and $\{Q_n\}_{n\in \N}$  are relatively compact subsets in the 2-Wasserstein topology, 
	Theorem~7.12 in \cite{villani2003topics} implies that 
	\begin{equation}
		\label{TightWasserst}
		\lim_{R\to +\infty}\sup_{P\in \mathcal{V}_\beta}\int_{{ \| \x\|^2>R}} \|\x\|^2 {\rm d} P(\x) =0
	\end{equation}
	and 
	\begin{equation}
		\label{TightWasserstQn}
		\lim_{R\to +\infty}\sup_{n\in \N}\int_{{ \| \x\|^2>R}} \|\x\|^2 {\rm d} Q_n(\x) =0.
	\end{equation}
	The last limit \eqref{TightWasserstQn} implies that the sequence $\{\|\X_n\|^2\}_{n\in \N}$ is uniformly integrable, so that the first term of  the right-hand-side of \eqref{eq:decompositionUnfiormInt} is uniformly integrable. For the second, we observe that 
	\begin{align*}
		\E[&\|T_{Q_{n} ,P}(\X_n)\|\|\X_n\|{\bf 1}_{\|T_{Q_{n} ,P}(\X_n)\|\|\X_n\|>R}] \\
		&\leq \E\left[\|T_{Q_{n} ,P}(\X_n)\|\|\X_n\|{\bf 1}_{\|\X_n\|>R^{\frac{1}{2}}}\right]+\E\left[\|T_{Q_{n} ,P}(\X_n)\|\|\X_n\|{\bf 1}_{\|T_{Q_{n} ,P}(\X_n)\|>R^{\frac{1}{2}}}\right]\\
		&\leq \left(\E\left[\|T_{Q_{n} ,P}(\X_n)\|^2\right]\E\left[\|\X_n\|^2{\bf 1}_{\|\X_n\|>R^{\frac{1}{2}}}\right]\right)^{\frac{1}{2}}\\
		&\qquad \qquad+\left(\E\left[\|T_{Q_{n} ,P}(\X_n)\|^2 {\bf 1}_{\|T_{Q_{n} ,P}(\X_n)\|>R^{\frac{1}{2}}}\right] \E\left[\|\X_n\|^2\right]\right)^{\frac{1}{2}}\\
		&\leq  \left(\sup_{P\in \mathcal{V}_\beta}\int \|\x\|^2 {\rm d} P(\x)\int_{\|\x\|^2> R} \|\x\|^2 {\rm d} Q_n(\x) \right)^{\frac{1}{2}}\\
		&\qquad \qquad+\left(\sup_{P\in \mathcal{V}_\beta}\int_{\|\x\|^2> R} \|\x\|^2 {\rm d} P(\x) \int \|\x\|^2 {\rm d} Q_n(\x)\right)^{\frac{1}{2}},
	\end{align*}
	where we used the fact that $T_{Q_{n} ,P}(\X_n)\sim P$ for all $n\in \N$.  Since, 
	$ \sup_{n\in \N}\int \|\x\|^2 {\rm d} Q_n(\x)  $ and $\sup_{P\in \mathcal{V}_\beta}\int \|\x\|^2 {\rm d} P(\x)$ are bounded, the previous display,  \eqref{TightWasserst} and \eqref{TightWasserstQn} imply that the second  term of \eqref{eq:decompositionUnfiormInt} is uniformly integrable. Since $P$ and $P'$ are exchangeable, the same holds for the third term. The uniform integrability of the last one follows directly from \eqref{TightWasserst}. 
\end{proof}

\begin{proof}[Proof of Lemma~\ref{Stabilty}]
	
	From \cite[Corollary~5.23]{villani2008optimal},
	for every $\epsilon>0$, it holds that 
	$Q(\| T_{Q,P_n} - T_{Q,P}\| \geq \epsilon) \to 0$. As $\| T_{Q,P_n} - T_{Q,P}\|_{L^2(Q)}$ is uniformly bounded, the sequence   $\{ T_{Q,P_n} - T_{Q,P}\}_{n\in \N}$ is compact w.r.t.\ the weak topology of $L^2(Q)$ by the Banach-Alaoglu–Bourbaki theorem (cf.\ \cite[Theorem~3.16]{Brezis}). Therefore, for each subsequence $\{ T_{Q,P_{n_k}} - T_{Q,P}\}_{k\in \N}$  there exists a further subsequence $\{ T_{Q,P_{n_{k_\ell}}} - T_{Q,P}\}_{\ell\in \N}$ such that 
	$$ \langle T_{Q,P_{n_{k_\ell}}} - T_{Q,P},h \rangle_{L^2(Q)}\to  \langle  L,h \rangle_{L^2(Q)}$$
	for some $L\in L^2(Q)$ and all $h\in L^2(Q)$. We prove now that $L=0$, irrespective of the subsequences. To improve readability, we write $\{ T_{Q,P_n} - T_{Q,P}\}_{n\in \N}$ instead of $\{ T_{Q,P_{n_{k_\ell}}} - T_{Q,P}\}_{\ell\in \N}$.   
	Since $Q(\| T_{Q,P_n} - T_{Q,P} \| \geq \epsilon) \to 0 $ and $$\|T_{Q,P_n}\|_{L^2(Q)}^2=\int \|\x\|^2 {\rm d} P_n(\x) \to  \int \|\x\|^2 {\rm d} P(\x)=\|T_{Q,P}\|^2_{L^2(Q)}<+\infty,$$ 
	Vitali convergence theorem implies that $\{ T_{Q,P_{n}} - T_{Q,P}\}_{n\in \N}$ converges to zero in the reflexive space  $L^{\frac{3}{2}}(Q)$.  Therefore,
	$0$ is also the weak limit of $ T_{Q,P_n} - T_{Q,P}$ in $L^{\frac{3}{2}}(Q)$, i.e., 
	$$ \int \langle T_{Q,P_{n}} - T_{Q,P},h \rangle dQ \to  0 $$
	for all $h\in L^{3}(Q)$. As a consequence, $L=0$,  $Q$-a.e. Moreover, 
	\begin{align*}
		\| T_{Q,P_{n}} - T_{Q,P}\|_{L^{2}(Q)}^2
		& =  \|T_{Q,P_{n}}\|_{L^{2}(Q)}^2 + \|T_{Q,P}\|_{L^{2}(Q)}^2-2 \langle T_{Q,P_{n}} , T_{Q,P}\rangle_{L^{2}(Q)}
		\\ 
		& \to  2\|T_{Q,P}\|^2-2 \langle T_{Q,P} , T_{Q,P}\rangle_{L^{2}(Q)}
		\\
		& = 0. 
	\end{align*}
	This concludes the proof.
\end{proof}

\begin{proof}[Proof of Theorem~\ref{Theorem:contP}]
	Fix $\epsilon>0$. As $\mathbf{P}\in \mathcal{P}(\mathcal{P}_{2}(\R^d)) $ is atomless there exists an open  Wasserstein ball
	$$ \mathbb{B}_{\mathcal{W}_2}(Q,\beta)=\{P\in \mathcal{P}_{2}(\R^d): \ \mathcal{W}_2(P,Q)< \beta\}$$
	with ${\bf P}(\mathbb{B}_{\mathcal{W}_2}(Q,\beta))\leq \epsilon/8$. Since ${\bf P}_n \xrightarrow{w} {\bf P} $ in $\mathcal{P}(\mathcal{P}_{2}(\R^d))$ and the closure of $\mathbb{B}_{\mathcal{W}_2}(Q,\beta/2)$ under the $\mathcal{W}_{2} $-metric,
	is contained in 
	$  \mathbb{B}_{\mathcal{W}_2}(Q,\beta) $,  there exists $n_0\in \N$  such that 
	$$  {\bf P}_n(\mathbb{B}_{\mathcal{W}_2}(Q,\beta/2)) \leq \epsilon/4 \quad \text{for all } n\geq n_0.$$
	As $\{{\bf P}_n\}_{n\in \N}\subset  \mathcal{P}(\mathcal{P}_{2}(\R^d)) $ is tight, there exists a compact set $K\subset \mathcal{P}_{2}(\R^d)$ such that 
	$$ {\bf P}_n(\mathcal{P}_{2}(\R^d)\setminus K) \leq \epsilon/4 \quad \text{for all } n\geq n_0. $$
	Call 
	$V= K\cap (\mathcal{P}_{2}(\R^d)\setminus \mathbb{B}_{\mathcal{W}_2}(Q,\beta/2))$ and $V^c= \mathcal{P}_{2}(\R^d)\setminus V$. Then 
	$$ {\bf P}( V^c)+ {\bf P}_n( V^c) \leq \epsilon \quad \text{for all } n\geq n_0. $$
	We call 
	$$ A:= \left| \left\| \int   \frac{I-T_{Q,P} }{\mathcal{W}_2(P,Q)} \mathrm{d}{\bf P}_n(P)\right\|_{L^2(Q)} - \left\| \int   \frac{I-T_{Q,P} }{\mathcal{W}_2(P,Q)} \mathrm{d}{\bf P}(P)\right\|_{L^2(Q)}  \right|.$$
	The triangle inequality yields 
	\begin{align*}
		A&\leq \left\| \int   \frac{I-T_{Q,P} }{\mathcal{W}_2(P,Q)} \mathrm{d}({\bf P}_n-{\bf P})(P)\right\|_{L^2(Q)}\\
		&\leq \left\| \int_{V}   \frac{I-T_{Q,P} }{\mathcal{W}_2(P,Q)} \mathrm{d}({\bf P}_n-{\bf P})(P)\right\|_{L^2(Q)}+ \left\| \int_{V^c}   \frac{I-T_{Q,P} }{\mathcal{W}_2(P,Q)} \mathrm{d}{\bf P} (P)\right\|_{L^2(Q)}\\
		&\qquad \qquad +\left\| \int_{V^c}   \frac{I-T_{Q,P} }{\mathcal{W}_2(P,Q)} \mathrm{d}{\bf P}_n (P)\right\|_{L^2(Q)}.
	\end{align*}
	Arguing as in Section~\ref{subsection:liesIn01} we get that, for $n \ge n_0$, 
	$$ \left\| \int_{V^c}   \frac{I-T_{Q,P} }{\mathcal{W}_2(P,Q)} \mathrm{d}{\bf P} (P)\right\|_{L^2(Q)} +\left\| \int_{V^c}   \frac{I-T_{Q,P} }{\mathcal{W}_2(P,Q)} \mathrm{d}{\bf P}_n (P)\right\|_{L^2(Q)}\leq {\bf P}( V^c)+ {\bf P}_n( V^c) \leq \epsilon. $$
	Moreover, as the function 
	$$ V \ni  P\mapsto  \frac{I-T_{Q,P} }{\mathcal{W}_2(P,Q)} \in  L^2(Q) $$ 
	is continuous and bounded (Lemma~\ref{Stabilty}),  for every $h\in L^2(Q)$ it holds that 
	$$ \int_{V}  \left\langle  \frac{I-T_{Q,P} }{\mathcal{W}_2(P,Q)} ,  h \right\rangle_{L^2(Q)}  \mathrm{d}({\bf P}_n-{\bf P})(P) \to 0 ,$$
	meaning that 
	$ \int_{V}    \frac{I-T_{Q,P} }{\mathcal{W}_2(P,Q)}  \mathrm{d}({\bf P}_n-{\bf P})(P)  $ converges to zero in the weak topology of $L^2(Q)$. However, as the set
	$$ \left\{ \frac{I-T_{Q,P} }{\mathcal{W}_2(P,Q)} : P\in V\right\} \cup \{ 0\} $$
	is compact (note that $V$ is compact in $\mathcal{P}_{2}(\R^d)$ and  $\mathcal{P}_{2}(\R^d)
	\backslash 
	\{ Q \}
	\ni  P\mapsto  \frac{I-T_{Q,P} }{\mathcal{W}_2(P,Q)}$ is continuous, see Lemma~\ref{Stabilty}), its closed convex hull, namely $C$, is compact as well.   Since $ \int_V  \frac{I-T_{Q,P} }{\mathcal{W}_2(P,Q)}  \mathrm{d} {\bf P}_n $ lies in $C$ for all $n\in \N$, the convergence of  $ \int_{V}    \frac{I-T_{Q,P} }{\mathcal{W}_2(P,Q)}  \mathrm{d}({\bf P}_n-{\bf P})(P)  $ towards zero holds in the strong topology of  $L^2(Q)$.  We have proven that 
	$A\leq 2 \epsilon$ for $n$ big enough. Since $\epsilon$ was arbitrarily chosen, the result follows. 
\end{proof}

\section{Proofs for Section \ref{consistency}}

\begin{proof}[Proof of Lemma \ref{lem:asympt:var:CLT}]
	First, by \eqref{eq:LLN:two} and the continuous mapping theorem,
	\[
	\frac{1}{
		\left(
		\rm WSD(Q; {\bf P}_n ) - 1
		\right)^2 
	}
	\xrightarrow{a.s.}
	\frac{1}{
		\left(
		\rm WSD(Q; {\bf P} ) - 1
		\right)^2},
	\]
	where the right-hand-side is deterministic.
	
	Next,  from \eqref{eq:LLN:one} and the continuous mapping theorem,
	\[
	\left\| 
	S_{{\bf P}_n,Q }
	\right\|_{L^2(Q)}^4
	\xrightarrow{a.s.}
	\left\| 
	S_{{\bf P},Q }
	\right\|_{L^2(Q)}^4.
	\]
	
	Finally,
	\begin{align*}
		\frac{1}{n}\sum_{i=1}^n
		\left( 
		\langle
		\xi_i
		,
		S_{{\bf P}_n,Q }
		\rangle_{L^2(Q)}
		\right)^2
		= &
		\frac{1}{n}\sum_{i=1}^n
		\left( 
		\langle
		\xi_i
		,
		S_{{\bf P},Q }
		\rangle_{L^2(Q)}
		+
		\langle
		\xi_i
		,
		S_{{\bf P}_n,Q }
		-
		S_{{\bf P},Q }
		\rangle_{L^2(Q)}
		\right)^2
		\\
		=  &
		\underbrace{\frac{2}{n}\sum_{i=1}^n 
			\langle
			\xi_i
			,
			S_{{\bf P},Q }
			\rangle_{L^2(Q)} \langle
			\xi_i
			,
			S_{{\bf P}_n,Q }-S_{{\bf P},Q }
			\rangle_{L^2(Q)}}_{=:A_n}\\
		& +\underbrace{\frac{1}{n}\sum_{i=1}^n
			\left( 
			\langle
			\xi_i
			,
			S_{{\bf P}_n,Q }
			-
			S_{{\bf P},Q }
			\rangle_{L^2(Q)}
			\right)^2}_{=:B_n}
		+\underbrace{\frac{1}{n}\sum_{i=1}^n
			\left( 
			\langle
			\xi_i
			,
			S_{{\bf P},Q }
			\rangle_{L^2(Q)}
			\right)^2}_{=:C_n}.
	\end{align*}
	We claim that $B_n \xrightarrow{a.s.} 0$ and that $C_n \xrightarrow{a.s.} \E
	\left[ 
	\left(
	\langle 
	\xi_1
	,
	S_{{\bf P},Q }
	\rangle_{L^2(Q)}
	\right)^2
	\right] \le 1$. From this claim and Cauchy-Schwartz inequality in $\R^n$, we will conclude that $A_n \xrightarrow{a.s.} 0$. From here, we obtain
	\[
	\widehat{\sigma}_{n,Q}^2
	\xrightarrow{a.s.}
	\frac{
		\E
		\left[ 
		\left(
		\langle 
		\xi_1
		,
		S_{{\bf P},Q }
		\rangle_{L^2(Q)}
		\right)^2
		\right]
		-
		\left\| 
		S_{{\bf P},Q }
		\right\|_{L^2(Q)}^4
	}{
		\left(
		\rm WSD(Q; {\bf P} ) - 1
		\right)^2
	}
	=
	\Var
	\left( 
	\frac{\langle \mathbb{G}_{{\bf P},Q}, S_{{\bf P},Q } \rangle_{L^2(Q)} }{ \rm WSD(Q; {\bf P} )-1}
	\right),
	\]
	since 
	$\xi_1$ has expectation $S_{{\bf P},Q }$ in $L^2(Q)$ and since
	$\mathbb{G}_{{\bf P},Q} \in L^2(Q)$ has the same covariance operator as $\xi_1$ (by the central limit theorem leading to Theorem \ref{Theorem:Glivenk}).
	
	Let us now prove the claim. For $B_n$ we have
	\[
	0 \le
	B_n 
	\le 
	\frac{1}{n}\sum_{i=1}^n
	\| 
	\xi_i
	\|_{L^2(Q)}^2
	\left\| 
	S_{{\bf P}_n,Q }
	-
	S_{{\bf P},Q }
	\right\|_{L^2(Q)}^2
	\le 
	\left\| 
	S_{{\bf P}_n,Q }
	-
	S_{{\bf P},Q }
	\right\|_{L^2(Q)}^2
	\xrightarrow{a.s.}0.
	\]
	The claim for $C_n$ follows directly from the strong law of large number in $\mathbb{R}$. 
	Hence, the lemma holds.
\end{proof}

\begin{proof}[Proof of Lemma~\ref{Lemma:ConvergenceBL}]
	Let  $S\subset \mathcal{P}_{p}(\R^d) $ be a closed set and define
	$$ {\rm BL}_1(S)=\left\{f:S\to \R : \ |f(P)|\leq 1\ {\rm and}\   |f(P)-f(Q) |\leq \mathcal{W}_{p}(P,Q), \ \forall \, P,Q\in S\right\}. $$
	Fix $f\in  {\rm BL}_1(\mathcal{P}_{p}(\R^d))$. Then 
	$$  \left|\int f(P)\mathrm{d}({\bf P}_{n,m}- {\bf P})(P)\right| \leq \underbrace{\left|\int f(P)\mathrm{d}({\bf P}_{n,m}- {\bf P}_{n})(P)\right|}_{A_{n,m}(f)}+ \underbrace{\left|\int f(P)\mathrm{d}( {\bf P}_{n}- {\bf P})(P)\right|}_{B_{n}(f)}, $$
	where ${\bf P}_{n}=\frac{1}{n}\sum_{i=1}^n \delta_{P_{i}}$.
	It can be proved by standard means that  $$\E\left[\sup_{ f\in  {\rm BL}_1(\mathcal{P}_{p}(\R^d))}B_n(f) \right] \to 0$$ as 
	$n\to \infty$. 
	Since $f\in {\rm BL}_1(\mathcal{P}_{p}(\R^d))$, it holds that  
	$$ A_{n,m}(f)=\left|  \frac{1}{n}\sum_{i=1}^n f(P_{i,m})-f(P_{i})  \right|\leq   \frac{1}{n}\sum_{i=1}^n \min(2,\mathcal{W}_p(P_{i,m},P_{i}))   $$
	which, by taking expectations, implies  
	$$\E\left[\sup_{ f\in  {\rm BL}_1(\mathcal{P}_{p}(\R^d))} A_{n,m}(f)\right] \leq  \frac{1}{n}\sum_{i=1}^n\E[\min(2,\mathcal{W}_p(P_{i,m},P_{i}))] .  $$
	Since the sequence 
	$\{\mathcal{W}_p(P_{i,m},P_{i})\}_{i=1}^n$ is exchangeable, it holds that 
	$$\E\left[\sup_{ f\in  {\rm BL}_1(\mathcal{P}_{p}(\R^d))} A_{n,m}(f)\right] \leq  \E[\min(2,\mathcal{W}_p(P_{1,m},P_{1}))] .  $$
	The latter tends to zero by Glivenko–Cantelli theorem and the fact that, conditionally to $P_1$, 
	$$\frac{1}{m} \sum_{j=1}^m  \X_{1,j}^p \xrightarrow{a.s.}  \int \|\x\|^p \mathrm{d}P_1(\x)  $$
	as $m\to \infty$.
\end{proof}

\begin{proof}[Proof of Lemma~\ref{lemma:rate:two:stage}]
	First, triangle inequality yields 
	$$ |  {\rm WSD}(Q; {\bf P}_{n,m})- {\rm WSD}(Q; {\bf P}_n)| \leq \| S_{{\bf P}_{n,m},Q}- S_{{\bf P}_{n},Q}\|_{L^ 2(Q)}. $$
	Next, since
	$$  S_{{\bf P}_{n,m},Q}= \frac{1}{n}  \sum_{i=1}^n \frac{I-T_{Q,P_{i,m}} }{\|I-T_{Q,P_{i,m}}\|_{L^2(Q)}}  $$
	and 
	$$  S_{{\bf P}_{n},Q}= \frac{1}{n}  \sum_{i=1}^n \frac{I-T_{Q,P_{i}}  }{\|I-T_{Q,P_{i}}\|_{L^2(Q)}}, $$
	we can bound 
	$$ \| S_{{\bf P}_{n,m},Q}- S_{{\bf P}_{n},Q}\|_{L^ 2(Q)} \leq \frac{1}{n}  \sum_{i=1}^n \left\| \frac{I-T_{Q,P_{i}}  }{\|I-T_{Q,P_{i}}\|_{L^2(Q)}} - \frac{I-T_{Q,P_{i,m}} }{\|I-T_{Q,P_{i,m}}\|_{L^2(Q)}}\right\|_{L^2(Q)}, $$
	and use Lemma~\ref{lem:diff:renormalized} to get 
	$$  \| S_{{\bf P}_{n,m},Q}- S_{{\bf P}_{n},Q}\|_{L^ 2(Q)} \leq \frac{2}{n}  \sum_{i=1}^n  \frac{\|T_{Q,P_{i,m}}-T_{Q,P_{i}}\|_{L^2(Q)}}{\|I-T_{Q,P_{i}}\|_{L^2(Q)}}.   $$
	Finally, we take expectations to get 
	\begin{align*}
		& \E 
		\left[ 
		\| S_{{\bf P}_{n,m},Q}- S_{{\bf P}_{n},Q}\|_{L^ 2(Q)}
		\right]
		\\
		&\leq \frac{2}{n}  \sum_{i=1}^n  \E\left[\frac{\|T_{Q,P_{i,m}}-T_{Q,P_{i}}\|_{L^2(Q)}}{\|I-T_{Q,P_{i}}\|_{L^2(Q)}} \right]\\
		&=\frac{2}{n}  \sum_{i=1}^n  \E\left[\E\left[\frac{\|T_{Q,P_{i,m}}-T_{Q,P_{i}}\|_{L^2(Q)}}{\|I-T_{Q,P_{i}}\|_{L^2(Q)}} \bigg\vert P_i \right] \right]\\
		&=\frac{2}{n}  \sum_{i=1}^n  \E\left[\E\left[ \|T_{Q,P_{i,m}}-T_{Q,P_{i}}\|_{L^2(Q)} \bigg\vert P_i \right] \frac{1}{\|I-T_{Q,P_{i}}\|_{L^2(Q)}} \right]\\
		\mbox{(using \eqref{Rates-OT}:)} 
		~ ~ ~
		&\leq 
		2
		\left( \alpha(d,m)\right)^{\frac{1}{2}} \E_{P_1\sim {\bf P}}\left[ W_2^{-1}(Q,P_{1})\right].
	\end{align*}
	This concludes the proof.
\end{proof}

\begin{Lemma}
	\label{lem:diff:renormalized}
	Let $x$, $y$ be two elements of a Hilbert space with  norm $\| \cdot \|$. Then we have
	\begin{equation} \label{eq:renorm:diff}
		\left\|  
		\frac{x}{\|x\|}
		-
		\frac{y}{\|y\|}
		\right\|
		\le 
		2 
		\frac{\|x-y\|}{\|x\|}.
	\end{equation}
\end{Lemma} 

\begin{proof}[Proof of Lemma \ref{lem:diff:renormalized}]
	If $\| x\| =0$ or $\| y\| =0$, it is simple to show that \eqref{eq:renorm:diff} holds, with the convention $\mathbf{0}/0 = \mathbf{0}$.
	Consider then $\| x\| > 0$ and $\| y\|  > 0$.  
	We have
	\begin{align*}
		\left\|  
		\frac{x}{\|x\|}
		-
		\frac{y}{\|y\|}
		\right\|
		&
		\le 
		\left\|  
		\frac{x}{\|x\|}
		-
		\frac{y}{\|x\|}
		\right\|
		+
		\left\|  
		\frac{y}{\|x\|}
		-
		\frac{y}{\|y\|}
		\right\|
		\\ & =
		\frac{\| x-y \|}{\|x\|}
		+
		\left\|  
		\frac{y \| y\|}{\|x\| \cdot \|y\| }
		-
		\frac{y \| x\|}{\|y\| \cdot \|x \|}
		\right\|
		\\ 
		& = 
		\frac{\| x-y \|}{\|x\|}
		+
		\frac{\left| \| y\| - \| x\| \right|}{\|x\|}
		\\ 
		& \le 
		2 \frac{\| x-y \|}{\|x\|},
	\end{align*}
	which concludes the proof.
\end{proof}

\section{Proof of Lemma~\ref{lemma:RelationDephts}}

Since $ \mathcal{W}_2(P,P')\leq \|T_{Q,P}-T_{Q,P'}\|_{L^2(Q)} $ then 
\begin{align*}
	\mathrm{MSD}(Q;{\bf P} )&= 1-\frac{1}{2}\E_{(P',P)\sim {\bf P} \otimes {\bf P} }\left[\frac{\mathcal{W}_2^2( P,Q)+\mathcal{W}_2^2( P',Q)-\mathcal{W}_2^2( P,P')}{\mathcal{W}_2( Q,P)\mathcal{W}_2( Q,P')} \right]\\
	&\leq 1-\underbrace{\frac{1}{2}\E_{(P',P)\sim {\bf P} \otimes {\bf P} }\left[\frac{\mathcal{W}_2^2( P,Q)+\mathcal{W}_2^2( P',Q)-\|T_{Q,P}-T_{Q,P'}\|_{L^2(Q)}^2  }{\mathcal{W}_2( Q,P)\mathcal{W}_2( Q,P')} \right]}_{=:A}
\end{align*}
with equality if and only if 
$ \mathcal{W}_2(P,P')=\|T_{Q,P}-T_{Q,P'}\|_{L^2(Q)} $ for ${\bf P} \otimes {\bf P}$-almost all
$P,P'$. Therefore, since 
\begin{align*}
	\langle T_{Q,P}-I, T_{Q,P'}-I\rangle_{L^2(Q)} &=  \| T_{Q,P}-I\|_{L^2(Q)}^2+ \langle T_{Q,P}-I,  T_{Q,P'}-T_{Q,P}\rangle_{L^2(Q)}
\end{align*}
and 
\begin{align*}
	\langle T_{Q,P}-I, T_{Q,P'}-I\rangle_{L^2(Q)} &=  \| T_{Q,P'}-I\|_{L^2(Q)}^2+ \langle T_{Q,P'}-I,  T_{Q,P}-T_{Q,P'}\rangle_{L^2(Q)}
\end{align*}
we get 
$$ 2\langle T_{Q,P}-I, T_{Q,P'}-I\rangle_{L^2(Q)} = \|T_{Q,P}-I\|_{L^2(Q)}^2+\|T_{Q,P'}-I \|_{L^2(Q)}^2-\|T_{Q,P}-T_{Q,P'}\|_{L^2(Q)}^2.  $$
As a consequence,
\begin{align*}
	A&= \frac{1}{2}\E_{(P',P)\sim {\bf P} \otimes {\bf P} }\left[\frac{\|T_{Q,P}-I\|_{L^2(Q)}^2+\|T_{Q,P'}-I \|_{L^2(Q)}^2-\|T_{Q,P}-T_{Q,P'}\|_{L^2(Q)}^2  }{\mathcal{W}_2( Q,P)\mathcal{W}_2( Q,P')} \right]\\
	&=\E_{(P',P)\sim {\bf P} \otimes {\bf P} }\left[\frac{\langle T_{Q,P}-I, T_{Q,P'}-I\rangle_{L^2(Q)}  }{\mathcal{W}_2( Q,P)\mathcal{W}_2( Q,P')} \right]\\
	&=(1- {\rm WSD}(Q;{\bf P} ))^2,
\end{align*}
and the result follows.

\section{Proofs for Section \ref{section:testing}}

\begin{proof}[Proof of Proposition \ref{prop:level:two:sample}]
	
	Let us first work conditionally to the independent copy
	$\widetilde{\mathbf{P}}_{n,m}$.  
	Under $H_0$, the elements of the vector 
	\begin{equation*}
		\mathcal{W} := (( W^P_i)_{i=1}^n ,  ( W^Q_k )_{k=1}^n) 
	\end{equation*}
	are i.i.d. and thus exchangeable. 
	Furthermore, there is a deterministic function $f : \mathbb{R}^{2n} \to [0,1]$ such that $    T_{\text{obs}} = f(  [\mathcal{W}_i]_{i=1}^{2n} )$. 	
	Write $\mathcal{S}_q$ for the set of all permutations of $\{1 , \dots , q\}$ for $q \in \mathbb{N}$.
	Then, for $b=1,\ldots,B$, there is a random permutation $\Pi_{b+1}$ uniformly distributed on $\mathcal{S}_{2n}$ such that  $T_{b} = f(  [\mathcal{W}_{\Pi_{b+1}(i)}]_{i=1}^{2n} )$. For convenience, for the rest of the proof,
	write $T_1,\ldots,T_{b}$ as $T_2,\ldots,T_{b+1}$ and write $T_{\text{obs}}$  as $T_1$. Let us show that $T_1,\ldots,T_{B+1}$ are exchangeable. For any measurable subsets $A_1 , \ldots, A_{B+1}$ of $[0,1]$, we have
	\begin{align*}
		& \mathbb{P} \left( 
		T_1 \in A_1 , \dots, T_{B+1} \in A_{B+1}
		\right) 
		\\
		=& 
		\mathbb{P} \left( 
		f(  [\mathcal{W}_i]_{i=1}^{2n} ) \in A_1
		, 
		f(  [\mathcal{W}_{\Pi_{2}(i)}]_{i=1}^{2n} ) \in A_2
		,
		\dots, 
		f(  [\mathcal{W}_{\Pi_{B+1}(i)}]_{i=1}^{2n} ) \in A_{B+1}
		\right) 
		\\ 
		=& 
		\frac{1}{ ((2n)!)^B }
		\sum_{\pi_2,\dots,\pi_{B+1} \in \mathcal{S}_{2n} }
		\\
		& 
		\mathbb{P} \left( 
		f(  [\mathcal{W}_i]_{i=1}^{2n} ) \in A_1
		, 
		f(  [\mathcal{W}_{\pi_{2}(i)}]_{i=1}^{2n} ) \in A_2
		,
		\dots, 
		f(  [\mathcal{W}_{\pi_{B+1}(i)}]_{i=1}^{2n} ) \in A_{B+1}
		\right).
	\end{align*}
	Above, since $\mathcal{W}_1, \ldots, \mathcal{W}_{2n}$ are exchangeable, the above is equal to
	\begin{align*}
		&	\frac{1}{ ((2n)!)^{B+1} }
		\sum_{\pi_1, \pi_2,\dots,\pi_{B+1} \in \mathcal{S}_{2n} }
		\\
		&	\mathbb{P} \left( 
		f(  [\mathcal{W}_{ \pi_1(i) }]_{i=1}^{2n} ) \in A_1
		, 
		f(  [\mathcal{W}_{\pi_{2} \circ \pi_1(i)}]_{i=1}^{2n} ) \in A_2
		,
		\dots, 
		f(  [\mathcal{W}_{\pi_{B+1} \circ \pi_1 (i)}]_{i=1}^{2n} ) \in A_{B+1}
		\right).
	\end{align*}
	For any fixed $\pi_1$,	we can apply a change of indices in the above summation over $\pi_2,\dots,\pi_{B+1}$, with the bijection $(\pi_2,\ldots,\pi_{B+1}) \mapsto (\pi_2 \circ \pi_1, \ldots, \pi_{B+1} \circ \pi_1)$. With that, the above display is equal to
	\begin{align*}
		&	\frac{1}{ ((2n)!)^{B+1} }
		\sum_{\pi_1, \pi_2,\dots,\pi_{B+1} \in \mathcal{S}_{2n} }
		\\
		&	\mathbb{P} \left( 
		f(  [\mathcal{W}_{ \pi_1(i) }]_{i=1}^{2n} ) \in A_1
		, 
		f(  [\mathcal{W}_{\pi_{2} (i)}]_{i=1}^{2n} ) \in A_2
		,
		\dots, 
		f(  [\mathcal{W}_{\pi_{B+1}  (i)}]_{i=1}^{2n} ) \in A_{B+1}
		\right).
	\end{align*}
	Hence, in distribution, $T_1,\ldots,T_{B+1}$ is composed of the same function $f$ applied to $B+1$ \textit{i.i.d.} uniform permutations of the same vector $\mathcal{W}$. Hence $T_1,\ldots,T_{B+1}$ are exchangeable.

	Now, let us show that the ranks of $ \widetilde{T}_1,\ldots, \widetilde{T}_{B+1}$ are uniformly distributed on $\mathcal{S}_{B+1}$, where $ \widetilde{T}_1,\ldots, \widetilde{T}_{B+1}$ are obtained from $ T_1,\ldots, T_{B+1}$ by the random mechanism that breaks ties (we also shift indices from $\{0,\dots,B\}$ to $\{1,\dots,B+1\}$ for convenience). For this, write the set of possible cardinalities among clusters of equal values,
	\[
	\mathcal{E}
	=
	\left\{ 
	(n_1,\dots,n_m) \in \mathbb{N}^m;
	m \in \mathbb{N}, n_1 + \dots + n_m = B+1 
	\right\}.
	\]
	For $(n_1,\dots,n_m) \in \mathcal{E}$, and for $t_1,\dots,t_{B+1} \in [0,1]$, we let $E_{n_1,\dots,n_m}(t_1,\dots,t_{B+1})$ be the event where $t_1 \le \dots \le t_{B+1}$ and  among the set $\{ t_1,\dots,t_{B+1} \}$ there are $m$ distinct values, with the smallest one reached by $m_1$ elements of $(t_1,\dots,t_{B+1})$, the second to smallest one reached by $m_2$ elements of $(t_1,\dots,t_{B+1})$, and so on.
	With this formalism, we have for any permutation $\sigma \in \mathcal{S}_{B+1}$,
	\begin{align*}
		&		\mathbb{P}
		\left(
		\widetilde{T}_{\sigma(1)}
		< \dots < 
		\widetilde{T}_{\sigma(B+1)}
		\right)
		\\
		&= 
		\sum_{(n_1,\dots,n_m) \in 	\mathcal{E}}
		\mathbb{P}
		\left(
		\left. 
		\widetilde{T}_{\sigma(1)}
		< \dots <
		\widetilde{T}_{\sigma(B+1)}
		\right| 
		E_{n_1,\dots,n_m}(T_{\sigma(1)} , \dots , T_{\sigma(B+1)})
		\right)
		\mathbb{P}
		\left(
		E_{n_1,\dots,n_m}(T_{\sigma(1)} , \dots , T_{\sigma(B+1)})
		\right)
		\\
		&= 
		\sum_{(n_1,\dots,n_m) \in 	\mathcal{E}}
		\frac{1}{
			(n_1!)
			\times \dots \times
			(n_m!)
		}
		\mathbb{P}
		\left(
		E_{n_1,\dots,n_m}(T_{\sigma(1)} , \dots , T_{\sigma(B+1)})
		\right)
		\\
		&= 
		\sum_{(n_1,\dots,n_m) \in 	\mathcal{E}}
		\frac{1}{
			(n_1!)
			\times \dots \times
			(n_m!)
		}
		\mathbb{P}
		\left(
		E_{n_1,\dots,n_m}(T_1 , \dots , T_{B+1})
		\right)
		~ ~ ~ ~ ~
		\text{($T_1,\dots,T_{B+1}$ are exchangeable)}
		\\
		& = \sum_{(n_1,\dots,n_m) \in 	\mathcal{E}}
		\mathbb{P}
		\left(
		\left. 
		\widetilde{T}_{1}
		< \dots < 
		\widetilde{T}_{B+1}
		\right| 
		E_{n_1,\dots,n_m}(T_{1} , \dots , T_{B+1})
		\right)
		\mathbb{P}
		\left(
		E_{n_1,\dots,n_m}(T_{1} , \dots , T_{B+1})
		\right)
		\\
		& = 
		\mathbb{P}
		\left(
		\widetilde{T}_{1}
		< \dots < 
		\widetilde{T}_{B+1}
		\right).
	\end{align*}
	Hence, indeed the ranks of $\widetilde{T}_1,\ldots,\widetilde{T}_{B+1}$ are uniformly distributed over $\mathcal{S}_{B+1}$. 	
	It follows that 
	\[
	p=	\frac{ 1+ \#\{b \in \{2,\dots,B+1\}: \widetilde{T}_b\geq \widetilde{T}_{1} \} }{ 1+B } 
	\]
	is uniformly distributed on $\{ \frac{1}{B+1} , \frac{2}{B+1}, \dots ,\frac{B}{B+1}, 1\}$. This is shown conditionally to the independent copy
	$\widetilde{\mathbf{P}}_{n,m}$ at this stage.  By the law of total expectation, $p$ is also uniformly distributed  on $\{ \frac{1}{B+1} , \frac{2}{B+1}, \dots ,\frac{B}{B+1}, 1\}$ unconditionally.
	
	Finally, it is simple to show that, as a consequence, $p \stackrel{w}{\longrightarrow} \rm{Unif}(0, 1)$. This concludes the proof.
\end{proof}

\begin{proof}[Proof of Proposition \ref{testing_power}]
	
	Let $F_{n,n}$ be the random empirical CDF of
	\[
	\frac{1}{n}
	\sum_{i=1}^n
	\delta_{ \rm{WSD}(P_{i}; \widetilde{\mathbf{P}}_{n}) }.
	\]
	Let $F_n$ be the random CDF of $\rm{WSD}(P; \widetilde{\mathbf{P}}_{n})$ for $P \sim \mathbf{P}$ and conditionally to $\widetilde{\mathbf{P}}_{n}$. By applying the DKW inequality conditionally to $\widetilde{\mathbf{P}}_{n}$ and then taking an expectation, we obtain
	\[
	\sup_{x \in [0,1]}
	\left| 
	F_{n,n}(x)
	-
	F_n(x)
	\right| 
	\stackrel{w}{\longrightarrow} 
	0
	\]
	as $n \to \infty$. 
	
	By Theorem \ref{Theorem:Glivenk}, for any fixed $Q \in \mathcal{P}_2^{a.c}(\R^d)$, we have $\rm{WSD}(Q; \widetilde{\mathbf{P}}_{n})
	\stackrel{a.s.}{\longrightarrow} 
	\rm{WSD}(Q; \mathbf{P})
	$. Hence, a.s., the distribution with CDF $F_n$ converges in distribution to the distribution with CDF $F$. Since $F$ is continuous, a.s., for every $x \in \mathbb{R}$,
	$F_n(x) \to F(x)$. Hence from the second theorem of Dini, we have
	\[
	\sup_{x \in [0,1]}
	\left| 
	F_{n}(x)
	-
	F(x)
	\right| 
	\stackrel{a.s.}{\longrightarrow} 
	0.
	\]
	Hence we have 
	\[
	\sup_{x \in [0,1]}
	\left| 
	F_{n,n}(x)
	-
	F(x)
	\right| 
	\stackrel{}{\longrightarrow} 
	0
	\]
	as $n \to \infty$ in probability. 
	
	Next, let $G_{n,n}$ be the random empirical CDF of
	\[
	\frac{1}{n}
	\sum_{i=1}^n
	\delta_{ \rm{WSD}(Q_{i}; \widetilde{\mathbf{P}}_{n}) }.
	\]
	With the same arguments as above, we have
	\[
	\sup_{x \in [0,1]}
	\left| 
	G_{n,n}(x)
	-
	G(x)
	\right| 
	\stackrel{}{\longrightarrow} 
	0
	\]
	as $n \to \infty$ in probability.
	Hence, by the triangle inequality,
	\[
	T_{\text{obs}}
	\stackrel{}{\longrightarrow} 
	\sup_{x \in [0,1]}
	\left| 
	F(x)
	-
	G(x)
	\right| 
	>0,
	\]
	as $n \to \infty$ in probability. 
	
	Next, write, as in the proof of Proposition \ref{prop:level:two:sample},
	\[
	\mathcal{W} := (( W^P_i)_{i=1}^n ,  ( W^Q_k )_{k=1}^n),
	\]
	consider a fixed $b$
	and write
	\[
	T_b = 
	\sup_{x \in [0,1]}
	\left| 
	F^b_{n,n}(x)
	-
	G^b_{n,n}(x)
	\right|, 
	\]
	where $F^b_{n,n}$ is the empirical CDF of 
	$
	\mathcal{W}_{\Pi_b(1)},
	\dots,
	\mathcal{W}_{\Pi_b(n)}
	$, where
	$G^b_{n,n}$ is the empirical CDF of 
	$
	\mathcal{W}_{\Pi_b(n+1)},
	\dots,
	\mathcal{W}_{\Pi_b(2n)}
	$
	and where $\Pi_b$ is a uniformly distributed permutation on $\{1 , \dots,2n\}$.
	
	Then, from \cite[Thm 3.7.2]{vanderVaart1996}, we have, a.s.,
	\[
	\sup_{x \in [0,1]}
	\left| 
	F^b_{n,n}(x)
	-
	\frac{1}{2}
	(F^b_{n,n}(x)
	+
	G^b_{n,n}(x))
	\right| 
	\to 0,
	\]
	and 
	\[
	\sup_{x \in [0,1]}
	\left| 
	G^b_{n,n}(x)
	-
	\frac{1}{2}
	(F^b_{n,n}(x)
	+
	G^b_{n,n}(x))
	\right| 
	\to 0,
	\]
	as $n \to \infty$.
	Hence, 
	a.s., $T_b \to 0$ as $n \to \infty$. It follows that, for any $\epsilon >0$,
	\[
	\mathbb{E}
	\left[ 
	\frac{1+\sum_{b=1}^{B}
		\mathbf{1} \left\{ 
		T_b \ge \epsilon
		\right\}
	}{B+1}
	\right]
	\le 
	\frac{1}{B+1}
	+
	\mathbb{E}
	\left[ 
	\mathbf{1} \left\{ 
	T_b \ge \epsilon \right\}
	\right]
	\to 0
	\]
	as $n,B \to \infty$. Hence $ p$ goes to $0$ as $n,B \to \infty$ in probability.
	
\end{proof}

\begin{funding}
Fran\c{c}ois Bachoc was supported by the Project GAP (ANR-21-CE40-0007)
of the French National Research Agency (ANR)
and by the Chair UQPhysAI of the Toulouse ANITI AI Cluster. 
\end{funding}

\end{document}